\documentclass[11pt, oneside]{amsart}
\usepackage{amsmath}
\usepackage{amssymb}
\usepackage{amsthm}
\usepackage{amscd}
\usepackage[all]{xy}

\numberwithin{equation}{section}

\newcounter{AbcT}

\newtheorem {Theorem}    {Theorem}[section]

\newtheorem* {Definition} {Definition} 

\newtheorem {Question}    [Theorem]{Question}

\newtheorem {Lemma}      [Theorem]    {Lemma}
\newtheorem {Corollary}  [Theorem]    {Corollary}
\newtheorem {Proposition}[Theorem]    {Proposition}
\newtheorem {Claim}      [Theorem]    {Claim}

\newtheorem {Observation}[Theorem]    {Observation}

\newcounter{DM@bibnum}

\newcommand {\F} {{\mathbb F}}

\newcommand{\la}{\langle}
\newcommand{\ra}{\rangle}

\def\Lie{{\rm Lie}}

\def\log{{\rm log\,}}
\def\deg{{\rm deg\,}}

\def\Ker{{\rm Ker\,}}
\def\Im{{\rm Im\,}}
\def\im{{\rm Im\,}}

\def\LT{{\rm LT\,}}
\def\Lie{{\rm Lie}}

\def\NSL_2{{\mathcal N SL_2}}

\def\eps{\varepsilon}

\def\lam{\lambda}            
                
\def\phi{\varphi}

\def\calU{{\mathcal U}}

\def\hbar{\bar h}               \def\Hgag{\overline{ H}} \def\Ggag{\overline{G}}  \def\Ugag{\overline{U}}

               \def\Kgag{\overline{ K}}

               \def\Ugag{\overline{U}}

\def\phat{\widehat p}

\def\dbF{{\mathbb F}}

\def\dbN{{\mathbb N}}

\def\dbQ{{\mathbb Q}}
\def\dbR{{\mathbb R}}

\def\dbZ{{\mathbb Z}}

\def\Fp{{\dbF_p}}

\newcommand{\lla}{\la\!\la}
\newcommand{\rra}{\ra\!\ra}
\def\skv{{\vskip .12cm}}

\begin{document}

\title{Groups of positive weighted deficiency and their applications}
\author{Mikhail Ershov}
\address{University of Virginia}
\thanks{The first author is supported by the NSF grant DMS-0901703.}
\email{ershov@virginia.edu}
\author{Andrei Jaikin-Zapirain}
\address{Departamento de Matem\'aticas Universidad Aut\'onoma de Madrid\\ and
Instituto de Ciencias Matem\'aticas  CSIC-UAM-UC3M-UCM}
\thanks{The second author is supported by   Spanish Ministry of Science
and Innovation, grant   MTM2008-06680}
\email{andrei.jaikin@uam.es}
\date{June 27th, 2011}
\subjclass[2000]{Primary 20F05, 20F50, Secondary 20E18, 20E07, 20F69, 17B50}
\keywords{Golod-Shafarevich, deficiency, Tarski monsters, just-infinite}

\begin{abstract}
In this paper we introduce the concept of weighted deficiency
for abstract and pro-$p$ groups and study groups of positive weighted deficiency
which generalize Golod-Shafarevich groups. In order to study weighted deficiency
we introduce weighted  versions of the notions of rank for groups and index for subgroups
and establish weighted analogues of several classical results in combinatorial group theory,
including the Schreier index formula.

Two main applications of groups of positive weighted deficiency
are given. First we construct infinite finitely generated residually finite $p$-torsion groups in which every
finitely generated subgroup is either finite or of finite index -- these groups
can be thought of as residually finite analogues of Tarski monsters. Second we
develop a new method for constructing just-infinite groups  (abstract or pro-$p$)
with prescribed properties; in particular, we show that graded group algebras
of just-infinite groups can have exponential growth. We also prove that
every group of positive weighted deficiency has a hereditarily just-infinite quotient.
This disproves a conjecture of Boston on the structure of quotients of certain
Galois groups and solves Problem~15.18 from Kourovka notebook.
\end{abstract}
\maketitle
\section{Introduction}
\subsection{What this paper is about}
Let $\Gamma$ be a finitely presented group. Recall that
the {\it deficiency} of $\Gamma$ denoted by $def(\Gamma)$
is the maximum value of the quantity $|X|-|R|$ where
$(X,R)$ runs over all presentations of $\Gamma$ by generators
and relations.
Note that $def(\Gamma)$ cannot equal $\infty$
as $|X|-|R|$ does not exceed $d(\Gamma)$, the number of generators of $\Gamma$.
One of the best known results involving deficiency is a theorem of Baumslag and Pride
asserting that a group of deficiency $\geq 2$ has a finite index subgroup which
homomorphically maps onto a non-abelian free group. While the conclusion of this theorem is very strong,
the class of groups of deficiency $\geq 2$ to which it applies is rather limited. In this paper
we shall consider groups of {\it positive weighted deficiency (PWD)}, a much broader
class of groups which nevertheless exhibit a number of ``largeness'' properties.

Groups of positive weighted deficiency generalize another class of groups, first introduced
by Golod and Shafarevich in 1964 as a tool for solving two outstanding problems: the
general Burnside problem in group theory and the class field tower problem in number theory. In the last two decades there has been a lot of interest in studying the general structure
of Golod-Shafarevich groups, in particular, establishing similarities between  Golod-Shafarevich groups
and free groups (see, e.g., \cite{Ze}, where it is proved that Golod-Shafarevich pro-$p$ groups always
contain non-abelian free pro-$p$ groups). A recent work in this area~\cite{EJ} dealt with what were called
{\it generalized Golod-Shafarevich groups} -- these are groups of positive weighted deficiency
in our terminology. The reason that a larger class of groups was considered in \cite{EJ}
was not an attempt to seek the most general results, but the fact that the class of
PWD groups is more natural than that of Golod-Shafarevich groups and is easier to work with.

The purpose of this paper is three-fold. First, we shall give a strong
motivation for the concept of weighted deficiency and thus provide
a new point of view on (generalized) Golod-Shafarevich groups.
Second, we obtain new largeness results about  PWD groups,
building on the techniques introduced in \cite{EJ}. Third,
we provide two important applications of PWD groups:
\begin{itemize}
\item[(i)] construction of residually finite analogues of Tarski monsters (see \S 1.4)
\item[(ii)] a new construction of hereditarily just-infinite abstract and pro-$p$ groups (see \S 1.5).
\end{itemize}

\subsection{Weighted deficiency} We shall define weighted deficiency in
two categories of groups -- finitely generated abstract groups
and countably based pro-$p$ groups, but here for simplicity we limit
our discussion to finitely generated groups in the pro-$p$ case as well.
First we introduce a class of functions on pro-$p$ groups with values in
the interval $[0,1)$ which we call {\it valuations} (see \S~2.2) --
informally speaking these are ``multiplicative valuations'' (in the usual sense)
with some additional properties. Now let $X$ be a finite set and $F=F(X)$
the free pro-$p$ group on $X$. Given any function $W_0:X\to (0,1)$, one
can extend it to a valuation $W$ on $F$ such that the value of $W$ on each element of $F$
is largest possible. Any function $W:F\to [0,1)$ obtained in such a way
will be called a {\it weight function on $F$ with respect to $X$} (see \S~2.4
for a formal definition).

If $W$ is a weight function on $F(X)$ with respect to $X$,
for any subset $S$ of $F(X)$ we define $W(S)=\sum_{s\in S}W(s)\in \dbR_{\geq 0}\cup\{\infty\}$.
Given a pro-$p$ group $G$, we define the {\it weighted deficiency
\footnote{Our notion of weighted deficiency is somewhat similar to the
notion of $p$-deficiency introduced in a recent paper of Schlage-Puchta~\cite{SP}.}
of $G$} denoted by $wdef(G)$ to be the supremum of the quantities $$W(X)-W(R)-1$$
where $(X,R)$ runs over all pro-$p$ presentations of $G$ and $W$ runs over weight
functions on $F(X)$ with respect to $X$. We note that a pro-$p$ group $G$
is Golod-Shafarevich if and only if it has a presentation $(X,R)$ such that
$W(X)-W(R)-1>0$ for some {\it uniform} weight function $W$ on $F(X)$
with respect to $X$ (see \S~2.4 for the definition).

If $\Gamma$ is an abstract group and $p$ is a prime, one could define the
{\it weighted $p$-deficiency of $\Gamma$} in the same way (using abstract
presentations instead of pro-$p$ presentations), but it will be more
convenient to define the weighted $p$-deficiency of $\Gamma$ as the
weighted deficiency of its pro-$p$ completion $\Gamma_{\phat}$ (the small
difference between these two definitions will be discussed in Section~6).
Finally, the {\it weighted deficiency of $\Gamma$} is the supremum of its
weighted $p$-deficiencies over all primes $p$.

Sometimes it is useful to think about weighted deficiency
in a slightly different way. If $(X,R)$ is a presentation of a pro-$p$ group $G$
and $F$ is the free pro-$p$ group on $X$, then any weight function on $F$ with respect
to $X$ induces certain valuation on $G$; conversely, we will show that any valuation
on $G$ comes from some weight function. Therefore, given a pro-$p$ group $G$
and a valuation $W$ on $G$, we can consider the {\it deficiency of $G$
with respect to $W$}, denoted by $def_W(G)$ -- it is defined in the same way as
$wdef(G)$, except that instead of all weight functions we only consider the ones
which induce $W$. Thus, $wdef(G)$ can also be defined as the supremum of the
set $\{def_W(G)\}$ where $W$ runs over all valuations on $G$.
\vskip .1cm

This definition does not reveal much about the structure of groups of positive weighted deficiency (PWD).
To begin with, it is not clear at all that PWD groups cannot be finite (or even trivial).
The fact that PWD groups are always infinite is a consequence of the {\it Golod-Shafarevich inequality}
first proved in a slightly weaker form in \cite{GS} (see \cite{Ko} for a full version).
It is also easy to see that there exist torsion
PWD groups. Combining these two results, Golod~\cite{Go} produced the first examples of infinite finitely generated torsion groups.
In this paper we shall give a different and arguably more elementary proof of the fact that PWD groups must be infinite (see \S~4.1).

Many important families of abstract and pro-$p$ groups have positive weighted deficiency.
First of all, it is clear that any group of deficiency $\geq 2$ has PWD. It takes a simple
computation to show that a group $G$ (abstract or pro-$p$) has  positive weighted deficiency
whenever $def(G)\geq 0$ and $d(G/[G,G]G^p)\geq 5$ where $d(\cdot)$ stands for the minimal
number of generators. This implies that the fundamental group of a hyperbolic $3$-manifold
always has a finite index subgroup with PWD (see Section~6 for more details). 

More generally, a pro-$p$ group $G$ has PWD whenever $d(G)>1$ and $r(G)<d(G)^2/4$,
where $r(\cdot)$ is the minimal number of relators (see, e.g., \cite[\S~2.3]{Er}). These inequalities 
hold for many Galois  groups of the from $Gal(K_{p,S}/K)$ where $K$ is a number field,
$S$ is a finite (possibly empty) set of primes of $K$ and $K_{p,S}$ is the maximal $p$-extension
of $K$ unramified outside of $S$. 
For instance, the group $Gal(K_{p,S}/K)$ has PWD
if $K=\dbQ$ and $S$ contains at least $5$ primes congruent to $1$ mod $p$ or
if $\rho_p(K)> 2+2\sqrt{\nu(K)+1}$, where $\rho_p(K)$ is the rank of the $p$-component
of the ideal class group of $K$ and $\nu(K)$ is the rank of the unit group of $K$ --
this follows from \cite[Theorems~1,5]{Sha} (see also \cite[Theorems~11.5, 11.8]{Ko}).

In addition, the class of groups of positive weighted deficiency is closed
under two natural operations:
\begin{itemize}
\item[(a)] If $G$ is a pro-$p$ PWD group, then any open subgroup of $G$ also has PWD.
\item[(b)] If $G$ is PWD (abstract or pro-$p$), then $G\times H$ has PWD for any other group $H$.
\end{itemize}
We note that the smaller class of Golod-Shafarevich groups contains the families 
of hyperbolic $3$-manifold groups and Galois groups mentioned above,
but is not closed under operations (a) and (b).

\subsection{Quotients of PWD groups} Unlike groups of deficiency $\geq 2$, PWD groups need not be finitely presented
and do not necessarily satisfy the conclusion of the Baumslag-Pride theorem, but they do have
various largeness properties, as was already evident from the original work of Golod and Shafarevich.
A common type of statement about PWD groups is the existence of an infinite quotient with a prescribed property.
Here are two typical results of this kind:
\begin{itemize}
\item[(1)] Every abstract PWD group has an infinite torsion quotient which still has PWD \cite{Wil1}
\item[(2)] Every abstract PWD group has an infinite quotient with Kazhdan's property $(T)$ \cite{EJ}
\end{itemize}

One of the key results in this paper provides a natural addition to this list. It shows that given an abstract
PWD group $\Gamma$ and a finitely generated subgroup $\Lambda$ of $\Gamma$, one has a surprising amount of control
on the image of  $\Lambda$ in a suitable quotient of $\Gamma$:

\begin{Theorem}
\label{subgroup_main}
Let $\Gamma$ be an abstract PWD group and $\Lambda$ a finitely generated subgroup of $\Gamma$.
Then $\Gamma$ has an infinite quotient $\Gamma'$ such that
\begin{itemize}
\item[(i)] the image of $\Lambda$ in $\Gamma'$ is either finite or of finite index;
\item[(ii)] some finite index subgroup of $\Gamma'$ has PWD.
\end{itemize}
\end{Theorem}
 An intriguing part of the proof of Theorem~\ref{subgroup_main} is a very
simple way to decide which of the two possibilities in (i) holds.
Let $G=\Gamma_{\phat}$ be the pro-$p$ completion of $\Gamma$, let $L$ be the closure of (the image of) $\Lambda$ in $G$,
and let $W$ be a valuation on $G$ such that $def_W(G)>0$. Then the quantity we
need to look at is the {\it $W$-index} of $L$ in $G$,
denoted by $[G:L]_W$ (see \S 3.3 for the definition). The $W$-index
$[G:L]_W$ may take values in $\dbR_{\geq 1}\cup\{\infty\}$
and is finite whenever the usual index $[G:L]$ is finite (note that $[G:L]\leq [\Gamma:\Lambda]$).
The converse, however, is very far from being true -- for instance, if $\Lambda$ is normal in $\Gamma$
and the quotient $\Gamma/\Lambda$ is of subexponential growth, then $[G:L]_W$
is finite (for any $W$).

Going back to Theorem~\ref{subgroup_main}, we will show that if  $[G:L]_W<\infty$,
the image of $\Lambda$ in $\Gamma'$ can be made of finite index, and if $[G:L]_W=\infty$,
the image of $\Lambda$ can be made finite. In addition to playing a key role in the
proof of Theorem~\ref{subgroup_main}, $W$-index has many similarities with the usual index,
of which probably the most interesting one is the weighted version of the Schreier index formula
(see Theorem~\ref{ineq}).
We believe that this notion of $W$-index deserves further study.

Let us mention another new result about quotients of PWD groups.
It will be used in the proof of one of our results about just-infinite groups (see  Theorem~\ref{justinfexp} below), but is probably of independent interest:

\begin{Theorem}
\label{PWD_LERF} Every abstract PWD group has a quotient with LERF which
still has PWD.
\end{Theorem}

Property LERF plays a prominent role in many recent works in geometric group theory
and 3-manifold topology, and there are very few groups for which LERF has been
established. While Theorem~\ref{PWD_LERF} does not prove LERF for any  ``familiar'' groups,
it provides what we believe is one of the most general constructions of groups with LERF.

Using Theorem~\ref{PWD_LERF} we will also obtain a slight improvement of the main result of \cite{EJ}:
\begin{Theorem}
\label{Kazhdan_resfinite} Every abstract PWD group has an infinite residually finite quotient
with property $(T)$.
\end{Theorem}
\vskip .12cm

We now turn to the discussion of applications of PWD groups discovered in this paper.

\subsection{Constructing residually finite monsters}
In 1980 Ol'shanskii~\cite{Ol} proved the following remarkable theorem:

\begin{Theorem}[Ol'shanskii]
\label{thm:Olsh}
For every sufficiently large
prime $p$ there exists an  infinite group $\Gamma$
in which every proper subgroup is cyclic of order $p$.
\end{Theorem}

Groups described in this theorem along with their torsion-free counterparts
were the first examples of Tarski monsters,  and many other groups
with extremely unusual finiteness properties have been constructed since
then. Most of the  ``monster'' constructions produce groups which are neither
finitely presented nor residually finite, and there is significant interest
in building monsters satisfying one of these properties (or showing
that this is impossible). A major progress in this direction was obtained a few years
ago by Ol'shanskii and Sapir~\cite{OS} who constructed non-amenable finitely presented
torsion-by-cyclic groups.

If $\Gamma$ is a residually finite group, it clearly cannot satisfy the conclusion
of Theorem~\ref{thm:Olsh}, and the best approximation one can hope for
is that $\Gamma$ is a globally zero-one group as defined below.

\begin{Definition}\rm Let $\Gamma$ be a finitely generated abstract group
which is not virtually cyclic. We will say that
\begin{itemize}
\item[(a)] $\Gamma$ is a {\it globally zero-one group} if every subgroup of $\Gamma$
is either finite or of finite index.
\item[(b)] $\Gamma$ is a {\it locally zero-one group}
if every finitely generated subgroup
of $\Gamma$ is either finite or of finite index.
\end{itemize}
\end{Definition}

The existence of residually finite globally zero-one groups remains a major open problem
(see \S~8.1 for a brief discussion).
In this paper we construct the first examples of residually finite locally zero-one groups:

\begin{Theorem}
\label{zeroone}
Every abstract group of positive weighted deficiency has a residually finite quotient which is
locally zero-one.
\end{Theorem}

Theorem~\ref{zeroone} will be proved using repeated applications of Theorem~\ref{subgroup_main}
(in fact, a slightly stronger version of it).

We note that other examples of residually finite torsion groups with ``unusual'' properties were
constructed by Ol'shanskii and Osin~\cite{OO}.

\subsection{Applications to just-infinite groups}
An abstract (resp. profinite) group is called {\it just-infinite}
if it is infinite and all its proper quotients \footnote{By a quotient of a profinite group we will always mean a quotient by a {\it closed} normal subgroup}
are finite.
Being just-infinite is as close to being simple as an infinite 
residually finite group can be -- this is one of the reasons why
just-infinite groups are an interesting class to study.
A classification theorem due to Wilson~\cite{Wil3} implies that any abstract just-infinite group
$\Gamma$ satisfies one of the following mutually exclusive conditions:
\begin{itemize}
\item[(i)] $\Gamma$ is virtually simple
\item[(ii)] $\Gamma$ is {\it hereditarily just-infinite}, which means that
$\Gamma$ is residually finite and every finite index subgroup of $\Gamma$ is also just-infinite
\item[(iii)] $\Gamma$ has a finite index subgroup isomorphic to a direct power $\Lambda^n$
where $\Lambda$ is another just-infinite group and $n\geq 2$
\end{itemize}
We note that the most interesting groups in class (iii) are branch groups.
There is an analogous statement in the profinite case (see \cite{Gr}) except that
this time there are only two ``possibilities'', as infinite profinite groups  cannot be simple.

An easy application of Zorn lemma shows that any infinite finitely generated abstract or virtually
pro-$p$ group has a just-infinite quotient \cite[Proposition~3]{Gr}, but it seems that not much was known regarding what type of just-infinite groups one may obtain as quotients of a given class of groups. We are able
to address this issue in this paper. Let $\Gamma$ be an abstract PWD group and $\Omega$
its residually finite locally zero-one quotient from Theorem~\ref{zeroone}. Note that
any infinite quotient of $\Omega$ must also be a locally zero-one group; moreover, we will
show that any infinite quotient of $\Omega$ has infinite profinite completion.
It follows that just-infinite quotients of $\Omega$ cannot be of type (i) or (iii),
so any just-infinite quotient of $\Omega$ must be hereditarily just-infinite.
Thus Theorem~\ref{zeroone} has the following consequence:

\begin{Corollary}
\label{HJI}
Every abstract group of positive weighted deficiency has a hereditarily just-infinite torsion quotient.
\end{Corollary}

It was not even known whether finitely generated hereditarily just-infinite torsion groups exist
(this question was posed as Problem~15.18 in Kourovka notebook~\cite{Kou}). We note that
examples of just-infinite torsion groups of other types were known -- e.g., the groups in \cite{Ol}
are simple torsion groups, while the first Grigorchuk's group~\cite{Gr1} is a just-infinite branch
torsion group.
\vskip .1cm

Our next result is a pro-$p$ analogue of Corollary~\ref{HJI} which will require more work:
\begin{Theorem}
\label{HJI_prop}
Every pro-$p$ group of positive weighted deficiency has a hereditarily just-infinite quotient.
\end{Theorem}

Theorem~\ref{HJI_prop} disproves the following number-theoretic conjecture of Boston.
Let $p$ be a fixed prime, $K$ a number field, $S$ a finite set of primes of $K$
none of which lies above $p$ and $K_{p,S}$ the maximal $p$-extension
of $K$ unramified outside of $S$. Boston conjectured that whenever
the Galois group $Gal(K_{p,S}/K)$ is infinite, all of its just-infinite quotients
must be branch groups (see, e.g. \cite{Bo}). As we already mentioned, many
groups of the form  $Gal(K_{p,S}/K)$ have positive weighted deficiency
and thus by Theorem~\ref{HJI_prop} Boston's conjecture is false for each of those groups.
\vskip .2cm

We also obtain several results showing that hereditarily just-infinite groups may be rather large.
For instance, we will prove the following:

\begin{Theorem}
\label{justinfexp}
For every prime $p$ there exists a hereditarily just-infinite abstract group $\Delta$ such that
if $\omega$ is the augmentation ideal of the group algebra $\Fp[\Delta]$, then
the graded algebra $gr \Fp[\Delta]=\bigoplus\limits_{n\geq 0} \omega^n/\omega^{n+1}$ has exponential growth.
\end{Theorem}

The corresponding result for pro-$p$ groups also holds, and this time we will
establish the pro-$p$ version first and then deduce the result about abstract groups
(Theorem~\ref{justinfexp}).

We note that another result showing that hereditarily just-infinite profinite groups can be large
(in a very different sense) was recently obtained by Wilson~\cite[Theorem~A]{Wil4}.
\vskip .12cm

We finish with a peculiar application involving Kazhdan's property $(T)$.
In \cite{Er} it was shown that groups of positive weighted deficiency can have property~$(T)$.
Hence Theorem~\ref{zeroone} implies the existence of residually finite locally zero-one groups with property~$(T)$.
Note that an infinitely generated subgroup of a locally zero-one group must be the union
of an ascending chain of finite groups and hence amenable. Thus, we deduce the following:

\begin{Corollary}  There exist infinite finitely generated residually finite groups in which all
finitely generated subgroups have $(T)$ and all infinitely generated subgroups
are amenable.
\end{Corollary}

\subsection{Organization} In Section 2 we define weight functions
and valuations on pro-$p$ groups and introduce some standard tools that will
be used in this paper. In Section~3 we establish basic properties of
weight functions and valuations, including the weighted Schreier formula.
In Section~4 we discuss weighted deficiency of pro-$p$ groups; the case
of abstract groups is deferred until Section~6. Section~5 contains the
proof of the pro-$p$ version of Theorem~\ref{subgroup_main} which is a key
step in the proof of Theorem~\ref{zeroone}. In Section~7 we introduce property
LERF and some of its variations and prove Theorem~\ref{PWD_LERF} and
Theorem~\ref{Kazhdan_resfinite}. In Section~8 we establish the main applications of PWD groups;
in particular, we prove Theorem~\ref{zeroone}, Corollary~\ref{HJI} and Theorem~\ref{HJI_prop}.
Theorem~\ref{subgroup_main} will be established in the course of the proof of Theorem~\ref{zeroone}.
Section~9 is devoted to the proof of Theorem~\ref{justinfexp}.

\subsection{Acknowledgments} We would like to thank Dan Segal for
suggesting an idea that helped initiate this project, Efim Zelmanov
for proposing Theorem~\ref{HJI_prop} (as a conjecture) and
Yiftach Barnea for useful comments on an earlier version of this paper. 
We are also very grateful to the anonymous referee for pointing out several
inaccuracies and many suggestions which helped clarify the exposition.

\subsection{Glossary of notation}
$\empty$
\begin{itemize}
\item[$-$] $F(X)$. The free pro-$p$ group on a set $X$ (where $p$ is a prime fixed in advance).
\item[$-$] $\Gamma_{\phat}$. The pro-$p$ completion of an abstract group $\Gamma$.
\item[$-$] $G_{\alpha}$. The subgroup $\{g\in G: W(g)\leq \alpha\}$ of a pro-$p$ group $G$, where
$W$ is a valuation on $G$  fixed in advance; defined in  \S~2.2.
\item[$-$] $G_{<\alpha}$. The subgroup $\{g\in G: W(g)< \alpha\}$ of a pro-$p$ group $G$, where
$W$ is a valuation on $G$  fixed in advance; defined in  \S~2.2.
\item[$-$] $gr_w(B)$. The graded associative algebra of an associative algebra $B$ with respect to a valuation $w$ on $B$; defined in  \S~2.3.  
\item[$-$] $L_W(G)$. The graded restricted Lie algebra of a pro-$p$ group $G$ with respect to a valuation $W$ on $G$; defined in  \S~2.3.
\item[$-$] $\LT(g)$. The coset $g G_{< W(g)}$ (called the leading term of $g$), where $g\neq 1$ is an element of 
a pro-$p$ group $G$ and $W$ is a valuation on $G$  fixed in advance; defined in  \S~2.3.
\item[$-$] $rk_W(G)$. The $W$-rank of a pro-$p$ group $G$, where $W$ is a valuation on $G$; defined in \S~3.2.
\item[$-$] $[G:H]_W$. The $W$-index of a closed subgroup $H$ of a pro-$p$ group $G$, where $W$ is a valuation on $G$; 
defined in \S~3.3.
\item[$-$] $def_W(X,R)$. The deficiency of a pro-$p$ presentation $(X,R)$ with respect to $W$, where
$W$ is a finite weight function on $F(X)$ with respect to $X$; defined in \S~4.1.
\item[$-$] $def_W(F,N)$. The deficiency of a pair $(F,N)$ with respect to $W$, where $F$ is a free pro-$p$
group, $N$ is closed normal subgroup of $F$ and $W$ is a finite weight function on $F$; defined in \S~4.1.
\item[$-$] $wdef(G)$. The weighted deficiency of a pro-$p$ group $G$; defined in \S~4.1.
\item[$-$] $def_W(G)$. The deficiency of a pro-$p$ group $G$ with respect to a finite valuation $W$ on $G$; defined in \S~4.2.
\item[$-$] $def_W^G(U)$. The $G$-invariant deficiency of a pro-$p$ group $U$ with respect to a finite $G$-invariant valuation
$W$ on $U$, where $G$ is a profinite group containing $U$ as an open normal subgroup; defined in \S~4.4. 
\end{itemize}

\section{Preliminaries}

In this section we introduce some of the key notions that will be used in this paper,
including valuations on pro-$p$ groups and profinite $\Fp$-algebras and
weight functions on free pro-$p$ groups. We also recall several (mostly standard)
results about graded associative and restricted Lie algebras and
fix basic terminology (some of which is not standard).

{\bf Some conventions.}
In this paper the word `countable' will mean finite or countably infinite.
All abstract groups that we consider will be assumed countable and all pro-$p$ groups  will be assumed countably based (hence also countably generated). We are mostly interested in finitely generated
groups (abstract or pro-$p$), but the larger classes of countable abstract groups (resp. countably based
pro-$p$ groups) are more convenient to work with since they are closed not only under taking
homomorphic images but also under taking subgroups (closed subgroups in the pro-$p$ case).

\subsection{Algebras of noncommutative powers series and free pro-$p$ groups}

In this short subsection we define free pro-$p$ groups and relate them to
algebras of noncommutative powers series. For more details and proofs of
the results referred to as `well known' the reader may consult \cite{RZ} or \cite{Wil2}.

If $X$ is a countable set, by $F(X)$ we shall denote the free pro-$p$ group on $X$
(where a prime $p$ is fixed in advance).
If $X$ is finite, $F(X)$ can be defined
as the pro-$p$ completion of $F_{abs}(X)$, the free abstract group on $X$.

If $X=\{x_1,x_2,\ldots\}$ is countably infinite, we consider the topology
on $F_{abs}(X)$ where a base of neighborhoods of identity is given by those normal subgroups of $p$-power index which contain all but finitely many elements of $X$.
The free pro-$p$ group $F(X)$ is defined as the completion of $F_{abs}(X)$ in this topology.
It is not hard to see~\cite[Corollary~3.3.10]{RZ} that $F(X)$ is isomorphic to the inverse limit of the
free pro-$p$ groups on $\{x_1,\ldots, x_n\}$ as $n\to\infty$. Note that $F(X)$ is much smaller than the pro-$p$ completion of $F_{abs}(X)$, the latter not even being countably based.

It is well known~\cite[Corollary~7.7.5]{RZ} that if $X$ is any countable set, then any closed
subgroup of the free pro-$p$ group $F(X)$ is isomorphic to $F(Y)$ for a countable set $Y$.

\vskip .1cm
It is very convenient to study free pro-$p$ groups using their embedding into
(multiplicative groups of) algebras of power series over $\Fp$.
If  $Q=\{q_1,\ldots, q_n\}$ is a finite set, then $\Fp\lla Q\rra=\Fp\lla q_1,\ldots, q_n\rra$ denotes the $\F_p$-algebra of noncommutative power series in $Q$ over $\F_p$. If  $Q=\{q_1, q_2\ldots, \}$ is countably infinite, then $\F_p\lla Q\rra$ can be defined as the inverse limit of $\{\F_p\lla q_1,\ldots, q_n \rra\}$.

If $Q$ is a countable set and $A=\F_p\lla Q\rra$, then the closed subgroup of $A^*$ generated by $X=\{1+q:\ q\in Q\}$ is 
a free pro-$p$ group on $X$. Moreover, the natural map from the completed group algebra
$\F_p[[F(X)]]$ to $A$ that extends the inclusion map $X\to A$ is an isomorphism
(see \cite[\S~II.3]{Laz} or \cite[Theorem~7.3.3]{Wil2}).

\subsection{Valuations} In this subsection we define the concept of a valuation on
profinite $\Fp$-algebras and pro-$p$ groups. We warn the reader that our definition
of a valuation on a pro-$p$ group is more restrictive than the corresponding notion used 
in many other works, e.g. \cite{Laz}.

\begin{Definition}\rm Let $B$ be an associative profinite $\Fp$-algebra with $1$.
A continuous  function $w:B\to [0,1]$ is called a
\emph{valuation} if
\begin{itemize}
\item[(i)] $w(f)=0$ if and only if $f=0$;
\item[(ii)] $w(1)=1$;
\item[(iii)] $w(f+g)\leq \max\{w(f),w(g)\}$ for any $f,g\in B$;
\item[(iv)] $w(fg)\leq w(f)w(g)$ for any $f,g\in B$.
\end{itemize}
\end{Definition}
\begin{Remark} Valuations in our sense are sometimes called {\it multiplicative valuations}.
The usual (additive) valuations are logarithms of multiplicative valuations.
\end{Remark}

For each $\alpha\in (0,1]$ we set
$$B_{\alpha}=\{b\in B : w(b)\leq \alpha\}\mbox{ and }B_{<\alpha}=\{b\in B : w(b)<\alpha\}.$$
Since $[0, \alpha)$ is open in $[0,1)$ and $w$ is continuous, $B_{<\alpha}$
is an open ideal of $B$ (hence has finite index in $B$).
Thus, each valuation automatically satisfies an additional condition:
\begin{itemize}
\item[(v)] For each $\alpha>0$ the set $\Im w\cap (\alpha, 1]$ is finite,
where $\Im w$ is the image of $w$.

\end{itemize}
It also follows that $B_{<\alpha}=B_{\beta}$ for some $\beta$.

\vskip .1cm Next we define valuations on pro-$p$ groups.

\begin{Definition}\rm Let $G$ be a pro-$p$ group.
A continuous  function $W:G\to [0,1)$ is called a \emph{valuation} if
\begin{itemize}
\item[(i)] $W(g)=0$ if and only if $g=1$;
\item[(ii)] $W(fg)\leq \max\{W(f),W(g)\}$ for any $f,g\in G$;
\item[(iii)] $W([f,g])\leq W(f) W(g)$ for any $f,g\in G$, where as usual $[f,g]=f^{-1}g^{-1}fg$.
\item [(iv)] $W(g^{p})\leq W(g)^p$ for any $g\in G$.
\end{itemize}
\end{Definition}
\noindent
Each valuation $W$ on a pro-$p$ group $G$ satisfies
the following additional conditions:

\begin{itemize}
\item[(v)] For each $\alpha>0$ the set $\Im W\cap (\alpha, 1]$ is finite;
\item[(vi)] $W(g^{-1})=W(g)$ for each $g\in G$.
\end{itemize}

Condition (v) is established as in the case of algebras, and (vi)
follows from (ii), continuity of $W$ and the fact that $g^{-1}$
lies in the closure of the set $\{g^n :n\in\dbN\}$.
\vskip .2cm

The basic connection between valuations on pro-$p$ groups and profinite $\Fp$-algebras
is very simple. 

\begin{Observation} 
\label{obser_easy}
Let $B$ be a profinite $\Fp$-algebra, let $w$ be a valuation on $B$,
and let $M=B_{<1}=\{b\in B: w(b)<1\}$. Let $G$ be a closed subgroup of 
$1+M=\{1+m: m\in M\}$, and define the function $W:G\to [0,1)$ by $W(g)=w(g-1)$.
Then $W$ is a valuation on $G$.
\end{Observation}
\begin{proof} Continuity of $W$ and condition (i) in the definition of a group valuation are clear.
To prove (ii), note that $fg-1=(f-1)+f(g-1)$. Since $w$ is a valuation on $B$, we have
\begin{multline*}
W(fg)=w(fg-1)\leq \max\{w(f-1), w(f)w(g-1)\}\\
\leq \max\{w(f-1),w(g-1)\}=\max\{W(f),W(g)\}.
\end{multline*}
Similarly, (iii) follows from the identity $[f,g]-1=f^{-1}g^{-1}((f-1)(g-1)-(g-1)(f-1))$
and (iv) from the identity $f^p-1=(f-1)^p$.
\end{proof}

Again let $G$ be a pro-$p$ group and $W$ a valuation on $G$.
For each $\alpha\in (0,1]$ we set
\begin{equation}
\label{galpha}
G_{\alpha}=\{g\in G : W(g)\leq \alpha\}\mbox{ and }G_{<\alpha}=\{g\in G : W(g)<\alpha\}.
\end{equation}
Similarly to the algebra case, each of the subgroups $G_{\alpha}$ and $G_{<\alpha}$ is open
and normal in $G$ and $G_{<\alpha}=G_{\beta}$ for some $\beta$. Since clearly $\cap G_{\alpha}=\{1\}$,
we deduce that $\{G_{\alpha}\}$ is a base of neighborhoods of 1 in $G$, so in particular $G$ must be countably based.

\vskip .1cm
\noindent
{\bf Induced valuations.} Let $W$ be a valuation on a pro-$p$ group $G$.
\begin{itemize}
\item[(a)] If $H$ is a closed subgroup of $G$, then the restriction of $W$ to $H$
(also denoted by $W$) is a valuation of $H$. In the notations \eqref{galpha} we have
$H_\alpha=H\cap G_\alpha$ for each $\alpha\in (0,1)$.
\item[(b)] If $\phi:G\to G'$ is an epimorphism of pro-$p$ groups, then
$W$ induces a valuation $W'$ on $G'$ given by
$$W'(\phi(x))=\min\{W(y):\ \phi(y)=\phi(x)\}.$$
In the notations \eqref{galpha} we have $\phi(G_{\alpha})=G'_{\alpha}$
for each $\alpha\in (0,1)$.
\item[(c)] If $N$ is a closed normal subgroup of $G$ and $\phi: G\to G/N$
the natural projection, the induced valuation on $G/N$
will also be denoted by $W$, that is,
$$W(xN)=\min\{W(y):\ y\in xN\}.$$
\end{itemize}
{\bf Weight of a subset.} If $S$ is a countable subset of $G$, we define
$$W(S)=\sum_{s\in S}W(s).$$

Finally, we introduce the notion of a pseudo-valuation that will be occasionally used.
\begin{Definition}\rm Let $G$ be a pro-$p$ group.
A continuous  function $W:G\to [0,1)$ is called a \emph{pseudo-valuation} if
it satisfies conditions (ii), (iii) and (iv) in the definition of a valuation
(but not necessarily condition (i)).
\end{Definition}
Thus, valuations are precisely pseudo-valuations which take positive values
on non-identity elements.

\subsection{Graded algebras associated to valuations}

It is common in combinatorial group theory to consider Lie, restricted Lie and associative algebras
graded by natural numbers $\dbN$ (see, e.g, \cite[\S~11,12]{DDMS}). We will naturally encounter
gradings by more general abelian semigroups.

\begin{Definition}\rm Let $\Omega$ be an abelian semigroup (written multiplicatively) and
$F$ a field.

\noindent (a) Let $A$ be an associative $F$-algebra. A collection of $F$-subspaces
$\{A_{\alpha}\}_{\alpha\in \Omega}$ of $A$
will be called an {\it $\Omega$-grading} of $A$ if
\begin{itemize}
\item[(i)] $A=\oplus_{\alpha\in \Omega} A_{\alpha}$
\item[(ii)] $A_{\alpha} A_{\beta} \subseteq A_{\alpha\beta}$ for each $\alpha,\beta\in \Omega$
\end{itemize}

\noindent (b) Assume that ${\rm char} F=p$ and let $L$ be a restricted Lie $F$-algebra.
A collection of $F$-subspaces
$\{L_{\alpha}\}_{\alpha\in \Omega}$ of $L$
will be called an {\it $\Omega$-grading} of $L$ if
\begin{itemize}
\item[(i)] $L=\oplus_{\alpha\in \Omega} L_{\alpha}$
\item[(ii)] $[L_{\alpha}, L_{\beta}] \subseteq L_{\alpha\beta}$ for each $\alpha,\beta\in \Omega$
\item[(iii)] $L_{\alpha}^p \subseteq L_{\alpha^p}$ for each $\alpha\in \Omega$
\end{itemize}
\end{Definition}

\begin{Remark} If $L$ is a restricted Lie algebra graded by $\Omega$, its (restricted) universal
enveloping algebra $\calU(L)$ admits a natural grading by $\Omega \sqcup\{1\}$
where $1$ is the identity element adjoined to $\Omega$.
\end{Remark}

We are now ready to define graded associative (resp. restricted Lie) algebras
corresponding to valuations on profinite $\Fp$-algebras (resp. pro-$p$ groups).
These algebras will be graded by subsemigroups of the multiplicative group $((0,1],\cdot)$.

\begin{Definition}\rm
(a) Let $B$ be an associative profinite $\F_p$-algebra, $w$ a valuation on $B$
and $\Omega=\la\Im w\setminus\{0\}\ra$, the subsemigroup of $((0,1],\cdot)$ generated by 
$\Im w\setminus\{0\}$. The corresponding $\Omega$-graded (associative) algebra $gr_w(B)$ is defined by
$$gr_w(B)=\oplus_{\alpha\in \Omega} B_{\alpha}/B_{<\alpha}$$
with componentwise addition and multiplication given by
$$(b +B_{<\alpha})\cdot (c +B_{<\beta}) = bc +B_{<\alpha\beta}$$
For each $b\in B$ we let $\bar b=b+B_{<w(b)}$ be the image of $b$ in $gr_w(B)$.
\vskip .12cm

(b) Let $G$ be a pro-$p$ group, $W$ a valuation on $G$ and $\Omega=\la\Im W\ra$.
The corresponding $\Omega$-graded restricted Lie algebra $L_W(G)$ is defined by
$$L_W(G)=\oplus_{\alpha\in \Omega} G_{\alpha}/G_{<\alpha}$$
with componentwise addition, Lie bracket given by
$$[g G_{<\alpha}, h G_{<\beta}] = [g,h] G_{<\alpha\beta}$$
and $p$-power operation given by $$(g G_{<\alpha})^p=g^p G_{<\alpha^p}.$$
For each $g\in G\setminus\{1\}$ we let $$\LT(g)=gG_{<W(g)}$$ be the image of $g$ in $L_W(G)$,
called the {\it leading term of $g$}.
We also set $\LT(1)=0$.
\end{Definition}

In the examples of this type the grading set $\Omega$ is often finitely generated,
and in any case $\Omega\cap (\alpha, 1]$ is finite for any $\alpha>0$, as observed above. This property makes it
possible to give inductive arguments for such gradings. Note that the familiar
$\dbN$-gradings naturally appear in this setting in the special case
when $\Omega$ is the set of positive powers of a fixed real number $\alpha\in (0,1)$.

We shall use the following elementary result:
\begin{Lemma}
\label{unenv}
Let $G$ be a pro-$p$ group, $w$ a valuation on $\Fp[[G]]$ with $w(g-1)<1$ for all $g\in G$, and define the valuation
$W$ on $G$ by $W(g)=w(g-1)$. Consider the map $\iota: L_W(G)\to gr_w(\Fp[[G]])$ defined on homogeneous elements by
$$g G_{<W(g)}\mapsto (g-1)+\Fp[[G]]_{<W(g)}\mbox{ for each } g\in G\setminus\{1\}$$
and extended to $L_W(G)$ by linearity. Then $\iota$ is a monomorphism of restricted Lie algebras, and thus we have a natural homomorphism of graded associative algebras $\calU_W(G)\to gr_w(\Fp[[G]])$ where $\calU_W(G)$ is the (restricted)
universal enveloping algebra of $L_W(G)$.
\end{Lemma}
\begin{proof}
First, we check that $\iota$ is well-defined. Indeed, suppose that $W(g)=W(h)=\alpha>0$
and $gG_{<\alpha}=hG_{<\alpha}$. Then $W(h^{-1}g)<\alpha$. Hence
$w((g-1)-(h-1))=w(g-h)=w(h(h^{-1}g-1))\leq w(h)w(h^{-1}g-1)\leq 1\cdot W(h^{-1}g)<\alpha$,
whence $g-1+\Fp[[G]]_{<\alpha}=h-1+\Fp[[G]]_{<\alpha}$. Similarly, one checks that
$\iota$ is linear on homogeneous components of $L_W(G)$, hence by construction linear
on the entire $L_W(G)$.

Next note that for any $g\in G\setminus\{1\}$ we have $\iota(g G_{<W(g)})=(g-1)+\Fp[[G]]_{<W(g)}\neq 0$
for otherwise $g-1\in \Fp[[G]]_{<W(g)}$, so $w(g-1)<W(g)=w(g-1)$, a contradiction.
Since $\iota$ is grading preserving by construction, we conclude that $\iota$ is injective.

Given $g,h\in G\setminus\{1\}$, let $a=gh-hg=(g-1)(h-1)-(h-1)(g-1)$.
Since $[g,h]-1=g^{-1}h^{-1}a$, we have
$\iota([g G_{<W(g)}, h G_{<W(h)}])=a+(g^{-1}h^{-1}-1)a + \Fp[[G]]_{<W(g)W(h)}.$
Note that $w(g^{-1}h^{-1}-1)=w((hg)^{-1}-1)<1$, so either $w(a)=0$ (hence $a=0$) or
$w((g^{-1}h^{-1}-1)a)<w(a)\leq w(g-1)w(h-1)=W(g)W(h)$. In any case,
$(g^{-1}h^{-1}-1)a\in \Fp[[G]]_{<W(g)W(h)}$, so
\begin{multline*}
\iota([g G_{<W(g)}, h G_{<W(h)}])=a + \Fp[[G]]_{<W(g)W(h)}=\\
[g-1+\Fp[[G]]_{<W(g)},h-1+\Fp[[G]]_{<W(h)}]=[\iota(g G_{<W(g)}), \iota(h G_{<W(h)})].
\end{multline*}
Thus, $\iota$ preserves the Lie bracket. The fact that $\iota$ preserves
the $p$-power operation is checked similarly.
\end{proof}

\subsection{Weight functions}

The algebras of power series $\Fp\lla Q \rra$ and free pro-$p$ groups
admit a class of valuations with particularly nice properties.
These valuations are called weight functions. The notion of a weight function
was introduced in \cite{EJ}, but associated filtrations on free pro-$p$ groups
have been considered for a long time under different names, probably starting
with Lazard's work on $p$-adic analytic groups~\cite{Laz}.

\begin{Definition}\rm Let $Q=\{q_1,q_2,\ldots \}$ and $A=\Fp\lla Q \rra$.
A continuous function $w: A\to [0,1]$ is called
a {\it weight function with respect to $Q$} if
\begin{itemize}
\item[(a)] $w$ is a valuation on $A$
\item[(b)] $w(fg)= w(f)w(g)$ for any $f,g\in A$
\item[(c)] If $f=\sum_{\alpha} \lam_{\alpha} m_{\alpha}$ where each
$\lam_{\alpha}\in \Fp$, each $m_{\alpha}$ is a monic monomial in $Q$
(that is, $m_{\alpha}$ is of the form $q_{i_1}\ldots q_{i_l}$)
and all $m_{\alpha}$ are distinct, then
$w(f)=\max\{w(m_{\alpha}) : \lam_{\alpha}\neq 0\}$.
\end{itemize}
\end{Definition}

If $w$ is a weight function on $A=\Fp\lla Q \rra$ (with respect to $Q$),
then clearly $w$ is determined by its values on $Q$; moreover, we must have
$w(q)<1$ for all $q\in Q$, for if $w(q)=1$, then $w(q^n)=1$ for all $n\in\dbN$
by condition (b), which would violate continuity of $w$ since $q^n\to 0$ as $n\to\infty$. 
Similarly, if $Q=\{q_1, q_2,\ldots \}$ is infinite, we must have $\lim_{i\to \infty} w(q_i)=0$.
Conversely, if $w_0: Q\to (0,1)$ is any function such that $\lim_{i\to \infty} w_0(q_i)=0$
if $Q$ is infinite, then $w_0$ can be (uniquely) extended to a weight function $w$
on $A$ with respect to $Q$; it is also easy to see that if $w'$ is any valuation on $A$
such that $w'_{| Q}=w_0$, then $w'(f)\leq w(f)$ for any $f\in A$.
This provides a simple characterization of weight functions among
all valuations.

\vskip .1cm

Now let $F$ be a free pro-$p$ group. Recall from \S~2.1 that if $X$ is a free generating set of $F$,
there is a canonical isomorphism $\Fp[[F]]\cong \Fp\lla Q\rra$ where $Q=\{x-1: x\in X\}$.
We define weight functions on free pro-$p$ groups using this isomorphism.

\begin{Definition}\rm Let $F$ be a  free pro-$p$ group  and $W$ a valuation on $F$.
 \begin{itemize}
\item[(a)] Let $X=\{x_1, x_2, \ldots \}$ be a free generating set of $F$. We will say that
$W$ is  a {\it weight function with respect to $X$} if there exists a weight function $w$ on $\F_p[[F]]$ 
with respect to $\{x-1:\ x\in X\}$ such that $W(f)=w(f-1)$ for all $f\in F$.
\item[(b)] We will call $W$ a {\it weight function} if $W$ is a weight function with respect to some free generating set $X$.
Any set $X$ with this property will be called {\it $W$-free}.
\end{itemize}
\end{Definition}
\vskip -.1cm
\begin{Remark} If $X$ is a free generating set of $F$ and $w$ is a weight function on $\F_p[[F]]$ with respect to $Q=\{x-1: x\in X\}$, the function $W$ on $F$ given by $W(f)=w(f-1)$ will always be a valuation (hence a weight function) on $F$. 
By Observation~\ref{obser_easy} we just need to check that $w(f-1)<1$ for all $f\in F$, which is easy. Indeed, 
for any $f\in F$ the expansion of $f-1$ as a power series in $Q$ has zero constant term. Therefore,
$w(f-1)\leq \max\{w(q): q\in Q\}$ and $\max\{w(q): q\in Q\}<1$, as observed above.
\end{Remark}
\vskip .2cm
While the definition of a weight function with respect to $X$ substantially depends on $X$, we will see that
for any weight function $W$ there are a lot of $W$-free generating sets. We will often need to establish that a given 
function $W$ is a weight function without specifying a $W$-free generating set, 
and the above definition is not well suited for this purpose. 
Instead we shall obtain a ``coordinate-free'' characterization of weight functions in terms of 
associated restricted Lie algebras (see Corollary~\ref{cor3} in \S~3.1).

\vskip .1cm

If $F$ is a finitely generated free pro-$p$ group, there is a natural family
of weight functions on $F$ which we call {\it uniform}.

\begin{Definition}\rm $\empty$
\begin{itemize}
\item[(a)] Let $F$ be a finitely generated free pro-$p$ group and $W$
a weight function on $F$. We will say that $W$ is {\it uniform}
if there is a $W$-free generating set $X$ of $F$ such that
$W$ restricted to $X$ is constant.

\item[(b)]  Let $G$ be a finitely generated pro-$p$ group and
$W$ a valuation on $G$. We will say that $W$ is {\it uniform} if
there exists a finitely generated free pro-$p$ group $F$, an epimorphism $\pi:F\to G$
and a uniform weight function $\widetilde W$ on $F$ which induces $W$ under $\pi$.
\end{itemize}
\end{Definition}

\begin{Definition}\rm Let $G$ be a finitely generated pro-$p$ group.
The {\it Zassenhaus filtration} $\{D_k G\}_{k\geq 1}$ is defined by
$D_k G=\prod\limits_{i p^j\geq k}(\gamma_i G)^{p^j}$. Equivalently,
$D_k G=\{g\in G : g-1\in \omega^k \}$ where $\omega$ is the augmentation
ideal of the completed group algebra $\Fp[[G]]$. For the proof of the
equivalence of these definitions see, e.g., \cite[Theorems~11.2~and~12.9]{DDMS}.
\end{Definition}

\begin{Proposition}
\label{uniform2}
Let $F$ be a finitely generated free pro-$p$ group, $W$ a uniform weight
function on $F$, and let $\omega$ be the augmentation ideal of $\Fp[[F]]$. 
Given $f\in F\setminus\{1\}$, let $d(f)$ be the unique integer such 
$f\in D_{d(f)} F\setminus D_{d(f)+1}$; also set $d(1)=\infty$. Then there exists $\beta\in (0,1)$ 
such that $W(f)=\beta^{d(f)}$ for all $f\in F$.
\end{Proposition}
\begin{proof} Since $W$ is uniform, there exists $\beta\in (0,1)$, a free generating set $X$ of $F$
and a weight function $w$ on $\Fp[[F]]$ with respect to $Q=\{x-1 : x\in X\}$ such that
$w(q)=\beta$ for all $q\in Q$ and $W(f)=w(f-1)$ for all $f\in F$.

Each nonzero $a\in\Fp[[F]]$ can be uniquely written as a power series in $Q$, and
by definition of $w$ we have $w(a)=\beta^{k(a)}$ where
$k(a)$ is the minimal degree of a monomial in $Q$ which appears in the
expansion of $a$ with nonzero coefficient. 

Note that the augmentation ideal $\omega$
is precisely the set of elements of $\Fp[[F]]$ whose power series expansion in $Q$ 
has zero constant term, so any nonzero $a\in \Fp[[F]]$ lies in $\omega^{k(a)}\setminus \omega^{k(a)+1}$.
Thus for any $f\in F\setminus\{1\}$ we have $d(f)=k(f-1)$, whence $W(f)=w(f-1)=\beta^{d(f)}$. 
\end{proof}

\begin{Proposition}
\label{uniform}
Let $G$ be a finitely generated pro-$p$ group and $W$ a valuation on $G$.
\begin{itemize}
\item[(a)] If $W$ is uniform, then the associated filtration $\{G_{\alpha}\}$ coincides with the Zassenhaus filtration of $G$
(apart from repetitions). In particular, each $G_{\alpha}$ is characteristic in $G$.
\item[(b)] Assume that $G$ is free and $W$ is a weight function. Then the converse of (a) holds, that is, $\{G_{\alpha}\}$ consists of  characteristic subgroups only if $W$ is uniform.
\end{itemize}
\end{Proposition}
\begin{proof} (a) Let $F$ be a free pro-$p$ group, $\pi:F\to G$ an epimorphism
and $\widetilde W$ a uniform weight function on $F$ which induces $W$.
By Proposition~\ref{uniform2} there is $\beta\in (0,1)$ such that $\widetilde W(f)=\beta^k$
for any $f\in D_k F\setminus D_{k+1}F$,  whence for any $\alpha\in (0,1)$ we have $F_{\alpha}=D_k F$ where $k=]\frac{\log(\alpha)}{\log(\beta)}[$.
Therefore, $G_{\alpha}=\pi(F_{\alpha})=\pi(D_k F)=D_k G$.
\skv
(b) Assume that $W$ is not uniform. Let $X$ be a $W$-free generating set,
and let $\alpha=\min\{W(x): x\in X\}$. Then $G_{\alpha}\neq G$ and
$G_{\alpha}$ contains at least one element of $X$, so $G_{\alpha}$
cannot be characteristic.
\end{proof}

\begin{Remark} Another characterization of uniform weight functions, somewhat similar to Proposition~\ref{uniform},
will be clear from results of Section~3: a weight function $W$ on a finitely generated free pro-$p$ group $F$ is uniform if and only if every free generating set of $F$ is $W$-free.
\end{Remark}

\subsection{Free restricted Lie algebras}
In this subsection we collect some facts about free restricted Lie algebras that will be used in the paper. 
We start with a classical result of E. Witt \cite{Witt}, for a recent
exposition of which we recommend \cite{BKS} (see also \cite[Theorem~8, p.68]{Ba}):
\begin{Theorem}[Witt]\rm
\label{witt} Every subalgebra of a free restricted Lie algebra is free.
\end{Theorem}
\begin{Remark} If $L$ is a restricted Lie algebra, by a subalgebra (resp. ideal) of $L$
we shall mean a subalgebra (resp. ideal) closed under the $p$-power operation.
By the universal enveloping algebra of $L$ we shall mean its restricted
universal enveloping algebra.
\end{Remark}
\vskip .1cm
Free restricted Lie algebras can be explicitly realized inside free associative algebras.
\begin{Proposition}
\label{restassoc}
Let $Q$ be a set and $F$ a field of characteristic $p$. Let $A=F\la Q\ra$ be the free associative
$F$-algebra on $Q$ (that is, the algebra of noncommutative polynomials), and let $L$ be the restricted
Lie subalgebra of $A$ generated by $Q$. Then $L$ is a free restricted Lie $F$-algebra on $Q$.
Moreover, if $\calU(L)$ is the universal enveloping algebra of $L$,
then the natural map $\calU(L)\to A$ induced by the embedding $L\to A$ is an isomorphism.
\end{Proposition}
\begin{proof} This is a standard result which follows, e.g., from \cite[Proposition 14, p.66]{Ba}
and  \cite[Corollary, p.52]{Ba}.
\end{proof}

The following result is well known, but we provide a proof since we are not aware of any reference.

\begin{Lemma} 
\label{hopfian}
Finitely generated free restricted Lie algebras are hopfian.
\end{Lemma}
\begin{proof} Let $F$ be a field and $L$ a restricted Lie $F$-algebra generated by a finite set $Q$. 
Let $\calU(L)$ be its universal enveloping algebra and $\omega$ the ideal of $\calU(L)$ generated by $L$.
From an explicit construction of $\calU(L)$ (see, e.g., \cite[\S~12.1]{DDMS}) it is clear that
$\calU(L)$ is generated by $Q$ as an associative $F$-algebra and $\omega$ is generated by $Q$ as an ideal.
This easily implies that for each $n\in\dbN$ the quotient $\calU(L)/\omega^n$ is finite-dimensional. 
  
 Let $D_n L=L\,\cap\, \omega^n$. It is clear that $D_n L$ is an ideal of $L$, which is
invariant under all endomorphisms of $L$; moreover, the quotient $L/D_n L$ is finite-dimensional,
since it embeds in $\calU(L)/\omega^n$. Hence if $\phi:L\to L$ is an epimorphism, then for each 
$n\in\dbN$, $\phi$ induces an epimorphism $\phi_n: L/D_n L\to L/D_n L$ which must be an isomorphism.
Therefore, $\Ker\phi\subseteq \bigcap\limits_{n\in\dbN} D_n L$.

Assume now that $L$ is free on $Q$. Then $\bigcap\limits_{n\in\dbN} D_n L\subseteq \bigcap\limits_{n\in\dbN} \omega^n=\{0\}$ 
(this is clear, e.g., from Proposition~\ref{restassoc}),
so from the previous paragraph we conclude that $L$ is hopfian.
\end{proof}

The next result provides a simple way to verify when a given generating set of a free restricted Lie algebra is a free generating set.
\begin{Lemma}\label{critfree} Let $L$ be a free restricted Lie algebra and $S$ a generating set of $L$.
Then $S$ is a free generating set of $L$ if and only if the elements of $S$ are linearly independent
mod $[L,L]+L^p$.
\end{Lemma}
\begin{proof} The forward direction is immediate by Proposition~\ref{restassoc}. Conversely,
assume that $S$ is linearly independent $\mod [L,L]+L^p$, and let $S_0$ be a finite subset of $S$.
By Theorem~\ref{witt} the subalgebra $L_0=\langle S_0\rangle$ is free.   If $X$ is a free generating set of $L_0$, then $|X|=\dim L_0/([L_0,L_0]+L_0^p)\geq \dim L_0/(([L,L]+L^p)\cap L_0)\geq |S_0|$. Thus, $|X|=|S_0|$ and $S_0$ is also a free generating set of $L_0$ by Lemma~\ref{hopfian}. Hence $S$ freely generates $L$.
\end{proof}

Our last result concerns possible gradings on free restricted Lie algebras.
If $L$ is a free restricted Lie algebra and $\Omega$ is an abelian semigroup, one can obtain
a class of $\Omega$-gradings on $L$ as follows. Choose a free generating set $Q$ of $L$ and
an arbitrary function $D:Q\to \Omega$; then there is unique $\Omega$-grading $\{L_{\alpha}\}$
of $L$ such that $q\in L_{D(q)}$ for each $q\in Q$.
We shall now show that if $\Omega$ is a subsemigroup of $((0,1),\cdot)$ such that
$\Omega\cap (\alpha, 1)$ is finite for any $\alpha>0$, then all $\Omega$-gradings of $L$ are obtained in such a way.

\begin{Proposition} \label{gradedgenerators}
Let $\Omega$ be a subsemigroup of $((0,1),\cdot)$ such that $\Omega\cap (\alpha,1)$ is finite for any
$\alpha>0$. Let $L$ be a free restricted Lie algebra. Then for any $\Omega$-grading on $L$
there exists a free generating set of $L$ consisting  of homogeneous elements.
\end{Proposition}
\begin{proof} Let $L=\oplus_{\alpha\in\Omega}L_{\alpha}$ be an $\Omega$-grading of $L$.
Since $[L,L]+L^p$ is a graded subspace of $L$, we can find its graded complement $U$ in $L$. Let $S$  be a graded basis of $U$.  We will show by (downward) induction on $\alpha\in \Omega$ that
$L_{\alpha}\subseteq \langle S\rangle$.

If $\alpha=\max \Omega$, then $L_{\alpha}\cap  ([L,L]+L^p)=\{0\}$, whence $L_{\alpha}\subseteq U$.
Now take any $\alpha\in (0,1)$, and assume that $L_{\beta}\subseteq  \langle S\rangle$ for any $\beta>\alpha$. Take any $l\in L_{\alpha}$. Then we can write $l$ as $l=u+k^p+\sum_i [l_i,m_i]$, where
$u\in U$, $k\in L_{\alpha^{1/p}}$, $l_i\in L_{\alpha_i}$, $m_i\in L_{\beta_i}$ and
$\alpha_i\beta_i=\alpha$. Hence by the induction hypothesis $l\in  \langle S\rangle$.

Thus $S$ generates $L$, and so by Lemma \ref{critfree} $S$ is a free generating set.
\end{proof}
\section{Some properties of valuations and weight functions}

In this section we establish basic properties of valuations and weight functions
on pro-$p$ groups. Some of these properties have already been proved in \cite{EJ}
(usually under additional restrictions), but we have chosen to give independent
proofs even when a precise citation to \cite{EJ} was possible. This enables us
to make this section self-contained and give simpler and more transparent proofs
for these results.

In \S~3.1 we obtain several characterizations of weight functions among
all valuations on free pro-$p$ groups.
Given a valuation $W$ on a pro-$p$ group $G$, in \S~3.2
we define the {\it $W$-rank} of $G$ and the notion of a $W$-optimal generating set
and establish several characterizations of $W$-optimal generating sets. Finally,
in \S~3.3 we define the notion of {\it $W$-index} for every closed subgroup of $G$ and
prove the weighted analogue of the Schreier formula relating $W$-rank
and $W$-index.

\subsection{Characterizations of weight functions}

We begin by recalling the definition of standard Lie and group commutators from Hall's commutator calculus~\cite{Hall}.

Let $Q$ be a countable set and $\Lie(Q)$ the free $\Fp$-Lie algebra on $Q$.
The {\it standard commutators in $Q$ of degree $1$} are simply elements of $Q$.
Suppose we have defined standard commutators of degrees $\le n -1$ and have
chosen some well-order on the set of those commutators so that $u <v$ whenever $\deg u<\deg v$.
An element $u\in \Lie(Q)$ is called a {\it standard Lie commutator in $Q$ of degree $n$} if $u=[u_1,u_2]$
where
\begin{itemize}
\item[(i)] $u_1$ and $u_2$ are standard Lie commutators with $\deg u_1 +\deg u_2=n$;
\item[(ii)] $u_1>u_2$;
\item[(iii)] If $\deg u_1>1$ and $u_1=[v_1,v_2]$, where $v_1$ and $v_2$ are standard commutators
with $v_1>v_2$, then $u_2\geq v_2$.
\end{itemize}
Standard group commutators are defined in the analogous way.

\begin{Proposition} $\empty$
\label{powercommutators}
\begin{itemize}
\item[(a)] Let $Q$ be a countable set. The elements $\{c^{p^k}: c \textrm {\ is a standard Lie commutator in $Q$}\}$
form a basis of the free restricted Lie algebra on $Q$.
\item[(b)] Let $X$ be a countable set, $F=F(X)$ the free pro-$p$ group on $X$ and $$\mathcal C=\{c:\ c \textrm {\ is a standard group commutator in $X$}\}.$$
Fix an order on $\{c^{p^k}:\ c\in \mathcal C,k\ge 0\}$.
Then any element $f\in F$ can be written in the form
$$f=\prod\limits_{c\in\mathcal C,k\ge 0} c^{p^k \alpha_{c,k}} \eqno(***)$$
where   each $\alpha_{c,k}\in\{0,\ldots, p-1\}$.
\end{itemize}
\end{Proposition}
A factorization as in (***) will be called a {\it power-commutator factorization in $X$}.
\begin{proof}
(a) is a well known result which follows, e.g. from \cite[Theorem~1, p.51]{Ba}
and \cite[Proposition~14, p.66]{Ba}.

(b) is fairly standard as well, but we will give a proof as we are not aware of a satisfactory reference.
Let $\mathcal{PC}=\{c^{p^k}:\ c\in \mathcal C,k\ge 0\}$. For every open subgroup $U$ of $F$
let $\mathcal{PC}(U)$ be the set of all elements of $\mathcal{PC}$ which do not lie in $U$
(note that $\mathcal{PC}(U)$ is always finite).
\vskip .1cm
{\it Step 1:} First we claim that for every open normal subgroup $U$ of $F$,
we can write $f=\prod_{w\in \mathcal{PC}(U)}w^{\alpha_{w, U}} r_{f,U}$ with $0\leq \alpha_{w,U}\leq p-1$
and  $r_{f,U}\in U$ (where the product is taken in the order fixed in the statement of the proposition).
This is done by induction on $[F:U]$, with the base case $U=F$ being trivial.
For the induction step, let $G=F/U$, and let $\gamma_n G$ be the last non-trivial term of the lower
central series of $G$. By Hall's basis theorem~\cite[Theorem~11.2.4]{Hall}, 
$\gamma_n G$ is generated by standard commutators of degree $n$ in $X$; moreover, 
$\gamma_n G$ is a finite $p$-group, so there exists $z\in \mathcal{PC}$ such that
$z\not \in U$, but the image of $z$ in $F/U$ is a central element of order $p$.
Applying the induction hypothesis to the subgroup $U'=\la U,z\ra$, we get that
$f=f' \cdot r_{f,U'}$, where $f'=\prod_{w\in \mathcal{PC}(U')}w^{\alpha_{w, U'}}$ with $0\leq \alpha_{w, U'}\leq p-1$
and $r_{f,U'}\in U'$. We can write $r_{f,U'}=z^{\alpha}r_{f,U}$ for some $0\leq \alpha\leq p-1$ and $r_{f,U}\in U$,
so $f=f' z^{\alpha}r_{f,U}$. This factorization of $f$ almost has the required form, except that the factor
$z^{\alpha}$ may not be in the correct position. However, since $z$ commutes with $F$ modulo $U$, we can 
permute $z^{\alpha}$ with any of the factors in $f'$, possibly replacing $r_{f,U}$ by another element of $U$.
\vskip .2cm
{\it Step 2:} Now let $\{U_n\}$ be a descending chain of open normal subgroups of $F$ which form a base
of neighborhoods of identity. By Cantor's diagonal argument, replacing $\{U_n\}$ by a subsequence, we can assume
that for any $w\in  \mathcal{PC}$, the exponent sequence $\{\alpha_{w, U_n}\}$ from Step~1 is eventually constant; 
call this constant $\alpha_w$. Since $r_{f,U_n}\to 1$ as $n\to\infty$ and every open subgroup of $F$ contains all
but finitely many elements of $\mathcal{PC}$, it is clear that $f=\prod_{w\in \mathcal{PC}}w^{\alpha_{w}}$,
which is a factorization we were looking for.
\end{proof}

We are now ready to state our first characterization of weight functions.
\vskip .1cm

\noindent
{\bf Convention:} If $G$ is a pro-$p$ group, $W$ a valuation of $G$
and $Y$ a subset of $G$, we shall consider $\{\LT(y) : y\in Y\}$
as a {\it multiset}.

\begin{Proposition} \label{cor1}Let $F$ be a free pro-$p$ group, $X$ a generating set of $F$
and $W$ a valuation on $F$. Let $S=\{\LT(x) : x\in X\}\subset L_W(F)$ and $\la S\ra$
the restricted Lie subalgebra generated by $S$. Then the following are equivalent:
\begin{itemize}
\item[(i)] $W$ is a weight function with respect to $X$
\item[(ii)] For any standard commutator $c$ in $X$ we have $W(c)=\prod_{i=1}^d W(x_i)^{n_i}$
where $x_1,\ldots, x_d$ are the elements of $X$ which occur in $c$
and $n_i$ is the number of occurrences of $x_i$ in $c$, and for any $f\in F$ given by its power-commutator factorization 
$f=\prod\limits_{c\in\mathcal C, k\geq 0} c^{p^k \alpha_{c,k}}$ we have $W(f)=\max \{W(c)^{p^k}:\ \alpha_{c,k}\ne 0\}$
\item[(iii)] $\la S\ra$ is freely generated by $S$ (as a restricted Lie algebra)
\item[(iv)] $L_W(F)$ is a free restricted Lie algebra freely generated by $S$.
\end{itemize}
\end{Proposition}
\begin{proof} We shall proceed via the sequence of implications (i)$\Rightarrow$ (iii) $\Rightarrow$ (ii) $\Rightarrow$ (i)
and (iii)$\Leftrightarrow$ (iv).

``(i)$\Rightarrow$ (iii)'' By assumption, there exists a weight function $w$ on $\F_p[[F]]$ with respect to $\{x-1 : x\in X\}$
such that $W(f)=w(f-1)$ for all $f\in F$. The definition of weight functions on algebras easily implies that $gr_w(\F_p[[F]])$ is a free associate algebra on $\overline Q=\{\overline {x-1}:\ x\in X\}$.

Let $\pi: \calU(L_W(F))\to gr_w(\Fp[[F]])$ be the canonical map (see Lemma~\ref{unenv}).
Since $\pi_{|S}$ is a bijection between $S$ and $\overline Q$, we deduce that
the associative subalgebra of $\calU(L_W(F))$ generated by $S$ is free on $S$.
Hence by Proposition~\ref{restassoc} the restricted Lie subalgebra $\la S \ra$ must be free on $S$.

``(iii)$\Rightarrow$ (ii)'' Let $c=c(x_1,\ldots,x_d)$ be a standard group commutator in $X$, and let
$\Lie(c)=\Lie(c)(\LT(x_1),\ldots, \LT(x_d))$ be the corresponding standard Lie commutator in $S$.
Since $\la S\ra$ is free restricted on $S$, we have $\Lie(c)^{p^k}\neq 0$ for any $k\geq 0$,
and an easy induction on $\deg(c)$ shows that $\LT(c^{p^k})=\Lie(c)^{p^k}$. In particular, the equality $\LT(c)=\Lie (c)$
  implies the desired formula for $W(c)$.

  Now take any $1\neq f\in F$. Let $f=\prod_{c,k} c^{p^k \alpha_{c,k}}$ be its power-commutator factorization in $X$, and
let $t= \max \{W(c)^{p^k}:\ \alpha_{c,k}\ne 0\}$. Since $W$ is a valuation, we have $W(f)\leq t$, and we need to show
that $W(f)=t$. 

As before, let $F_t=\{g\in F: W(g)\leq t\}$, $F_{<t}=\{g\in F: W(g)< t\}$, and let $\pi: F_t\to F_t/F_{<t}\subset L_W(F)$
be the natural projection. Note that $\pi(c^{p^k \alpha_{c,k}})=\alpha_{c,k}\Lie(c)^{p^k}$ if $W(c)^{p^k}=t$
and $\pi(c^{p^k \alpha_{c,k}})=0$ if $W(c)^{p^k}<t$, so 
$$\pi(f)=\sum_{W(c)^{p^k}=t}\alpha_{c,k}\Lie(c)^{p^k}.$$
Suppose now that $W(f)<t$. Then $\pi(f)=0$, whence
$ \sum_{W(c)^{p^k}=t}\alpha_{c,k}\Lie(c)^{p^k}=0.$
This contradicts the assumption that $\la S\ra$ is freely generated by $S$.

``(ii)$\Rightarrow$ (i)'' This is almost obvious. Indeed, let $W'$ be the unique weight function with respect to $X$
which coincides with $W$ on $X$. Applying the implication ``(i)$\Rightarrow$ (ii)'' to $W'$, we conclude that
the values of $W$ and $W'$ on each element of $F$ must coincide.

``(iii)$\Rightarrow$ (iv)'' We only need to show that $L_W(F)=\la S \ra$. Take any $1\ne f\in F$,
let $f=\prod_{c,k} c^{p^k \alpha_{c,k}}$ be its power-commutator factorization in $X$, and let $t=W(f)$.
The argument from the proof of the implication
``(iii)$\Rightarrow$ (ii)'' also shows that $$\LT(f)= \sum_{W(c)^{p^k}=t}\alpha_{c,k}\Lie(c)^{p^k}.$$
Thus, $\LT(f)\in \la S\ra$.

Finally, the implication ``(iv)$\Rightarrow$ (iii)'' is obvious.
\end{proof}

\begin{Corollary} \label{cor3}  Let $F$ be a free pro-$p$ group and $W$ a  valuation on $F$. Then  $L_W(F)$ is a free restricted Lie algebra if and only if $W$ is a weight function on $F$.
\end{Corollary}
\begin{proof} ``$\Leftarrow$'' If $W$ is a weight function then $L_W(F)$ is free by Proposition~\ref{cor1}.

``$\Rightarrow$'' Assume that $L_W(F)$ is free. By Proposition \ref{gradedgenerators}, there exists  a free graded generating set $S_0$  of $L_W(F)$.
We can write each element $s\in S_0$ in the form $\LT(y)$ for some $y\in F$, and
let $Y\subset F$ be such that $S_0=\{\LT(y):\ y\in Y\}$. Then $Y$ generates $F$
since $L_W(\la Y\ra)\supseteq \la S_0\ra=L_W(F)$, so
$F$ has a free generating set $X$ contained in $Y$. Put $S=\{\LT(x):\ x\in X\}$.
Since $S\subseteq S_0$,   $\la S\ra$ is a free restricted Lie algebra freely generated by $S$. Thus, $W$ is a weight function by Proposition~\ref{cor1}.
\end{proof}

Since subalgebras of free restricted Lie algebras are free restricted (Proposition \ref{witt}), Corollary~\ref{cor3}
implies the following important result:

\begin{Corollary}
\label{weight_preserve}
Let $F$ be a free pro-$p$ group and $W$ a weight function
on $F$. If $H$ is a closed subgroup of $F$, then the restriction of $W$
to $H$ is also a weight function.
\end{Corollary}

\subsection{$W$-rank and $W$-optimal generating sets}

\begin{Definition} Let $H$ be a pro-$p$ group, $W$ a valuation on $H$ and $\alpha\in [0,1]$. We put
$$c_{\alpha}(H)=\log_p[H_{\alpha}: H_{<\alpha}]$$
where as before
$$H_{\alpha}=\{h\in H : W(h)\leq \alpha\} \quad\mbox{ and }\quad
H_{<\alpha}=\{h\in H : W(h) <\alpha\}.$$
\end{Definition}

Now let $G$ be a pro-$p$ group, and denote by $\Phi(G)$ the closure of $[G,G]G^p$ (so that $\Phi(G)$
is the Frattini subgroup of $G$). Let $\Ggag=G/\Phi(G)$. Then any valuation
$W$ on $G$ induces the corresponding valuation on $\Ggag$, and we have
\begin{align*}
&\Ggag_{\alpha}=G_{\alpha}\Phi(G)/\Phi(G)=G_{\alpha}/\Phi(G)_{\alpha};&
&\Ggag_{<\alpha}=G_{<\alpha}\Phi(G)/\Phi(G)=G_{<\alpha}/\Phi(G)_{<\alpha}.&
\end{align*}

\begin{Observation}
\label{LieFrattini}
There is a natural isomorphism
$$G_{\alpha}/G_{<\alpha}/(\Phi(G)_{\alpha}/\Phi(G)_{<\alpha})\cong \Ggag_{\alpha}/\Ggag_{<\alpha}$$
In particular,
$c_{\alpha}(\Ggag)=c_{\alpha}(G)-c_{\alpha}(\Phi(G)).$
\end{Observation}
\begin{proof} This is an easy consequence of isomorphism theorems for groups.
\end{proof}

\begin{Definition}\rm Let $G$ be a pro-$p$ group and $W$ a valuation on $G$. The
quantity
$$rk_W(G)=\sum_{\alpha\in \im W} c_\alpha(\Ggag)\alpha$$ is called  the \emph{$W$-rank} of $G$.
\end{Definition}

The following proposition explains why this notion of $W$-rank is natural:

\begin{Proposition}
\label{optimal1} Let $G$ be a pro-$p$  group and $W$  a valuation on $G$.
\begin{itemize}
\item[(a)] There exists a minimal generating set $X$ of $G$
such that for each $\alpha\in\Im W$ we have $|\{x\in X: W(x)=\alpha\}|=c_{\alpha}(\Ggag)$.
In particular, if $G$ is of finite $W$-rank, $$W(X)=rk_W(G).$$
\item[(b)] If $Y$ is any generating set of $G$
then for any $\alpha\in \Im W$ we have $|\{y\in Y: W(y)\geq \alpha\}|\geq\sum_{\beta\geq \alpha}
c_{\beta}(\Ggag)$. In particular, if $G$ is of finite $W$-rank, $$W(Y)\geq rk_W(G).$$
\end{itemize}
\end{Proposition}
\begin{proof}
(a) For each $\alpha$ with $c_{\alpha}(\Ggag)\neq 0$ choose a basis $T_{\alpha}$ of
the $\Fp$-vector space $\Ggag_{\alpha}/\Ggag_{<\alpha}$. By definition of $\Ggag_{\alpha}$
and $\Ggag_{<\alpha}$ we can lift $T_{\alpha}$ to a subset $X_{\alpha}\subset G$ such that $W(x)=\alpha$
for each $x\in X_{\alpha}$. Then the set $X=\sqcup X_{\alpha}$ clearly has the required property.
\vskip .1cm
(b) Fix $\alpha\in \Im W$, and let $k=\sum_{\beta\geq \alpha} c_{\beta}(\Ggag)$.
Note that $k=\log_p [G: \Phi(G) G_{<\alpha}]$, and thus at least $k$ elements of the generating
set $Y$ of $G$ have non-trivial projection onto the space $G/\Phi(G) G_{<\alpha}$.
On the other hand, any $y\in Y$ with $W(y)< \alpha$ will have trivial projection,
and thus there must be at least $k$ elements in $Y$ with $W$-weight $\geq\alpha$.
\end{proof}

\begin{Definition}\rm Let $G$ be a pro-$p$ group and $W$ a valuation on $G$. We say that a generating set $X$ of $G$ is \emph {$W$-optimal} if for all $\alpha\in \im W$
$$|\{x\in X: W(x)=\alpha\}|=c_{\alpha}(\Ggag).$$
\end{Definition}

\begin{Corollary}
\label{cor_optimal}
Let $G$ be a pro-$p$ group and $W$ a valuation on $G$. Assume that $G$ is of finite $W$-rank. Then  a generating set $X$ of $G$
is {\it $W$-optimal} if only if $W(X)=rk_W(G)$.
\end{Corollary}
We will use the following  characterization of $W$-optimal generating sets.
\begin{Proposition} \label{critoptim} Let $G$ be a pro-$p$  group, $X\subset G$ a generating set and $W$   a valuation on $G$. Let  $S=\{\LT(x):\ x\in X\}$. Then  the image of $S$ in $ L_W(G)/L_W(\Phi(G))$ forms a basis if and only if  $X$ is
$W$-optimal.  \end{Proposition}
\begin{proof} Note that $L_W(G)/L_W(\Phi(G))$ is naturally isomorphic to
$\oplus_{\alpha} \Ggag_{\alpha}/\Ggag_{<\alpha}$ by
Observation~\ref{LieFrattini}.

``$\Rightarrow$'' If the image of $S$ in $L_W(G)/L_W(\Phi(G))$ forms a basis, then
$$|\{ x\in X : \ W(x)=\alpha\}|=|\{\LT(x): \ x\in X,\ W(x)=\alpha\}|=
\dim \Ggag_\alpha/\Ggag_{<\alpha} = c_\alpha(\Ggag),$$ so by definition $X$ is $W$-optimal.

``$\Leftarrow$''
If $X$ is $W$-optimal then again $  \dim \Ggag_\alpha/\Ggag_{<\alpha}=
|\{\LT(x): \ x\in X,\ W(x)=\alpha\}|$. Thus we only have to show that
$\{\LT(x): \ x\in X,\ W(x)=\alpha\}\subseteq G_{\alpha}/G_{<\alpha}$ is linearly independent modulo the subspace $\Phi(G)_{\alpha}/\Phi(G)_{<\alpha}$
for each $\alpha\in \Im W$.

If this does not hold, then there exists $g\in G_{<\alpha}$ and $\alpha_x\in \{0,\ldots, p-1\}$ not all equal to zero such that $g\prod_{x\in X, W(x)=\alpha} x^{\alpha_x}\in \Phi(G)$. But this means that for some $x\in X$ with $W(x)=\alpha$, the set $X\setminus\{x\}\cup\{g\}$ still generates $G$. Since $W(g)<W(x)$, this contradicts the assumption that $X$ is $W$-optimal.
\end{proof}
Next we will show that if $W$ is a weight function (on a free pro-$p$ group),
the concepts of $W$-free and $W$-optimal generating sets coincide.
\begin{Proposition} \label{optimalfree}
Let $W$ be a   weight function on a   free pro-$p$ group $F$ and $X$ a free
generating set for $F$. Then $X$ is $W$-free if and only if $X$ is $W$-optimal.
\end{Proposition}
\begin{proof} ``$\Leftarrow$'' Assume first that $X$ is $W$-optimal. Let
$S=\{\LT(x):\ x\in X\}\subset L_W(F)$ and $L=\la S\ra$,
the restricted Lie subalgebra generated by $S$.
By Theorem~\ref{witt},
$L$ is a free restricted Lie algebra. By Proposition \ref{critoptim}, the image of $S$ in $L_W(F)/L_W(\Phi(F))$ is linearly independent.
It is easy to see that $L_W(\Phi(F))\supseteq [L,L]+L^p$, so
the image of $S$ in $L/([L,L]+L^p)$ is also linearly independent. Hence by
Lemma~\ref{critfree}, $S$ is a free generating set of $L$, and so $X$ is $W$-free by Proposition~\ref{cor1}.

\vskip .1cm
``$\Rightarrow$''
Conversely, suppose that $X$ is $W$-free. To prove that $X$ is $W$-optimal, it is enough
to show that for any free generating set $X'$ of $F$ and any $\alpha>0$ we have  $|X'_{\ge \alpha}|\geq |X_{\ge  \alpha}|$ where for a set $Y$ we put
$$Y_{< \alpha}=\{y\in Y: W(y)< \alpha\}\quad\mbox{ and }\quad Y_{\geq \alpha}=\{y\in Y: W(y)\geq \alpha\}.$$

Fix $\alpha>0$. Take any $y\in (X')_{< \alpha}$, and choose its power-commutator factorization in $X$. 
By  Proposition~\ref{cor1} this factorization cannot include factors of the form $x^{n}$ with $0<n<p$ and $x\in  X_{\ge \alpha}$; in other words, $y$ must be written in the form $y=\prod_{x\in X_{<\alpha}} x^{n_x}\cdot r$
where $r\in\Phi(F)$. It follows that the image of the set $X'_{<\alpha}$ in $F/\Phi(F)$
lies in the subspace spanned by the image of $X_{< \alpha}$. 
Thus, since the image of $X^\prime$ spans $F/\Phi(F)$, the image of $X^\prime_{\ge \alpha}$ should span 
$F/\Phi(F)F_{<\alpha}$. Hence, since $X_{\ge \alpha}$ is linearly independent mod $\Phi(F)F_{<\alpha}$, 
we conclude that $|X^\prime_{\ge \alpha}|\ge |X_{\ge \alpha}|$.
\end{proof}

Given a pro-$p$ group $G$, a valuation $W$ on $G$ and some $W$-optimal generating
of $G$, it is natural to ask if one can (explicitly) construct a $W$-optimal
generating set for a closed subgroup $H$ of $G$. Our last result in this subsection
addresses this problem in the case when $G$ is free, $W$ is a weight function
and $[G:H]=p$.

\begin{Lemma}
\label{index_p0}
Let $G$ be a pro-$p$ group, $W$ a valuation on $G$ and $H$ an open subgroup of index $p$.  The following hold:
\begin{itemize}
\item[(a)] There exists a $W$-optimal generating set $X$ of $G$ and $x\in X$ such that $X\setminus\{x\}\subset H$
\item[(b)] For any $X$ and $x$ satisfying (a) the set
\begin{equation}
\label{commutators}
X'=\cup_{y\in X\setminus\{x\}}\{y, [y,x],[y,x,x], \ldots, [y,\underbrace{\! x,\ldots, x]}_{p-1\mbox
{ times }}\}\cup\{x^p\}
\end{equation}
is a generating set of $H$. Moreover  if $G$ is free and $W$ is a weight function, then $X'$ is $W$-optimal.
\end{itemize}

\end{Lemma}
\begin{proof} (a) Let $X_0$ be some $W$-optimal generating set of $G$.
Clearly, $X_0\setminus H$ is finite, and choose $x\in X_0\setminus H$ for which
$W(x)$ is smallest possible. For each $y\in X_0\setminus (H\cup\{x\})$
choose $j(y)\in\dbN$ such that $y x^{j(y)}\in H$, and let
$X=(X_0\cap H)\cup \{y x^{j(y)} : y\in X_0\setminus (H\cup\{x\})\}\cup\{x\}$.
Clearly, $X$ also generates $G$, by construction $X\setminus\{x\}\subset H$
and there is a bijection $\sigma:X_0\to X$ such that $W(\sigma(y))\leq W(y)$
for all $y\in X_0$. Since $X_0$ is $W$-optimal, $X$ is also $W$-optimal by Proposition~\ref{optimal1}.

(b) From the Schreier  process of writing generators for subgroups of finite index we know that the set $\cup_{y\in X\setminus\{x\}}\{y^{x^{i}}, i=0,\ldots, p-1 \}\cup\{x^p\}$ generates $H$. It is easy to see that
$y^{x^{i}}\in \la X'\ra$ for each $i$, so $X'$ also generates $H$.

Now, assume that $G$ is free and $W$ is a weight function.  Let $L=L_W(G)$, and let $S=\{\LT(y): y\in X\}$. By Proposition~\ref{cor1},
$L$ is a free restricted Lie algebra on $S$.

Now let $s=\LT(x)$, let
$S'=\cup_{t\in S\setminus\{s\}}\{t, [t,s],[t,s,s], \ldots, [t,\underbrace{\! s,\ldots, s]}_{p-1\mbox
{ times }}\}\cup\{s^p\}$, and let $L'=\la S'\ra$ be the restricted subalgebra
generated by $S'$. By \cite[Lemma 2.1]{BKS}, $L'$ is a subalgebra of index $p$ (=codimension $1$) in $L$, and $S'$ is a free generating set of $L'$.

Next note that $S'$ is precisely the set of leading terms of elements
of $X'$.
Thus, since $\la X'\ra=H$, we have $L_W(H)=L_W(\la X'\ra)\supseteq \la S'\ra=L'$.
We know that $[L:L']=p$ and $[G:H]=p$, whence $[L:L_W(H)]=p$.

Thus, $L_W(H)=L'$, so $L_{W}(H)$ is free restricted on $S'=\{\LT(x'): x'\in X'\}$.
By Proposition~\ref{cor1}, $W$ is a weight function on $H$ with respect to $X'$,
so $X'$ is $W$-free and hence $W$-optimal.
\end{proof}

\subsection{$W$-index and weighted Schreier formula}
If $G$ is a pro-$p$ group and $H$ is a closed subgroup of $G$,
it is easy to see that the index of $H$ in $G$ can be computed by the following formula:
$$
[G:H]=\prod\limits_{\alpha\in \Im W} p^{c_{\alpha}(G)-c_{\alpha}(H)}.$$
where the integers $c_{\alpha}(\cdot)$ are defined as in \S~3.2.
The notion of $W$-index is defined using certain generalization of this
formula:

\begin{Definition}\rm Let $G$ be a pro-$p$ group, $H$ a closed subgroup of $G$
and $W$ a valuation on $G$. The quantity $$[G:H]_W=\prod\limits_{\alpha\in \Im W}
\left(\frac{1-\alpha^p}{1-\alpha}\right)^{c_{\alpha}(G)-c_{\alpha}(H)}$$
is called the {\it $W$-index} of $H$ in $G$.
\end{Definition}

The following properties of $W$-index are straightforward:

\begin{Proposition}
\label{index}
Let $G$ be a pro-$p$ group, $W$ a valuation of $G$
and $H$ a closed subgroup of $G$. The following hold:
\begin{itemize}
\item[(a)] $[G:H]_W=\lim_{\alpha\to 0}\frac{[G:G_\alpha]_W}{[H:H_\alpha]_W}.$
\item[(b)] $H$ is of finite $W$-index in $G$ if and only if
$$\sum_{\alpha \in \im W} \alpha(c_\alpha(G)-c_\alpha(H))<\infty.$$
\item[(c)] $W$-index is multiplicative, that is, for any closed subgroup $K$ of $H$
we have $[G:K]_W=[G:H]_W\cdot [H:K]_W$.
\end{itemize}
\end{Proposition}

Our main goal in this subsection is to prove the following weighted analogue
of the Schreier index formula relating $W$-index and $W$-rank:

\begin{Theorem}[Weighted Schreier formula]
\label{ineq}
Let $G$ be a pro-$p$ group, $W$ a valuation on $G$ and $H$ a closed subgroup
for which $[G:H]_W<\infty$. Then
$$rk_W(H)-1\le [G:H]_W\cdot (rk_W(G)-1).$$
Moreover, if $G$ is free and $W$ is a weight function, then
$$rk_W(H)-1= [G:H]_W\cdot (rk_W(G)-1).$$
\end{Theorem}

We start with two lemmas dealing with the integers $c_{\alpha}(\cdot)$.
The first lemma is straightforward.

\begin{Lemma}
\label{prelim1}
Let $H$ be a closed subgroup of $G$ and $\alpha>0$.
\begin{itemize}
\item[(a)] If $K$ is a closed subgroup of $H$, the natural
map $K_{\alpha}/K_{<\alpha}\to H_{\alpha}/H_{<\alpha}$ is injective.
In particular, $c_{\alpha}(K)\leq c_{\alpha}(H)$
\item[(b)] $c_{\alpha}(H)=c_{\alpha}(G)-[H G_{\alpha}: H G_{<\alpha}]$
\item[(c)] Assume that $H$ is normal in $G$. Then $[H G_{\alpha}: H G_{<\alpha}]=c_{\alpha}(G/H)$,
and so $c_{\alpha}(G/H)=c_{\alpha}(G)-c_{\alpha}(H)$.
\end{itemize}
\end{Lemma}
The second lemma reduces the computation of the integers $c_{\alpha}$ to the case of open subgroups.
\begin{Lemma}
\label{prelim2} Let $H$ be a closed subgroup of $G$ and $\alpha>0$, and let
$U$ be an open subgroup with $H\subseteq U\subseteq HG_{<\alpha}$.

{\rm (a)} We have natural isomorphisms
$$H_{\alpha}/H_{<\alpha}\cong U_{\alpha}/U_{<\alpha},\quad
\Phi(H)_{\alpha}/\Phi(H)_{<\alpha}\cong \Phi(U)_{\alpha}/\Phi(U)_{<\alpha}\quad\mbox{ and }\quad
\Hgag_{\alpha}/\Hgag_{<\alpha}\cong \Ugag_{\alpha}/\Ugag_{<\alpha}.$$

{\rm (b)} The following equalities hold: $c_{\alpha}(H)=c_{\alpha}(U)$ and $c_{\alpha}(\Hgag)=c_{\alpha}(\Ugag)$.
\end{Lemma}
\begin{proof}
We prove (a) and (b) simultaneously.
Note that $UG_{\alpha}=HG_{\alpha}$ and $UG_{<\alpha}=UG_{<\alpha}$, so
$c_{\alpha}(H)=c_{\alpha}(U)$ by Lemma~\ref{prelim1}(b). Thus, the natural map
$H_{\alpha}/H_{<\alpha}\cong U_{\alpha}/U_{<\alpha}$, which is injective by Lemma~\ref{prelim1}(a), must be an isomorphism.

We also have $\Phi(H)\subseteq \Phi(U)\subseteq \Phi(HG_{<\alpha})\subseteq \Phi(H)G_{<\alpha}$,
and thus by the same argument $c_{\alpha}(\Hgag)=c_{\alpha}(\Ugag)$
and the natural map
$\Phi(H)_{\alpha}/\Phi(H)_{<\alpha}\to \Phi(U)_{\alpha}/\Phi(U)_{<\alpha}$ is an isomorphism.

The last isomorphism follows from the first two isomorphisms, Observation~\ref{LieFrattini}
and commutativity of the following diagram (which is straightforward to check):
\begin{equation}
\label{xy:GSdiag}
\xymatrix{
H_{\alpha}/H_{<\alpha}/(\Phi(H)_{\alpha}/\Phi(H)_{<\alpha})\ar[d]\ar[r] &
\Hgag_{\alpha}/\Hgag_{<\alpha}\ar[d]\\
U_{\alpha}/U_{<\alpha}/(\Phi(U)_{\alpha}/\Phi(U)_{<\alpha})\ar[r] &
\Ugag_{\alpha}/\Ugag_{<\alpha}
}
\end{equation}
\end{proof}

\begin{Lemma}\label{limits} Let $G$ be a pro-$p$ group, $W$ a valuation on $G$
and $H$ a closed subgroup of $G$. Then
$$[G:H]_W=\lim_{H\le U\le_o G} [G:U]_W.$$
\end{Lemma}
\begin{proof} First note that the above limit (either finite or infinite) exists
by multiplicativity of $W$-index. By Lemma~\ref{prelim2}(b) for any $\alpha>0$
and any open subgroup $U$ with
$H\subseteq U\subseteq HF_{<\alpha}$ we have $c_{\beta}(H)=c_{\beta}(U)$
for all $\beta\geq \alpha$. This observation immediately implies the desired equality.
\end{proof}

Our next result establishes the weighted Schreier formula in the case of
subgroups of index $p$.

\begin{Lemma}
\label{indexp} Let $G$ be a pro-$p$ group, $W$ a valuation on $G$ and
$U$ an open subgroup of $G$ of index $p$.
Let $X$ be a $W$-optimal generating set of $G$ and $x\in X$ be such that $X\setminus\{x\}\subset U$
(such $X$ and $x$ exist by Lemma~\ref{index_p0}).
We have
\begin{itemize}
\item[(a)] $[G:U]_W=\frac{1-\alpha^p}{1-\alpha}$ where $\alpha=W(x)$
\item[(b)] If $rk_W(G)<\infty$, then $rk_W(U)-1\le[G:U]_W\cdot (rk_W(G)-1).$ Moreover, if $G$ is free and $W$ is a weight function, then $rk_W(U)-1=[G:U]_W\cdot (rk_W(G)-1).$
\end{itemize}
\end{Lemma}
\begin{proof} (a) Since $p=[G:U]=\prod_{\gamma} p^{c_{\gamma}(G)-c_{\gamma}(U)}$, there exists unique $\beta\in (0,1)$
such that $c_{\beta}(G)= c_{\beta}(U)+1$ and $c_{\gamma}(G)=c_{\gamma}(U)$ for $\gamma\neq\beta$.
Thus we just have to prove that $\beta=\alpha$, for which
it suffices to show that $c_{\alpha}(G)\neq c_{\alpha}(U)$.  Since $c_{\alpha}(G)- c_{\alpha}(U)=\log_p [UG_{\alpha}: UG_{<\alpha}]$,
it is enough to show that $x\in UG_{\alpha}$ and $x\not\in UG_{<\alpha}$.

By definition $x\in G_{\alpha}\subseteq UG_{\alpha}$. On the other hand, if $x\in UG_{<\alpha}$, we can find $y\in G_{<\alpha}$ such that $x\in yU$.
Hence $\{y\}\cup X\setminus\{x\}$ is also a generating set of $G$, which
contradicts $W$-optimality of $X$.

\vskip .12cm
(b) Note that $W(x^p)\leq W(x)^p$ and
$W([y,\underbrace{\! x,\ldots, x}_{k\mbox { times }}])\leq W(y)W(x)^k$ for each
$y\in X\setminus\{x\}$ and $k\in\dbN$, and furthermore these inequalities become
equalities if $G$ is free and $W$ is a weight function.
If $X'$ is defined as in Lemma~\ref{index_p0}(b),
an easy computation shows that $W(X')-1\leq \frac{1-\alpha^p}{1-\alpha} (W(X)-1),$
with equality holding when $G$ is free and $W$ is a weight function.
Thus the desired result follows from part (a), Lemma~\ref{index_p0}(b)
and Corollary~\ref{cor_optimal}.
\end{proof}

Next we prove the analogue of Lemma~\ref{limits} for $W$-rank
which requires more work.

\begin{Lemma} \label{limitsrank} Let $G$ be a pro-$p$ group,  $H$ its closed subgroup and $W$ a valuation on $G$. Assume that $[G:H]_W<\infty$.
Then $$rk_W(H)=\lim_{H\le U\le _o G} rk_W(U).$$
\end{Lemma}
\begin{proof} Choose a descending chain of subgroups
$G=U_0\supset U_1\supset U_2\supset\ldots$ such that
$[U_i:U_{i+1}]=p$ for each $i$ and $\cap U_i=H$. To prove the lemma
we will show that
\begin{itemize}
\item[(a)] $rk_W(H)\leq \liminf_{i\to\infty} rk_W(U_i)$ and
\item[(b)] $rk_W(H)\geq \limsup_{i\to\infty} rk_W(U_i)$.
\end{itemize}

For a closed subgroup $K$ of $G$ and $\beta>0$ we set
$$rk_{W,\geq_{\beta}}(K)=\sum_{\alpha\geq\beta}c_{\alpha}(\Kgag)\cdot\alpha
\quad\mbox{ and }\quad rk_{W,<_{\beta}}(K)=\sum_{\alpha<\beta}c_{\alpha}(\Kgag)\cdot\alpha=rk_W(K)-rk_{W,\geq\beta}(K).$$
Equivalently, if $X$ is a $W$-optimal generating set of $K$, then
$$rk_{W,\geq_{\beta}}(K)=\sum_{x\in X,\, W(x)\geq \beta} W(x)
\quad\mbox{ and }\quad rk_{W,<_{\beta}}(K)=\sum_{x\in X,\, W(x)< \beta} W(x).$$
By Lemma~\ref{prelim2}(b) for each $\beta>0$ the sequence $rk_{W,\geq_{\beta}}(U_{i})$
eventually stabilizes and $rk_{W,\geq{\beta}}(H)=\lim_{i\to\infty} rk_{W,\geq_{\beta}}(U_{i})$. This implies (a).

To establish (b) it suffices to show that for any $\eps>0$ there exists
$\beta>0$ and $M\in\dbN$ such that $rk_{W,<_{\beta}}(U_n)<\eps$ for all $n> M$.

Lemmas~\ref{indexp} and \ref{limits} imply that there exist real numbers $\{\alpha_i\}_{i\geq 0}$
such that
\begin{itemize}
\item[(i)] $rk_W(U_{i+1})\le \frac{1-\alpha_i^p}{1-\alpha_i}\cdot rk_W(U_i)$
\item[(ii)] $[G:H]_W=\prod_{i\geq 0} \frac{1-\alpha_i^p}{1-\alpha_i}$.
\item[(iii)] $rk_{W,\geq_{\beta}}(U_{i+1})\geq rk_{W,\geq_{\beta}}(U_{i})-\alpha_i$
for any $\beta\in (0,1)$.
\end{itemize}
Recall that $[G:H]_W<\infty$. By (i) and (ii) this implies that
$\limsup_{i\to\infty} rk_W(U_i)<\infty$ and also that the series $\sum \alpha_i$ converges. Choose $M$ such that $\sum\limits_{i>M}\alpha_i<\eps/3$ and
$rk_{W}(U_{M})\geq \limsup_{i\to\infty}rk_W(U_i)-\eps/3$.
Choose $\beta>0$ such that $rk_{W,\geq_{\beta}}(U_{M})\geq rk_W(U_M)-\eps/3.$
Then for all $n>M$ we have
$$rk_{W,\geq_{\beta}}(U_{n})\geq rk_{W,\geq_{\beta}}(U_{M})-\sum_{M\leq i< n}\alpha_i
\geq \limsup_{i\to\infty}rk_W(U_i)-\eps,$$
which yields (b).
\end{proof}

Putting everything together we can now prove Theorem~\ref{ineq}.

\begin{proof}[Proof of Theorem~\ref{ineq}] In Lemma~\ref{indexp} we have already established $W$-Schreier formula
for $[G:H]=p$, and by multiplicativity of $W$-index (Proposition~\ref{index}(c)) the formula extends to arbitrary open subgroups. Finally, the general case follows from Lemma \ref{limits}
and Lemma \ref{limitsrank}.
\end{proof}

We finish this section with an analogue of Lemma~\ref{limitsrank}
dealing with weighted ranks of quotient groups. It will be used
in the next section.

\begin{Lemma}
\label{rank_quot} Let $G$ be a pro-$p$ group, $W$ a valuation on $G$,
and $\{H_n\}_{n\in\dbN}$ a descending chain of closed normal subgroups of $G$
with $\cap H_n=\{1\}$. Then
$$\lim_{n\to\infty}rk_W(G/H_n)=rk_W(G).$$
\end{Lemma}
\begin{proof}
First note that if $K,L$ are closed normal subgroups of $G$, with $K\subseteq L$,
then $rk_W(G/L)\leq rk_W(G/K)$. This implies that $\lim_{n\to\infty}rk_W(G/H_n)$
exists (finite or infinite) and $\lim_{n\to\infty}rk_W(G/H_n)\leq rk_W(G)$.

By a standard argument (see, for example, \cite[Proposition 2.1.4 (a)]{RZ}) for each $\beta>0$ there exists
$n=n(\beta)$ such that $H_k\subseteq G_{<\beta}$ for all $k\geq n$.
Then, in the notations of Lemma~\ref{limitsrank},
$rk_{W,\geq\beta}(G/H_k)=rk_{W,\geq\beta}(G)$ for all $k\geq n$.
This proves the opposite inequality $\lim_{n\to\infty}rk_W(G/H_n)\geq rk_W(G)$.
\end{proof}

\section{Weighted presentations and weighted deficiency}

In this section we introduce several variations of the notion of weighted deficiency
for pro-$p$ groups and their presentations.
We will not give a separate name for each type of weighted deficiency,
as it should always be clear from the context and notations what we are talking about.
We start in \S~4.1 by defining the deficiency of a pro-$p$ presentation with respect to a weight function and the weighted deficiency of a pro-$p$ group. In \S~4.2 we consider the more involved notion of
the deficiency of a pro-$p$ group with respect to a finite valuation. In \S~4.3
we introduce virtual valuations which will play a key role
in the proof of Theorem~\ref{zeroone}, and in \S~4.4 we define the
deficiency associated to a finite virtual valuation. Finally, in \S~4.5 we introduce
pro-$p$ groups of positive virtual weighted deficiency (PVWD) which naturally
arise in the proof of Theorem~\ref{zeroone}.

Before talking about weighted deficiency we need to define what it means
for a valuation on a pro-$p$ group to be {\it finite}. In the case of weight functions, the definition is the obvious one.

\begin{Definition}\rm A weight function $W$ on a free pro-$p$ group $F$ will be called {\it finite}
if $rk_W(F)<\infty$.
\end{Definition}

Somewhat surprisingly, the extension of this definition to arbitrary valuations
is more complex:

\begin{Definition}\rm A valuation $W$ on a  pro-$p$ group $G$ will be called \emph {finite} if there exists $Y\subset G$ such that $\{\LT(y):\ y\in Y\}$ generate $L_{W}(G)$ and $W(Y)=\sum_{y\in Y}W(y)$ is finite.
\end{Definition}

\begin{Remark}
The fact that the two definitions coincide in the case of weight functions follows
from Proposition~\ref{cor1}. In general, if a valuation $W$ on $G$ is finite, then $rk_W(G)<\infty$, but the converse need not hold.
\end{Remark}

We will see in \S~4.2 why the above definition of a finite valuation is
convenient to use.

\subsection{Weighted deficiency of pro-$p$ presentations and pro-$p$ groups}

\begin{Definition}\rm Let $(X,R)$ be a pro-$p$ presentation
by generators and relators. Let $F=F(X)$ be the free pro-$p$ group on $X$.
\begin{itemize}
\item[(a)] Given a finite weight function $W$ on $F$ with respect to $X$, the quantity
$$def_W(X,R)=W(X)-W(R)-1$$
is called the {\it $W$-deficiency of $(X,R)$}.
\item[(b)] The weighted deficiency of the presentation $(X,R)$,
denoted by $wdef_p(X,R)$ is the supremum of the set $\{def_W(X,R)\}$
where $W$ runs over all finite weight functions on $F$ with respect to $X$.
\end{itemize}
\end{Definition}
\begin{Remark} The subscript $p$ in the notation $wdef_p(X,R)$ is used
to avoid confusion in the case when $R$ is a subset of $F_{abs}(X)$
(the free abstract group on $X$) since in this case
we can consider $(X,R)$ as a pro-$p$ presentation for different primes $p$.
\end{Remark}

\begin{Definition}\rm Let $G$ be a pro-$p$ group. The {\it weighted deficiency}
of $G$ denoted by $wdef(G)$ is the supremum of the set $\{wdef_p(X,R)\}$ where $(X,R)$
runs over all (pro-$p$) presentations of $G$.
\end{Definition}

\begin{Lemma} $\empty$ \label{trivial}
\begin{itemize}
\item[(a)] The weighted deficiency of the trivial group $E$ (considered as a pro-$p$ group)
is equal to $-1$.

\item[(b)] Let $G$ be a finitely generated pro-$p$ group and $d(G)$ its minimal number of generators.
Then $wdef(G)\leq d(G)-1$.
\end{itemize}
\end{Lemma}
\begin{proof} (a) Clearly, $wdef(E)\geq -1$. To prove the reverse inequality, let $(X,R)$ be
any pro-$p$ presentation of $E$ and $W$ a finite weight function on $F=F(X)$ with respect to $X$.
Then $R$ generates $F$ as a (closed) normal subgroup and hence $R$ also generates $F$ as a pro-$p$ group.
Since $X$ is a $W$-optimal generating set for $F$, we have $W(R)\geq W(X)$, whence $W(X)-W(R)-1\leq -1$.

(b) Let $(X,R)$ be a presentation of $G$ and $W$ a weight function on $F=F(X)$ with respect to $X$.
Let $Y\subseteq F$ be a generating set of $G$ with $|Y|=d(G)$. Then $(X,R\cup Y)$ is a presentation
of the trivial group, so by (a) $W(X)-W(R\cup Y)-1\leq -1$. Since $W(R\cup Y)=W(R)+W(Y)\leq W(R)+d(G)$,
we obtain that $W(X)-W(R)-1\leq d(G)-1$.
\end{proof}

Next we define weighted deficiency corresponding to a slightly different notion of a (pro-$p$) presentation, where we will specify only a free pro-$p$ group and its normal subgroup, but
not generators and relations.

\begin{Definition}\rm $\empty$
\begin{itemize}
\item[(i)]  A {\it weighted presentation}
is a triple $(F,N,W)$ where $F$ is a free pro-$p$ group, $N$ is a (closed) normal subgroup of $F$ and $W$ is a finite weight function on $F$.
\item[(ii)] Given a weighted presentation
$(F,N,W)$, we define $def_W(F,N)$ to be the supremum of the set
$\{def_W(X,R)\}$ where $X$ is a $W$-free generating set of $F$
and $R$ is a set of normal generators of $N$.
\end{itemize}
\end{Definition}

We have the following ``closed'' formula for $def_W(F,N)$:
\begin{Lemma} \label{wdef}
Let $(F,N,W)$ be a weighted presentation. The following hold:
\begin{itemize}
\item[(a)] If $R$ is a set of normal generators for $N$, then
$W(R)\geq rk_W(N/[N,F])$, and there exists $R$ for which equality holds.
\item[(b)] $def_W(F,N)=rk_W(F)-rk_W(N/[N,F])-1.$
\end{itemize}
\end{Lemma}
\begin{proof}
(a) is a direct consequence of the following well-known fact: if $R$ is a subset
of $N$, then $R$ generates $N$ as a normal subgroup of $F$ if and only if
the image of $R$ in $N/[N,F]$ generates $N/[N,F]$ as a subgroup.

(b) follows from (a) and the fact that $rk_W(F)=W(X)$ for any
$W$-free generating set $X$ (since $W$-free = $W$-optimal by Proposition~\ref{optimalfree}).
\end{proof}

\begin{Proposition}
\label{prop:descent1}
Let $F$ be a free pro-$p$ group, $W$ a finite weight function on $F$ and $(F,N,W)$ a weighted presentation.  Let $F'$ be a closed subgroup  of $F$ containing $N$, and assume that
$[F:F']_W<\infty$. Then
\begin{equation}\label{def:eq}
def_W(F',N)\geq def_W(F,N)\cdot [F:F']_W.
\end{equation}
\end{Proposition}

\begin{proof}
{\it Case 1:} $F'$ has index $p$ in $F$. In view of
 Theorem \ref{ineq}, we only have to show that for any subset $R\subseteq N$ with
$\la R\ra^F=N$ there exists a subset $R'\subseteq N$ with $\la R'\ra^{F'}=N$
such that $$W(R')\le [F:F']_W \cdot W(R).$$
So assume that $R$ generates $N$ as a normal subgroup of $F$,
and set
 $$R'=\{r,[r,x],\ldots,[r,\underbrace{\! x,\ldots, x]}_{p-1\mbox
{ times }}\}:\ r\in R\}.$$
Then by the Schreier rewriting process $R'$ generates $N$ as a normal subgroup of $F'$.
Since $W([r,\underbrace{\! x,\ldots, x]}_{k\mbox
{ times }})\leq W(r)W(x)^k$, using Lemma~\ref{indexp} we get
$$W(R')\leq W(R)\frac{1-W(x)^p}{1-W(x)}=W(R)\cdot [F:F']_W.$$

\noindent
{\it Case 2:} $F'$ has arbitrary finite index in $F$. In this case the proposition follows
from Case~1 by multiplicativity of $W$-index.

\vskip .1cm
\noindent
{\it Case 3:} $F'$ is of infinite index in $F$. By Lemma~\ref{rank_quot}, $$rk_W(N/[N,F'])=\displaystyle \lim_{F'\le U\le _o F} rk_W(N/[N,U]).$$ Hence,
using Lemma~\ref{limits}, Lemma \ref{limitsrank}, Lemma \ref{wdef} and the result in Case~2
we have
\begin{multline*}
def_W(F',N)  =  rk_W(F')-rk_W(N/[N,F'])-1 = \displaystyle \lim_{F'\le U\le _o F} ( rk_W(U)-rk_W(N/[N,U])-1)\\
=\displaystyle \lim_{F'\le U\le _o F} def_W(U,N)\ge \displaystyle \lim_{F'\le U\le _o F}  def_W(F,N)\cdot [F:U]_W  =def_W(F,N)\cdot [F:F']_W.
\end{multline*}
\end{proof}
\begin{Remark}
The part of the argument dealing with the case $[F:F']<\infty$ essentially
appears in the proof of \cite[Theorem~3.11]{EJ}.
\end{Remark}

As an immediate consequence of Proposition~\ref{prop:descent1}, we deduce that the class of PWD pro-$p$ groups is closed under taking open subgroups, a fact we stated in the introduction. Proposition~\ref{prop:descent1} also yields a very simple proof of the infiniteness of PWD groups; in fact, the argument applies to groups of non-negative weighted deficiency.

\begin{Corollary} A pro-$p$ group of non-negative weighted deficiency 
must be infinite.
\end{Corollary}
\begin{proof}
Let $G$ be a pro-$p$ group of non-negative weighted deficiency. By Lemma~\ref{trivial}(a) $G$ is non-trivial,
so we can find a proper open subgroup $H$ of $G$. By Proposition~\ref{prop:descent1}
$H$ also has non-negative weighted deficiency, and we can apply the same argument to $H$. 
This process can be continued indefinitely, so $G$ must be infinite.
\end{proof}

\subsection{Presentations of valued pro-$p$ groups}

Let $G$ be a pro-$p$ group and $(F,N,\widetilde W)$ a weighted presentation of $G$,
and fix an epimorphism $\pi:F\to G$ with $\Ker\pi=N$. Recall that $\widetilde W$
induces a valuation $W$ on $G$ given by
$$W(g)=\min\{\widetilde W(f) : \pi(f)=g\}.$$
We shall show (see Proposition~\ref{valuation_up} below) that each valuation $W$ on a
pro-$p$ group $G$ arises from some weighted presentation $(F,N,\widetilde W)$ of $G$ in such way. If in addition we want $F$ to be of finite $\widetilde W$-rank, we need to assume that the valuation $W$ is finite -- this explains the definition of
a finite valuation given at the beginning of the section.

\vskip .1cm
We start with two auxiliary lemmas.
\begin{Lemma}
\label{lalg_induce}
Let $\phi: K\to G$ be a homomorphism of pro-$p$ groups, $\widetilde W$ a valuation on $K$ and $W$ a valuation on $G$.  Assume   $\phi(K_\alpha)\subseteq G_\alpha$ for all $\alpha$. Then  the induced map  $\overline \phi: L_{\widetilde W}(K)\to L_W(G)$ is a homomorphism of restricted Lie algebras.
Moreover if $\overline \phi$ is surjective, then $\phi(K_\alpha)=G_\alpha$, and so $\widetilde W$ induces $W$.
\end{Lemma}
\begin{proof}The first assertion of the lemma is standard.
Let us show that $\phi(K_\alpha)=G_\alpha$. Since $L_{\widetilde W}(K)$ maps onto $L_W(G)$, we have $G_\alpha=\phi(K_\alpha)G_{<\alpha}$. Let
$\alpha=\alpha_0>\alpha_1>\alpha_2>\ldots$ be all possible values of $W$
which are $\leq \alpha$, so that $G_{\alpha_i}=G_{<\alpha_{i-1}}$ for $i\geq 1$.
Then we have
$$G_\alpha=\phi(K_\alpha)G_{<\alpha_0}=\phi(K_\alpha)G_{\alpha_1}=\phi(K_\alpha)\phi(K_{\alpha_1})G_{<\alpha_1}=\phi(K_\alpha)G_{<\alpha_1}=\ldots= \phi(K_\alpha).$$
\end{proof}

\begin{Lemma}
\label{group_induce} Let $\pi: F\to G$ be a homomorphism of pro-$p$ groups
where $F$ is free, $\widetilde W$ a weight function on $F$ and $W$ a valuation on $G$.
Suppose that $F$ has a $\widetilde W$-free generating set $X$ such that
$\widetilde W(x)\geq W(\pi(x))$ for any $x\in X$. Then $\widetilde W(f)\geq W(\pi(f))$
for any $f\in F$ and therefore $\pi(F_{\alpha})\subseteq G_{\alpha}$ for all $\alpha$.
\end{Lemma}
\begin{proof} This is an easy consequence of the implication ``(i)$\Rightarrow$(ii)'' in Proposition~\ref{cor1}.
\end{proof}

\begin{Definition}\rm Let $G$ be a pro-$p$ group and $W$ a valuation on $G$.
A {\it presentation} of $(G,W)$ is a triple $(F,\pi ,\widetilde W)$ where
$F$ is a free pro-$p$ group, $\pi:F\to G$ is an epimorphism and $\widetilde W$ is a weight function on $F$ which induces $W$ under $\pi$.
\end{Definition}

\begin{Proposition}
\label{valuation_up}
 Let $G$ be a pro-$p$ group and $W$ a valuation on $G$. Then there exists a presentation $(F,\pi ,\widetilde W)$ of $(G,W)$. Moreover, the weight function $\widetilde W$ can be chosen
finite if and only if $W$ is finite.
\end{Proposition}
\begin{proof}  Let $Y$ be a countable subset of $G\setminus \{1\}$ such that
$S=\{\LT(y):\ y\in Y\}$ generates $L_W(G)$. Then $Y$ generates $G$, and if
$Y$ is infinite, we can assume that $Y$ converges to $1$.
Consider the free pro-$p$ group $F=F(Y)$ and let $\pi: F\to G$ be the natural
epimorphism that extends the map $Y\to G$. Let $\widetilde W$ be the weight function on $F(Y)$
with respect to $Y$ such that $\widetilde W(y)=W(\pi(y))$ for all $y\in Y$.

By Lemma~\ref{group_induce}, $\pi (F_\alpha)\subseteq G_\alpha$ for all $\alpha$. Furthermore, $\overline \pi(\widetilde{\LT}(y))=\LT(y)$ for all $y\in Y$ (where $\widetilde{\LT}$
denotes the leading term in $L_{\widetilde W}(F)$), and so by the choice of $S$
we have $\overline\pi (L_{\widetilde W}(F))=L_W(G)$. Hence by Lemma~\ref{lalg_induce},
$\widetilde W$ induces $W$ under $\pi$.

By our construction, if $W$ is finite, then the weight function $\widetilde W$
is finite. On the other hand, if $(F,\pi ,\widetilde W)$ is any presentation of $(G,W)$
with $\widetilde W$ finite, then clearly $W$ is finite.
\end{proof}

Proposition~\ref{valuation_up} enables us to define the deficiency of
a pro-$p$ group with respect to a finite valuation.

\begin{Definition}\rm
Let $G$ be a pro-$p$ group and $W$ a finite valuation on $G$.
The {\it deficiency of $G$ with respect to $W$}, denoted by $def_W(G)$, is defined to be the supremum of the set  $\{def_{\widetilde W}(F,\Ker\pi)\}$, where $(F,\pi,\widetilde W)$ runs over all presentations of $(G,W)$, with $\widetilde W$ finite.
\end{Definition}

We can now give a useful reformulation of Proposition~\ref{prop:descent1} (see Corollary~\ref{prop:descent11}).

\begin{Lemma}
\label{index_stupid}
Let $\pi:\widetilde G\to G$ be an epimorphism of pro-$p$ groups.
Let $W$ be a valuation on $G$ and $\widetilde W$ a valuation on $\widetilde G$ which induces
$W$ under $\pi$. Then for any closed subgroup $H$ of $G$ we have
$$[G:H]_W=[\widetilde G: \widetilde H]_{\widetilde W}\quad \mbox{ where }\quad  \widetilde H=\pi^{-1}(H).$$
\end{Lemma}
\begin{proof} Let $N=\Ker\pi$. By Lemma~\ref{prelim1}(c) we have
$c_{\alpha}(G)=c_{\alpha}(\widetilde G/N)=c_{\alpha}(\widetilde G)-c_{\alpha}(N)$,
and similarly $c_{\alpha}(H)=c_{\alpha}(\widetilde H)-c_{\alpha}(N)$.
Therefore, $c_{\alpha}(G)-c_{\alpha}(H)=c_{\alpha}(\widetilde G)-c_{\alpha}(\widetilde H)$,
which implies the equality $[G:H]_W=[\widetilde G: \widetilde H]_{\widetilde W}$ by definition.
\end{proof}

\begin{Corollary}
\label{prop:descent11}
Let $G$ be a pro-$p$ group, $W$ a finite valuation on $G$, and
let $H$ be a closed subgroup of $G$ with $[G:H]_W<\infty$. Then $$def_W(H)\geq def_W(G)\cdot [G:H]_W.$$
\end{Corollary}

\begin{proof} This follows directly from Proposition~\ref{prop:descent1} and Lemma~\ref{index_stupid}.
\end{proof}

\subsection{Virtual valuations}

In our analysis of PWD groups we will need to deal with situations
when a pro-$p$ group $G$ does not necessarily have PWD, but some
open subgroup $U$ of $G$ does have positive deficiency with respect to
some valuation $W$ on $U$ (note that by Corollary~\ref{prop:descent11}
we can then assume that $U$ is normal). As we will see, the group $G$ will
still satisfy key properties of PWD groups under the additional
assumption that $W$ is $G$-invariant, as defined below. In fact, we may even allow
$G$ to be a virtually pro-$p$ group (the open subgroup $U$ must be pro-$p$).

\begin{Definition}\rm Let $G$ be a profinite group and $U$ an open normal pro-$p$ subgroup of $G$.
A valuation $W$ on $U$ will be called \emph{$G$-invariant} if
$$W(u^g)=W(u)\mbox{ for any } g\in G\mbox{ and }u\in U.$$
We will also refer to such $W$ as a \emph{virtual valuation of $G$ defined on $U$}.
\end{Definition}

If $G$ is a pro-$p$ group, $W$ is a valuation on $G$ and $U$ is an open normal subgroup of $G$, then $W$ restricted to $U$ is a $G$-invariant valuation on $U$. Here is another simple way to produce virtual valuations:

\begin{Observation}
\label{exampvirtval}
Let $G$ be a finitely generated profinite group, $U$ an open normal pro-$p$ subgroup of $G$ and 
$W$ a uniform valuation on $U$. Then $W$ is $G$-invariant, so $W$ is a virtual valuation of $G$.
\end{Observation}
\begin{proof} This is a direct consequence of Proposition~\ref{uniform}.
\end{proof}

\begin{Remark} 
A virtual valuation of $G$ may not be extendable to a valuation on $G$,
even if $G$ is pro-$p$ and the hypotheses of Observation~\ref{exampvirtval} hold.
For instance, let $X$ be any finite set with $|X|\geq 2$
and $G=F(X)$ the free pro-$p$ group on $X$. Fix $x\in X$, and let $U$ be the unique normal subgroup of 
index $p$ in $G$ which contains $X\setminus\{x\}$. Let $W$ be a uniform weight function on $U$.
Then $W$ is a virtual valuation of $G$ by Observation~\ref{exampvirtval}.
By Proposition~\ref{uniform2}, $W$ is constant on the set $U\setminus [U,U]U^p$; 
in particular, $W(y)=W([x,y])$ for any $y\in X\setminus\{x\}$. On the other hand,
if $W$ was extendable to $G$, we would have had $W([x,y])\leq W(x)W(y)<W(y)$.   
\end{Remark}
\vskip .15cm

In the case of weight functions  invariance can often be checked using
the following criterion.

\begin{Lemma} \label{critinv} Let $F$ be a profinite  group and $U$ an open normal free pro-$p$ subgroup of $F$.
Let $W$ be a weight function on $U$ and $X$ a $W$-optimal generating set of $U$.
Then $W$ is $F$-invariant if and only if $W(x^f)\leq W(x)$ for all $x\in X$ and $f\in F$.
\end{Lemma}
\begin{proof} The forward direction is obvious. For the reverse direction 
first note that we must have equality $W(x^f)= W(x)$
for all $x\in X$ and $f\in F$ for otherwise $X$ will not be $W$-optimal.

Now fix $f\in F$, let $X^{f}=\{x^f: f\in F\}$, and define $W':F\to [0,1)$ by $W'(g)=W(g^{f^{-1}})$. 
Then clearly $W'$ is a weight function with respect to the set $X^{f}$. On the other hand, for each $x\in X$
we have
$$W'(x)=W(x^{f^{-1}})=W(x)=W'(x^f)$$
Hence, $X$ is also $W'$-optimal and $W'(x)=W(x)$ for all $x\in X$. Therefore,
$W'=W$ as functions, which is equivalent to $f$-invariance of $W$.
\end{proof}

If $W$ is a weight function on a free pro-$p$ group $F$, one can construct another
weight function $W'$ on $F$ with respect to the same free generating set $X$
by dividing the weights of all elements of $X$ by the same constant $c\geq 1$.
This simple-minded operation, called the {\it $c$-contraction}, proved to
be very useful in \cite{EJ} and will again play a key role in this paper.

\begin{Definition}\rm Let $F$ be a  free pro-$p$ group
and $W$ a weight function on $F$. Let $c\geq 1$ be a real number.
Choose any $W$-free generating set $X$ of $F$, and let $W'$ be the unique
weight function on $F$ with respect to $X$ such that $W'(x)=W(x)/c$ for all $x\in X$.
Then we will say that the pair $(F,W')$ is obtained from $(F,W)$
by the {\it $c$-contraction}. It is easy to see that $W'$
is independent of the choice of $X$.
\end{Definition}

The next result shows that $c$-contractions can be applied not only
to weight functions, but also to virtual weight functions.

\begin{Lemma}
\label{contr-inv}
Let $F$ be a profinite group, $U$ an open normal free pro-$p$ subgroup of $F$
and $W$ an $F$-invariant weight function on $U$. Let $c\geq 1$, and
let $(U,W)\to (U,W')$ be the $c$-contraction. Then
the weight function $W'$ is also $F$-invariant.
\end{Lemma}
\begin{proof}
First note that $W'(g)\leq W(g)/c$ for all $g\in U$.
This follows, for instance, from the formula for the weight of an
element given by its power-commutator factorization (see Proposition~\ref{cor1}(ii)).
Now if $X$ is a $W$-optimal generating set for $U$,
then for any $x\in X$ and $f\in F$ we have
$$W'(x^f)\leq \frac{W(x^f)}{c}=\frac{W(x)}{c}=W'(x).$$
Thus, $W'$ is $F$-invariant by Lemma \ref{critinv}.
\end{proof}

\subsection{Presentations of virtually valued virtually pro-$p$ groups}

\begin{Definition}\rm Let $G$ be a profinite group and $W$ a virtual valuation of $G$
defined on an open normal pro-$p$ subgroup $U$. A {\it presentation} of $(G,W)$ is a triple
$(F,\pi ,\widetilde W)$ where $F$ is a profinite group, $\pi:F\to G$ is an epimorphism
such that $\pi^{-1}(U)$ is a free pro-$p$ group and $\widetilde W$ is an $F$-invariant weight function
on $\pi^{-1}(U)$ which induces $W$ under $\pi$.
\end{Definition}

The following result generalizes Proposition~\ref{valuation_up}
to the case of virtual valuations.

\begin{Proposition}
\label{virtualvaluation_up}
 Let $G$ be a virtually pro-$p$ group and $W$ a virtual valuation of $G$. Then there exists a presentation $(F,\pi ,\widetilde W)$ of $(G,W)$. Moreover, $\widetilde W$ can be chosen finite
if and only if $W$ is finite.
\end{Proposition}
\begin{proof}  Assume that $W$ is defined on an open normal pro-$p$ subgroup $U$ of $G$. Let $Y_1$ be a countable subset  of $U\setminus \{1\}$ such that $S=\{\LT(y):\ y\in Y_1\}$ generates $L_W(U)$ (as in
Proposition~\ref{valuation_up} we can assume that $Y_1$ converges to $1$
if $Y_1$ is infinite). By definition we can make $W(Y_1)$ finite if and only if $W$ is finite. Let $Y_2\subset G\setminus U$ be a finite set that generates $G/U$. Let $Y=Y_1\sqcup Y_2$.
The rest of the proof is divided in two steps -- constructing a presentation $(F,\pi)$ for $G$
(Step~1) and then constructing a weight function $\widetilde W$ on $\pi^{-1}(U)$  which induces $W$ under $\pi$ (Step~2).

\vskip .1cm

{\it Step 1: Constructing $F$ and $\pi$.}
Let $F_{abs}=F_{abs}(Y)$ be the free abstract group on $Y$ and let $\pi_{abs}:F_{abs}\to G$
be the natural homomorphism that extends the map $Y\to G$. For $i=1,2$
let $F_{i,abs}$ be the subgroup of $F_{abs}$ generated by $Y_i$ (so that $F_{abs}=F_{1,abs}\ast F_{2,abs}$).
Let $V_{abs}=\pi_{abs}^{-1}(U)$ and $V_{2,abs}=V_{abs}\cap F_{2,abs}$.

Let $Z$ be any free generating set of $V_{2,abs}$.
Note that $|Z|<\infty$ since $V_{2,abs}$ has finite index in
$F_{2,abs}$. Let $T=\{t_1=1,\ldots, t_s\}$ be a
Schreier transversal of $V_{2,abs}$ in $F_{2,abs}$ with respect to $Y_2$.
Then $T$ is also a Schreier transversal of $V_{abs}$ in $F_{abs}$,
and it is easy to see that $V_{abs}$ is freely generated by the set
$$X=(\sqcup_{i=1}^s Y_1^{t_i})\sqcup  Z\eqno (***)$$
(simply apply the Schreier rewriting process to the generating set $Y$ of $F_{abs}$
and the transversal $T$ to obtain a free generating set of $V_{abs}$).

Now let $\widehat{F}$ be the free profinite group on $Y$ and let $\widehat{V}$ be the closure of $V_{abs}$ in
$\widehat{F}$. Let $V=\widehat V/H$ be the maximal pro-$p$ quotient of $\widehat V$; observe that
$V$ is a free pro-$p$ group on $X$. Also note that $H$ is normal in $\widehat{F}$ (since $\widehat{V}$ is normal in $\widehat{F}$ and $H$ is characteristic in $\widehat V$), and thus we can consider the virtually pro-$p$ group $F=\widehat{F}/H$ (containing $V$ as an open subgroup). Let $F_i$ be the closed subgroup of $F$ generated by $Y_i$
(for $i=1,2$) and $V_2=V\cap F_2$. Note that $T$ is a transversal of $V$ in $F$ (and also of $V_2$ in $F_2$).

Finally, let $\hat \pi:\widehat F\to G$ be the natural epimorphism that extends $\pi_{abs}:F_{abs}\to G$.
Since $\hat \pi(\widehat{V})=U$ is pro-$p$, we can factor $\hat \pi$
through the map $\pi: F\to G$. Note that $V=\pi^{-1}(U)$.
\vskip .14cm

\noindent {\it Step 2: Choosing $\widetilde W$.}
Recall that $V$ is free pro-$p$ on $X$.
Let $\alpha_0=\max \{W(g): g\in U\}$, fix $\alpha_0<\alpha<1$, and let
$\widetilde W$ be the weight function on $V$ with respect to $X$
such that  $\widetilde W(y)=W(\pi(y))$ for any $y\in \cup_{i=1}^s Y_1^{t_i}$ and
$\widetilde W(z)=\alpha$ for any $z\in Z$.

We will now show that $\widetilde W$ is $F$-invariant.
Let $f\in F$, and write $f=tv$, where $t\in T$ and $v\in V$.
Since $W$ is $G$-invariant and $\widetilde W$ is $V$-invariant,
for any $y\in Y_1$ we have
$$\widetilde W(y^f)=\widetilde W((y^t)^v)=\widetilde W(y^t)=W(\pi(y^t))=W(\pi(y)^{\pi(t)})=W(\pi(y))=\widetilde W(y).$$ Since $f$ is arbitrary, we deduce that $\widetilde W(y^f)=\widetilde W(y)$ for any $y\in \cup_{i=1}^s Y_1^{t_i}$  as well.

The restriction of $\widetilde W$ to $V_2$ is $F_2$-invariant by Observation~\ref{exampvirtval}, so
for any $z\in Z$ we obtain
$$\widetilde W(z^f)=\widetilde W((z^t)^v)=\widetilde W(z^t)=\widetilde W(z).$$
Hence $\widetilde W$ is $F$-invariant by Lemma~\ref{critinv}.

By Lemma~\ref{group_induce} we have $\pi (V_\gamma)\subseteq U_\gamma$
for all $\gamma$, and by the same argument as in Proposition~\ref{valuation_up}
we have $\bar \pi (L_{\widetilde W}(F_1))=L_W(U)$.
Hence  by Lemma~\ref{lalg_induce} $\widetilde W$ induces $W$.
Thus $(F,\pi,\widetilde W)$ is a presentation of $(G,W)$, and
$\widetilde W$ is finite if and only if $W$ is finite.
\end{proof}

We can now define the notion of deficiency with respect to a virtual valuation:

\begin{Definition}\rm
Let $G$ be a profinite group and $W$ a finite virtual valuation of $G$ defined on
an open normal pro-$p$ subgroup $U$. The \emph{$G$-invariant deficiency of $U$ with respect
to $W$}, denoted by $def_W^G(U)$, is the supremum of the set
$\{def_{\widetilde W}(\pi^{-1}(U),\Ker\pi)\}$, where
$(F,\pi,\widetilde W)$ is a presentation of $(G,W)$.
\end{Definition}

The following elementary result describes how weighted deficiency may change
when a profinite group $G$ is replaced by its quotient.
It will be applied very frequently in the sequel.

\begin{Lemma}\label{addrel} Let $G$ be a profinite group and $W$ a finite virtual valuation
of $G$ defined on an open normal pro-$p$ subgroup $U$.
Let $S$ be a subset of $U$ and $N$ the normal subgroup of $G$ generated by
$S$. Then $$def_W ^{G/N}(U/N)\ge def_W^G(U)-[G:U] W(S).$$
\end{Lemma}
\begin{proof} Let $T$ be a transversal of $U$ in $G$. Then the set $S'=\{s^t : s\in S, t\in T\}$
generates $N$ as a normal subgroup of $U$ and $W(S')\leq |T|W(S)=[G:U] W(S)$ since $W$ is $G$-invariant.
This yields the assertion of the lemma.
\end{proof}

We also point out two simple consequences of Lemma~\ref{addrel} that will be needed later.

\begin{Lemma}
\label{sameimage}
Let $G$ be a profinite group and $W$ a finite virtual valuation of $G$ defined
on an open normal pro-$p$ subgroup $U$. Let $\Lambda$ and $\Delta$ be
abstract subgroups of $G$ which have the same closure in $G$.
Then for any $\eps>0$ there exists a normal subgroup $N$ of $G$
such that $def_W^{G/N}(UN/N)\ge def_W^G(U)-\eps$ and $\Lambda N/N=\Delta N/N$,
that is, $\Lambda$ and $\Delta$ have the same image in $G/N$.
\end{Lemma}

\begin{proof} Let $\eps'=\frac{\eps}{2[G:U]}$.
Let $\{a_i\}_{i\geq 1}$ (resp. $\{b_i\}_{i\geq 1}$) be a countable
dense subset of $\Lambda\cap U$ (resp. $\Delta\cap U$). Since $\Lambda$ and $\Delta$
have the same closure, for each $i\geq 1$
we can choose $b_i'\in \Delta\cap U$ and $a_i'\in \Lambda\cap U$ such that
$W(a_i (b_i')^{-1})<\frac{\eps'}{2^i}$ and
$W(b_i (a_i')^{-1})<\frac{\eps'}{2^i}$. Applying Lemma~\ref{addrel}
to the set $S=\{a_i (b_i')^{-1}, b_i (a_i')^{-1}\}_{i\geq 1}$,
we get the desired result.
\end{proof}

\begin{Lemma}\label{fgquot} Let $G$ be a profinite group, $W$ a finite virtual valuation of $G$
defined on an open normal pro-$p$ subgroup $U$ and $\Gamma$ a dense abstract subgroup of $G$.
Then for any $\eps>0$ there exists a normal subgroup $N$ of $G$ with $N\subseteq U$ such that
\begin{itemize}
\item[(i)]$def_W ^{G/N}(U/N)\ge def_W^G(U)-\eps$
\item[(ii)] $\Gamma N/N$ is finitely generated (so $G/N$ is finitely generated)
\item[(iii)] $(\Gamma\cap U) N/N$ is $p$-torsion
\end{itemize}
\end{Lemma}
\begin{proof} Let $A$ be a generating set of $U$ with $W(A)<\infty$. By Lemma~\ref{sameimage}
without loss of generality we can assume that the abstract subgroup generated by $A$
coincides with $\Gamma$.

Now let $\eps'=\frac{\eps}{2[G:U]}$. We can find a subset $S_1\subseteq A$ such that $A\setminus S_1$
is finite and $W(S_1)<\eps'$. Next we enumerate elements of $\Gamma\cap U:$ $\,g_1,g_2,\ldots,\ldots$ and
choose a subset $S_2$ of the form $\{g_i^{p^{n_i}} : i\in\dbN \}$ with $W(S_2)<\eps'$.

Let $N$ be the normal subgroup of $G$ generated by $S=S_1\cup S_2$.
Condition (i) holds by Lemma~\ref{addrel}, (ii) holds since $N$ contains $S_1$ and
(iii) holds since $N$ contains $S_2$.
\end{proof}

\subsection{Profinite groups of positive virtual weighted deficiency}

\begin{Definition}\rm Let $G$ be a virtually pro-$p$ group. We will say that $G$ has
\emph{positive virtual weighted deficiency} if there exist an open normal pro-$p$ subgroup
$U$ of $G$ and a $G$-invariant valuation $W$ on $U$ such that $def_W^G(U)>0$.
\end{Definition}

Groups of positive virtual weighted deficiency (PVWD) will appear naturally
in the analysis of quotients of PWD groups (see Section~5), and their
consideration is necessary for the proof of our main results about
PWD groups. Moreover, it seems that all the interesting properties
of PWD groups extend to PVWD groups. At least this is true for
the results formulated in the introduction; in fact, in Section~8 we will
restate and prove most of these results for PVWD groups (this requires almost no extra work).

We point out that having PVWD is a stronger condition than being virtually of PWD
(that is, having an open subgroup of PWD). For instance, let $F$ be a non-abelian
free pro-$p$ group and $G$ the wreath product of $F$ and $\dbZ/n\dbZ$ (with $\dbZ/n\dbZ$
being the active subgroup). Then $G$ has an open subgroup
$\underbrace{F\times\ldots\times F}_{n\mbox{ times }}$ which is clearly of PWD.
On the other hand, $G$ does not have PVWD --
this will follow from Proposition~\ref{branch_pwd}(a).

\section{Key step}

The following theorem is the key step in the proof of Theorem~\ref{zeroone}.
It can be thought of as a pro-$p$ analogue of Theorem~\ref{subgroup_main}
with some technicalities added.

\begin{Theorem}
\label{thm:key}
Let $G$ be a profinite group and $W$ a finite virtual valuation of $G$ defined on an open
normal pro-$p$ subgroup $U$. Assume that $def_W^G(U)>0$. Let $H$ be a   closed subgroup of $G$.  Then   the following hold:
\begin{itemize}
\item[(a)] If $[U:(U\cap H)]_W<\infty$, then there exists a normal subgroup $N$ of $G$ such that $def_W^{G/N}(UN/N)>0$ and $HN/N$ is open in $G/N$.

\

\item[(b)] If  $[U:(U\cap H)]_W=\infty$ and $rk_W(U\cap H)<\infty$, then there exists an open subgroup $V$ of $U$
which is normal in $G$, a finite $G$-invariant valuation $W'$ on $V$ and a normal subgroup $N$ of $G$
such that
$def_{W'}^{G/N}(VN/N)>0$ and $HN/N$ is finite.
\end{itemize}
\end{Theorem}

The two parts of Theorem~\ref{thm:key} will be established using rather different arguments. The argument in part (b)
is more involved, but the key idea behind it is very old and goes back to the original paper of Golod~\cite{Go}.
We present it as a separate lemma.

\begin{Lemma}
\label{Golod}
Let $G$ be a profinite  group and $W$ a finite virtual valuation of $G$ defined on an open normal pro-$p$
subgroup $U$. Let $H$ be a closed subgroup of $U$ with $rk_W(H)<1$. Then for any $\eps>0$
there exists a normal subgroup $N$ of $G$ such that  $def_W^{G/N}(U/N)\geq def_W^G(U)-\eps$ and $HN/N$ is finite.
\end{Lemma}
\begin{proof} Let $Y$ be a $W$-optimal generating set for $H$, so that $W(Y)=rk_W(H)<1$.

{\it Case 1: $H$ is finitely generated.} For $m\in\dbN$ let
$$Y^{(m)}=\{[y_{i_1},\ldots,y_{i_m}] : y_{i_j}\in Y\}\cup \{ y^{p^m} : y\in Y\}$$ be the set consisting of all left-normed commutators of degree $m$ in $Y$
and all $p^m$-powers of elements of $Y$.
Clearly, $W(Y^{(m)})\leq W(Y)^m + |Y|\delta^m\to 0$ as $m\to \infty$ (where $\delta=\max \{W(y): y\in Y\}$).

Choose $m$ for which $W(Y^{(m)})<\frac{\eps}{[G:U]}$,
and let $N$ be the normal subgroup of $G$ generated by  $Y^{(m)}$.  Then by Lemma \ref{addrel}, $$def_{W}^{G/N}(U/N)\ge def_{W}^G(U)-[G:U] W(Y^{(m)})\geq def_W^G(U)-\eps.$$ On the other hand,
$HN/N$ is nilpotent and generated by a finite set of torsion elements, hence finite.

{\it General case:} Fix $0<\eps'<\frac{\eps}{[G:U]}$, and write $Y$ as disjoint union $Y_1\sqcup Y_2$
where $Y_1$ is finite and $W(Y_2)<\eps'$. Let $K$ be the normal subgroup of $G$ generated by $Y_2$.
Then $def_{W}^{G/K}(U/K)> def_{W}^{G}(U)-\eps$ (again by Lemma~\ref{addrel}), while
$rk_W(HK/K)\leq rk_W(H)<1$ and $HK/K$ is finitely generated. Thus, we are reduced to Case~1.
\end{proof}
\begin{Remark}In Section~8 we will need the following generalization of Lemma~\ref{Golod}, which
can be proved by the same argument.
Let $G,W$ and $U$ be as above, and let $\{H_i\}$ be a countable collection of closed subgroups of $U$ with
$rk_W(H_i)<1$ for each $i$. Then for any $\eps>0$
there exists a normal subgroup $N$ of $G$ such that  $def_W^{G/N}(U/N)\geq def_W^G(U)-\eps$ and $H_i N/N$ is finite
for each $i$.
\end{Remark}
\vskip .2cm

\begin{proof}[Proof of Theorem~\ref{thm:key}] Replacing  $H$ by $H\cap U$ if needed, we can assume
without loss of generality that $H\subseteq  U$.

\

{\it Case (a):} $[U:H]_W< \infty.$

\

Let $\eps=def_W^G(U)/[G:U]$. Since $[U:  H]_W<\infty$, by Lemma \ref{limits} we can find an open subgroup $V$ of $U$ containing $  H$ such that $\log_2([V:  H]_W)<\eps$.
Since $\log_2(1+\alpha)> \alpha$ for any $\alpha\in (0,1)$, we get
$$\sum_\alpha \alpha(c_\alpha(V)-c_\alpha(H))< \eps.$$ By Lemma~\ref{prelim1},
$|V_\alpha H/V_{<\alpha} H|=p^{c_\alpha(V)-c_\alpha(H)}$ for each $\alpha\in \im W$. Take  a subset  $T_\alpha$ of $ V_\alpha$ whose image forms a basis in $V_\alpha H/V_{<\alpha} H$, and put
$T=\cup T_\alpha$. Then $V=\la H, T\ra$ and $W(T)<\eps$.
Let $N$ be the normal subgroup of $G$ generated by $T$. By Lemma \ref{addrel}, $$def_W^{G/N}(U/N)\ge def_W^G(U)-[G:U]W(T)>0.$$ On the other hand, $HN/N=VN/N$
by construction and so $HN/N$ is open in $G/N$.

\

{\it Case (b):} $[U:H]_W=\infty$.

\

 Let $(F,\pi, \widetilde W)$ be a presentation of $(G,W)$ such that $def_{\widetilde W}(\pi^{-1}(U),\Ker \pi)>0$.  Since $[U:  H]_W=\infty$, by Proposition~\ref{index}(a) there exists $\alpha\in \im W$ such that   $$[ H:H_\alpha]_W< \frac{def_{\widetilde W}(\pi^{-1}(U),\Ker \pi)[U:U_\alpha]_W}{rk_W(H)-1}.$$ Then Theorem \ref{ineq} implies  that $$rk_W(H_\alpha)-1\le  [H:H_\alpha]_W(rk_W(H)-1)< def_{\widetilde W}(\pi^{-1}(U),\Ker \pi)[U:U_\alpha]_W.$$  Thus, by Proposition~\ref{prop:descent1} and Lemma~\ref{index_stupid},
  $$rk_W(H_\alpha)< def_{\widetilde W}(\pi^{-1}(U_\alpha),\Ker \pi)+1.$$

If $rk_W(H_\alpha)<1$, we are done by Lemma~\ref{Golod}, so let us assume
that $rk_W(H_\alpha)\geq 1$. Then we can choose $c>1$ with $$rk_W(H_\alpha)<c<def_{\widetilde W}(\pi^{-1}(U_\alpha),\Ker \pi)+1,$$ and let
$\widetilde W'$ be the weight function on $\pi^{-1}(U_\alpha)$ obtained from $\widetilde W$
by the $c$-contraction. Then
$$rk_{\widetilde W'}(\pi^{-1}(U_\alpha))=rk_{\widetilde W}(\pi^{-1}(U_\alpha))/c \textrm{\ and\ } \widetilde W'(f)\leq \widetilde W(f)/c \textrm{\ for any \ } f\in \pi^{-1}(U_\alpha), \eqno (*)$$
whence
$$def_{\widetilde W'}(\pi^{-1}(U_\alpha),\Ker \pi)+1\geq (def_{\widetilde W}(\pi^{-1}(U_\alpha),\Ker \pi)+1)/c. \eqno (**)$$
It is easy to see that $\pi^{-1}(U_\alpha)=\widetilde U_{\alpha}\Ker\pi$ where
$\widetilde U_{\alpha}=\{f\in \pi^{-1}(U): \widetilde W(f)\leq \alpha\}$. Since
$\widetilde W$ is $F$-invariant, $\widetilde U_{\alpha}$ is normal in $F$, whence
$\pi^{-1}(U_\alpha)$ is also normal in $F$. Thus we can apply Lemma \ref{contr-inv}
and deduce that $\widetilde W'$ is $F$-invariant.

Let $W'$ be the virtual valuation of $G$ induced by $\widetilde W'$.
Then (**) implies that
$$def_{W'}^G(U_\alpha) \geq def_{\widetilde W'}(\pi^{-1}(U_\alpha),\Ker \pi)\ge (def_{\widetilde W}(\pi^{-1}(U_\alpha),\Ker \pi)+1)/c-1>0 $$  while (*) yields
$$rk_{W'}(H_\alpha)\leq rk_{W}(H_\alpha)/c<1.$$
We can now finish the proof by applying Lemma~\ref{Golod} to the virtual valuation $W'$.
\end{proof}

\section{Weighted deficiency of abstract groups}

In this section we define weighted deficiency for abstract groups
and explain how results about weighted deficiency of pro-$p$ groups
obtained in the previous sections can be applied to abstract groups.

\begin{Definition}\rm Let $\Gamma$ be a finitely generated abstract group.
\begin{itemize}
\item[(i)] For each prime $p$ the
{\it weighted $p$-deficiency} of $\Gamma$ is the quantity $wdef(\Gamma_{\phat})$,
the weighted deficiency of the pro-$p$ completion of $\Gamma$.

\item[(ii)] The {\it weighted deficiency} of $\Gamma$ is the supremum of its
weighted $p$-deficiencies over all primes $p$.
\end{itemize}
\end{Definition}

The reader may find this definition slightly unsatisfactory for two reasons.
First, we indeed have to require that $\Gamma$ is finitely generated for
otherwise the pro-$p$ completion $\Gamma_{\phat}$ need not be countably based
(and therefore falls outside of our considerations). Second, it may be desirable
to define the weighted deficiency of an abstract group $\Gamma$ using only abstract
presentations of $\Gamma$ (not pro-$p$ presentations). Both problems
would be resolved if instead of $wdef(\Gamma_{\phat})$ we considered the related quantity 
$wdef_p(\Gamma)$ defined as follows:

\begin{Definition}\rm Let $\Gamma$ be an abstract group.
For each prime $p$ define $wdef_p(\Gamma)$ to be the supremum of the set
$\{wdef_p(X,R)\}$ where $(X,R)$ runs over all {\it abstract} presentations of $\Gamma$, 
that is, presentations where $R$ lies in the abstract group generated by $X$. (Recall that
the quantity $wdef_p(X,R)$ is defined at the beginning of \S~4.1.)
\end{Definition}

However, there are important technical advantages in defining the weighted $p$-deficiency 
of a finitely generated abstract group $\Gamma$ to be $wdef(\Gamma_{\phat})$ (and not $wdef_p(\Gamma)$).
For instance, the former quantity does not change when $\Gamma$ is replaced by its image in the pro-$p$ completion.
\vskip .2cm

We will not use the quantity $wdef_p(\Gamma)$ in this paper, but for completeness
we will briefly discuss its relationship with $wdef(\Gamma_{\phat})$.
It is clear that $wdef_p(\Gamma)\leq wdef(\Gamma_{\phat})$, but we do not
know if the opposite inequality holds. Moreover, we do not know if
$wdef_p(\Gamma)>0$ whenever $wdef(\Gamma_{\phat})>0$. However,
the following weaker statement holds:

\begin{Proposition}
\label{wdefp}
Let $\Gamma$ be a finitely generated abstract group
of positive weighted $p$-deficiency. Then $\Gamma$ has a quotient $\Gamma'$
with $wdef_p(\Gamma')>0$.
\end{Proposition}

To prove Proposition~\ref{wdefp} (which will not be needed in the sequel) it suffices to show that 
given a presentation $\la X|R \ra$ of a pro-$p$ group $G$, a weight function $W$ on $F(X)$ with respect to 
$X$ and $\eps>0$, one can find a presentation $\la X|R' \ra$ for some quotient of $G$ such that
$R'$ lies in the abstract group on $X$ and $W(R')<W(R)+\eps$.
The latter is an easy consequence of the following observation, which will also be used
for other purposes.

\begin{Observation}
\label{discretization}
Let $G$ be a countably based pro-$p$ group, $W$ a pseudo-valuation on $G$,
$\,\Gamma$ a dense abstract subgroup of $G$, and $\eps>0$ a real number.
Then any $g\in G$ can be written
as an infinite product $g=\prod_{i=0}^{\infty} g_i$ s.t.
\begin{align*}
&\mbox{\rm (a) } \mbox{each } g_i\in \Gamma &
&\mbox{\rm (b) } \sum_{i=1}^{\infty} W(g_i)\le \eps&
&\mbox{\rm (c) } W(g_0)=W(g).&
\end{align*}
\end{Observation}
\begin{Remark} Recall that a pseudo-valuation is defined in the same
way as a valuation except that we may have $W(g)=0$ for $g\neq 1$.
\end{Remark}
\begin{proof} 
First suppose that $W(g)\neq 0$,
and let $\delta=\min\{W(g), \eps/2\}$. Since $W$ is continuous
and $G$ is countably based, there is a descending chain $\{G_i\}_{i=0}^{\infty}$ 
of open subgroups of $G$ which form a base of neighborhoods of 
identity such that $W(h)<\delta/{2^i}$ for all $h\in G_i$.

Since $\Gamma$ is dense in $G$, there is $g_0\in\Gamma$ such that $g_0^{-1}g\in G_0$,
so that $W(g_0^{-1}g)<\delta$.
This implies that $W(g_0)=W(g)$: indeed, $W(g_0)=W(g(g_0^{-1}g)^{-1})\leq 
\max\{ W(g), W(g_0^{-1}g)\}= W(g)$, and on the other hand,
$W(g)\leq \max\{W(g_0), W(g_0^{-1}g)\}$, whence $W(g)\leq W(g_0)$
(since $W(g)>W(g_0^{-1}g)$).

Next, choose inductively elements $g_i\in\Gamma$ for $i\geq 1$ such that 
$g_i^{-1} g_{i-1}^{-1}\ldots g_0^{-1}g\in G_i$ for each $i$.
By construction, $g=\prod_{i=0}^{\infty} g_i$. Since $g_i=(g_{i-1}^{-1}\ldots g_0^{-1}g)\cdot(g_{i}^{-1}\ldots g_0^{-1}g)^{-1}$, we have $W(g_i)< \max\{\delta/{2^{i-1}}, \delta/{2^i}\}=\delta/{2^{i-1}}$.
Therefore, $\sum_{i=1}^{\infty} W(g_i)\leq 2\delta\leq \eps$, so all the required conditions hold.
\vskip .12cm
Finally, in the case $W(g)=0$ we let $\delta=\eps/2$ and $g_0=1$ and repeat the above argument.
\end{proof}

Suppose now that an abstract group $\Gamma$ sits densely inside a pro-$p$ group $G$ of PWD.
If we manage to construct a quotient of $G$ with some prescribed property $(P)$,
we can consider the corresponding quotient of $\Gamma$, which will often have
a property similar to $(P)$. However, the reverse transition (obtaining quotients of $G$
from quotients of $\Gamma$) may not be possible unless we know that $G$ is the pro-$p$
completion of $\Gamma$. Our next result essentially resolves this problem.

\begin{Theorem}\label{dense} Let $G$ be a pro-$p$ group, $W$ a finite valuation of $G$
and $\Gamma$ a dense abstract subgroup of $G$. Then for
any $\eps>0$ there exists a normal subgroup $N$ of $G$ such that $def_W(G/N)\ge
def_W(G)-\eps$ and $G/N$ is naturally isomorphic to the pro-$p$ completion of $\Gamma N/N$.
\end{Theorem}

\begin{proof} First, by Lemma~\ref{fgquot} without loss of generality
we can assume that $\Gamma$ is finitely generated.
Let $\psi:\Gamma_{\hat p}\to G$ be the epimorphism induced by the embedding of $\Gamma$ into
$G$. Choose a countable set $\{r_1,r_2,\ldots\}$ of normal generators of $\Ker \psi$. Applying
Observation~\ref{discretization} to the pseudo-valuation $W\circ \psi$ on $\Gamma_{\phat}$
we deduce that there are elements $\{r_{i,j}\in \Gamma \}_{i\ge 1, j\ge 0}$ such that
\begin{itemize}
\item[(i)] $r_i=r_{i,0}r_{i,1}r_{i,2}\ldots$
\item[(ii)] $\sum_{j=1}^{\infty} W(\psi(r_{i,j}))\le \frac {\eps}{2^{i}}$ for
$i\ge 1$.
\end{itemize}
Let $N$ be the normal subgroup of $G$ generated by $S=\{\psi(r_{i,j})\}_{i,j\ge 1}$.
Then $W(S)\le \eps$ and so $def_W(G/N)\ge def_W(G)-\eps$ by Lemma \ref{addrel}.
\vskip .1cm
We claim that the natural epimorphism $(\Gamma N/N)_{\phat}\to G/N$ is an isomorphism.
Since $G/N$ is hopfian (being a finitely generated pro-$p$ group), it suffices
to construct a (continuous) epimorphism $G/N\to (\Gamma N/N)_{\phat}$. By definition of $N$
we have
$$G/N\cong \Gamma_{\hat p}/\la \{r_i, r_{i,j}\}_{i,j\ge 1}\ra^{\Gamma_{\hat p}}=
\Gamma_{\hat p}/\la\{ r_{i,j}\}_{i\ge 1,j\ge 0}\ra^{ \Gamma_{\hat p}}\cong
(\Gamma /K)_{\phat},$$
where $K=\la \{ r_{i,j}\}_{i\ge 1,j\ge 0}\ra^{ \Gamma}$. On the other hand,
$K$ is clearly contained in $\Gamma\cap N$, so $(\Gamma /K)_{\phat}$ surjects
onto $(\Gamma /\Gamma\cap N)_{\phat}\cong (\Gamma N/N)_{\phat}$.
\end{proof}

\begin{Corollary}
\label{quotpropcomp}
Let $\Gamma$ be a dense abstract subgroup of a pro-$p$ group of positive weighted deficiency.
Then $\Gamma$ has a quotient with positive weighted deficiency.
\end{Corollary}

Similarly to pro-$p$ groups, analysis of abstract groups
of positive weighted deficiency (PWD) can be easily extended to
abstract groups of positive virtual weighted deficiency (PVWD)
defined below.

\begin{Definition}\rm Let $\Gamma$ be a finitely generated abstract group,
$\Lambda$ a finite index normal subgroup of $\Gamma$ and $p$ a prime.
Consider the topology on $\Gamma$ whose base of neighborhoods of identity
consists of (all) normal subgroups of $\Lambda$ of $p$-power index.
The completion of $\Gamma$ in this topology will be called
the {\it virtual pro-$p$ completion of $\Gamma$ relative to $\Lambda$}
and  denoted by $\Gamma_{\widehat{p,\Lambda}}$.
Note that the canonical image of $\Lambda$ in $\Gamma_{\widehat{p,\Lambda}}$
is naturally isomorphic to $\Lambda_{\phat}$.
\end{Definition}

\begin{Definition}\rm Let $\Gamma$ be a finitely generated abstract group.
\begin{itemize}
\item[(a)] Given a prime $p$, we will say that $\Gamma$
has {\it positive virtual weighted $p$-deficiency}
if there is a finite index normal subgroup $\Lambda$ of $\Gamma$
such that if $G=\Gamma_{\widehat{p,\Lambda}}$ and $U=\Lambda_{\phat}$
(considered as a subgroup of $G$), then $def_W^G(U)>0$ for
some $G$-invariant valuation $W$ on $U$. Any such $\Lambda$
will be referred to as an {\it invariant PWD subgroup of $\Gamma$}.

\item[(b)] We will say
that $\Gamma$ has {\it positive virtual weighted deficiency (PVWD)}
if $\Gamma$ has  positive virtual weighted $p$-deficiency for some $p$.
\end{itemize}
\end{Definition}

Here is the ``virtual'' version of Corollary~\ref{quotpropcomp}
(whose proof is analogous):

\begin{Theorem}
\label{quotpropcomp_virtual}
Let $\Gamma$ be a dense abstract subgroup of a profinite group of PVWD.
Then $\Gamma$ has a quotient with PVWD.
\end{Theorem}

We note that there are many {\it natural} examples of PVWD abstract groups
which do not have PWD. The key to producing such examples is the following observation:

\begin{Observation}
\label{exampvirtval2}
Let $\Gamma$ be a finitely generated abstract group
and $\Lambda$ a finite index normal subgroup of $\Gamma$. Suppose that
$def_W(\Lambda_{\phat})>0$ for some uniform valuation $W$ on $\Lambda_{\phat}$.
Then $\Gamma$ has positive virtual weighted $p$-deficiency.
\end{Observation}
\begin{proof} This follows directly from Observation~\ref{exampvirtval}.
\end{proof}

\begin{Corollary}
\label{hyperbolic} The following classes of groups have positive virtual
weighted $p$-deficiency for every prime $p$:
\begin{itemize}
\item[(a)] virtually free groups (which are not virtually cyclic)
\item[(b)] fundamental groups of (orientable) hyperbolic $3$-manifolds
\end{itemize}
\end{Corollary}
\begin{proof} (a) is clear by Observation~\ref{exampvirtval2}.

(b) This fact is essentially known, but for completeness we give a proof.
Fix a prime $p$, and let $\Gamma$ be the fundamental group of a hyperbolic
$3$-manifold $M$. It is well known that $\Gamma$ has a presentation $\la X|R \ra$ with
$|X|-|R|=0$ or $1$, and the same is true for all finite index subgroups of $\Gamma$
(see e.g. \cite{Lu1}). Moreover, $\Gamma$ is linear and not virtually solvable, so \cite[Theorem~B]{Lu2}
implies that for each $n\in\dbN$ there is a finite index
subgroup $\Lambda$ of $\Gamma$ such that $d(\Lambda_{\phat})\geq n$.

Let $G=\Lambda_{\phat}$ and $d=d(G)$. By \cite[Lemma~1.1]{Lu1}, $G$ has a pro-$p$ presentation
$\la X'|R'\ra$ where $|X'|-|R'|=0$ or $1$, $|X'|=d$ and all elements of $R'$ lie
in the Frattini subgroup. Now let $\widetilde W$ be the
uniform weight function on $F(X')$ such that $\widetilde W(x)=\frac{1}{2}$ for all $x\in X'$,
and let $W$ be the induced valuation on $G$. Then $def_W(G)\geq \widetilde W(X')-\widetilde W(R')-1\geq d(\frac{1}{2}-\frac{1}{4})-1=\frac{d}{4}-1$,
so in particular $def_W(G)>0$ if $d\geq 5$.

If $\Lambda$ is normal in $\Gamma$, we are done by Observation~\ref{exampvirtval2}.
In the general case we proceed as follows. Assume (as we may) that $d\geq 20$,
so that $def_W(G)\geq 4$, and let $\Delta$ be a finite index subgroup of $\Lambda$ which is normal in $\Gamma$.
Let $H$ be the closure of the image of $\Delta$ in $G$. Then
$def_W(H)\geq def_W(G)\geq 4$ by Proposition~\ref{def:eq}, whence
$d(H)\geq def_W(H)+1\geq 5$ by Lemma~\ref{trivial}(b).
Since $\Delta_{\phat}$ surjects onto $H$, we have $d(\Delta_{\phat})\geq 5$.
We can now finish the proof by applying the earlier argument to $\Delta$ instead of $\Lambda$.
\end{proof}

\section{LERF quotients of PWD groups}

\subsection{Property LERF}
\begin{Definition}\rm A finitely generated abstract group $\Gamma$
is said to be {\it LERF} if every finitely generated subgroup of $\Gamma$
is closed in the profinite topology.
\end{Definition}

Historically, the groups first shown to be LERF were
the free groups -- this follows from a classical result of Hall \cite{Hall2}.
However, property LERF was not formally introduced and studied
until very recently when it naturally arose in several problems
in geometric group theory and 3-manifold topology.
We refer the reader to \cite{LLR} and \cite{LR} for a discussion of some geometric
applications of LERF.

 Many important classes of groups do not have LERF -- for instance,
if $\Gamma$ is any group, which has a proper finitely generated
subgroup that is dense in the profinite topology on $\Gamma$,
then $\Gamma$ cannot be LERF, and the former property holds, for instance,
in any arithmetic group with the congruence subgroup property. On the contrary,
proving that a given group does have LERF is usually difficult,
and the list of known examples of LERF groups is not very long.

In this section we prove Theorem~\ref{PWD_LERF} whose statement
(in a slightly generalized form) is recalled below as Theorem~\ref{LERF2}.
It produces a large class of LERF groups constructed as quotients of PWD groups.

\begin{Theorem}
\label{LERF2}
Let $\Gamma$ be a group of positive weighted $p$-deficiency.
Then $\Gamma$ has a quotient which is LERF and $p$-torsion
and still has positive weighted $p$-deficiency.
\end{Theorem}

While Theorem~\ref{LERF2} probably cannot be used to prove LERF
for any naturally defined group, it does have interesting applications.
For instance, it yields an answer to a question of Long and Reid about
the existence of LERF groups with property $(T)$ (see \cite[Question~4.5]{LR}):

\begin{Corollary}
\label{longreid}
There exists a LERF group with property $(T)$.
\end{Corollary}

Corollary~\ref{longreid} is a direct consequence of Theorem~\ref{LERF2}
and the existence of Kazhdan groups of positive weighted deficiency
established in \cite{Er}.

\subsection{Properties $p$-LERF and $p$-WLERF and a useful corollary of Theorem~\ref{LERF2}}

In this subsection we establish a simple consequence of Theorem~\ref{LERF2}
(see Corollary~\ref{LERF1} below) which will be needed for the proofs
of Corollary~\ref{HJI} and Theorem~\ref{justinfexp}.

\begin{Definition}\rm Let $\Gamma$ be a finitely generated group.
We will say that $\Gamma$ is {\it $p$-LERF} if every finitely
generated subgroup of $\Gamma$ is closed in the pro-$p$ topology.
\end{Definition}

In general being $p$-LERF is stronger than being LERF; however,
a $p$-torsion group which is LERF is automatically $p$-LERF
(since the profinite topology on a $p$-torsion group coincides with
the pro-$p$ topology). Thus,
groups in the conclusion of Theorem~\ref{LERF2} have $p$-LERF.

Neither of the properties LERF and $p$-LERF is preserved by quotients;
however, there is a weaker version of $p$-LERF which is preserved
by quotients.

\begin{Definition}\rm Let $\Gamma$ be a finitely generated group.
We will say that $\Gamma$ is {\it weakly $p$-LERF} (or $p$-WLERF) if
for any infinite index subgroup $\Lambda$ of $\Gamma$ its pro-$p$ closure
$\overline{\Lambda}$ is also of infinite index in $\Gamma$ (note that $\Lambda$
is not assumed to be finitely generated).
\end{Definition}

\begin{Lemma} \label{WLERF}
The following hold:
\begin{itemize}
\item[(a)] The property $p$-WLERF is inherited by quotients
\item[(b)] Every $p$-LERF group is $p$-WLERF
\item[(c)] An infinite $p$-WLERF group has infinite pro-$p$ completion.
\end{itemize}
\end{Lemma}
\begin{proof} (a) Let $\pi:\Gamma\to \Gamma'$ be an epimorphism, and assume that
$\Gamma$ is $p$-WLERF. Let $\Lambda$ be an infinite index subgroup of $\Gamma'$.
Then $\pi^{-1}(\Lambda)$ is an infinite index subgroup of $\Gamma$, so by assumption
there exists an infinite strictly descending chain $H_1\supset H_2 \supset\ldots$
of open (in the pro-$p$ topology) subgroups of $\Gamma$ with $\pi^{-1}(\Lambda)\subset H_i$
for all $i$.  Note that $\pi$ maps open subgroups to open subgroups (since
open subgroups are precisely subnormal subgroups of $p$-power index), so
$\pi(H_1)\supset \pi(H_2)\supset\ldots$ is an infinite strictly descending chain
of open subgroups of $\Gamma'$ containing $\Lambda$. Thus, the pro-$p$ closure of $\Lambda$
has infinite index in $\Gamma'$, and we have shown that $\Gamma'$ is $p$-WLERF.

(b) Let $\Gamma$ be $p$-LERF and $\Lambda$ a subgroup of $\Gamma$ whose pro-$p$
closure $\overline\Lambda$ is of finite index in $\Gamma$. Then the group
$\overline\Lambda/[\overline\Lambda,\overline\Lambda]\overline\Lambda^p$ is finite, so there exists a finitely generated subgroup $\Lambda'$ of
$\Lambda$ which surjects onto $\overline\Lambda/[\overline\Lambda,\overline\Lambda]\overline\Lambda^p$. By a standard pro-$p$ argument the latter implies that $\overline\Lambda=\overline{\Lambda'}$. Since $\Lambda'$ is finitely generated and $\Gamma$ is $p$-LERF, we must have $\overline{\Lambda'}=\Lambda'$.
Hence $\overline\Lambda=\overline{\Lambda'}=\Lambda'\subseteq \Lambda$,
and thus $\Lambda=\overline\Lambda$ must be of finite index.

(c) follows by applying the definition of $p$-WLERF to the trivial subgroup.
\end{proof}

\begin{Corollary}
\label{LERF1} Let $\Gamma$ be a $p$-LERF group. Then any
infinite quotient of $\Gamma$ has infinite pro-$p$ completion.
\end{Corollary}

Combining Theorem~\ref{LERF2} and Corollary~\ref{LERF1} we can now prove
Theorem~\ref{Kazhdan_resfinite}:

\begin{proof}[Proof of Theorem~\ref{Kazhdan_resfinite}] Let $\Gamma$ be an abstract
group of positive weighted $p$-deficiency. By Theorem~\ref{LERF2} $\Gamma$
has a $p$-LERF quotient $\Gamma'$ which still has positive weighted $p$-deficiency.
By \cite[Theorem~4.6]{EJ} $\Gamma'$ has an infinite quotient $\Omega$ with $(T)$,
and by Corollary~\ref{LERF1} $\Omega$ has infinite pro-$p$ completion. Replacing
$\Omega$ by its image in the pro-$p$ completion, we obtain an infinite residually-$p$
group with $(T)$ which is a quotient of $\Gamma$.
\end{proof}

\begin{Remark} In \cite{Wil5}, Wilson constructed finitely generated residually
finite $p$-torsion groups which are $\widetilde N$-groups and all of whose subgroups
are $\overline Z$-groups. It is easy to see that groups with these properties
have $p$-WLERF.
\end{Remark}

\subsection{Proof of Theorem~\ref{LERF2}}

\begin{Lemma} 
\label{trivial4}
Let $\Gamma$ be a group of positive weighted $p$-deficiency. Then
$\Gamma$ has a $p$-torsion residually-$p$ quotient of positive weighted $p$-deficiency.
\end{Lemma}
\begin{proof} First, replacing $\Gamma$ by its image in its pro-$p$ completion,
we can assume that $\Gamma$ is residually-$p$. By Lemma~\ref{fgquot} applied with $G=U=\Gamma_{\phat}$,
there is a quotient $\Gamma'$ of $\Gamma$ which is $p$-torsion and sits
densely inside a PWD pro-$p$ group. By Corollary~\ref{quotpropcomp}, $\Gamma'$
has a quotient $\Gamma''$ of positive weighted $p$-deficiency. Finally, replacing
$\Gamma''$ by the image in its pro-$p$ completion, we can assume that $\Gamma''$ is residually-$p$.
\end{proof}

\begin{proof}[Proof of Theorem~\ref{LERF2}] 
Let $G=\Gamma_{\phat}$ be the pro-$p$ completion of $\Gamma$, so by assumption $def_W(G)>0$ for some finite valuation $W$.
By Lemma~\ref{trivial4}, we can assume that $\Gamma$ is a subgroup of $G$ and $\Gamma$ is $p$-torsion.
The latter implies that the pro-$p$ topology and profinite topology on $\Gamma$
(and any of its quotients) coincide.

We start by reformulating the assertion of the theorem in a more convenient language.
Let $\mathcal P$ be the set of pairs $(g,\Lambda)$ where $g\in \Gamma$ and $\Lambda$
is a finitely generated subgroup of $\Gamma$.

If $H$ is any pro-$p$ group defined as a quotient of $G$, let
$\pi_H:G\to H$ be the natural projection and $N_H=\Ker\pi_H$. Given a pair
$(g,\Lambda)\in \mathcal P$, we will say that
\begin{itemize}
\item[] $(g,\Lambda)$ is of {\it type I for $H$} if $\pi_H(g)$ lies in $\pi_H(\Lambda)$
\item[] $(g,\Lambda)$ is of {\it type II for $H$} if $\pi_H(g)$ does not lie
in the closure of $\pi_H(\Lambda)$
\item[] $(g,\Lambda)$ is of {\it type III for $H$} if $\pi_H(g)$ does not
lie in $\pi_H(\Lambda)$, but lies in its closure.
\end{itemize}

To prove Theorem~\ref{LERF2} it will be enough to establish the following claim:

\begin{Claim}
\label{claim:LERF} The group $G$ has a normal subgroup $N$ s.t.
\begin{itemize}
\item[(1)] $def_W(G/N)>0$
\item[(2)] $N$ is normally generated by a subset of $\Gamma$
\item[(3)] all pairs in $\mathcal P$ are of type I or II for $G/N$
\end{itemize}
\end{Claim}

Indeed, since $G$ is the pro-$p$ completion of $\Gamma$, condition (2) implies
that $G/N$ is the pro-$p$ completion of $\Gamma N/N$. Thus, $\Gamma N/N$
has PWD by (1), and the pro-$p$ topology on $\Gamma N/N$ is induced from $G/N$.
Therefore, $\Gamma N/N$ has LERF by (3).
\vskip .12cm

If all pairs in $\mathcal P$ are of type I or II for $G$, there is nothing to prove.
Here is the basic technique for eliminating pairs of type III.
If a pair $(g,\Lambda)$ is of type III for some pro-$p$ group $H$, then for any $\eps>0$
we can write $g=lnr$ where $l\in\Lambda$, $n\in N_H$ and $W(r)<\eps$.
If we now impose the relation $r=1$, that is, consider the group
$$H'=H/\la \pi_H(r)\ra^H\cong G/\la N_H, r \ra^G,$$
then $\pi_{H'}(g)=\pi_{H'}(l)$, and so $(g,\Lambda)$ is of type $I$ for $H'$.

Since the set $\mathcal P$ is countable,
this technique enables us to replace the group $G$ with its quotient $G'$
which still has positive weighted deficiency and such that all type III
pairs for $G$ have type I for $G'$. The problem is that this process
may create new type III pairs, that is, some of type II pairs for $G$
may be of type III for $G'$. Another problem is that the defining relations
in the obtained presentation of $G'$ do not necessarily come from $\Gamma$,
and thus $G'$ may not coincide with the pro-$p$ completion of the
abstract group $\pi_{G'}(\Gamma)$. To avoid these problems
we proceed slightly differently.
\vskip .12cm
First, we need one more notation. Given a pro-$p$ quotient $H$ of $G$,
define the pseudo-distance function $d_H$ on $G$ by
$$d_H(a,b)=\inf\{W(g): g\in G,\,\pi_H(g)=\pi_H(a^{-1}b)\}.$$ Given subsets
$A$ and $B$ of $G$, we put $d_H(A,B)=\inf\{d_H(a,b): a\in A, b\in B\}$.
Note that a pair $(g,\Lambda)$ is of type III for $H$ if and only if
$\pi_H(g)\not\in \pi_H(\Lambda)$ and $d_H(g,\Lambda)=0$.

Fix $\eps<def_W(G)$. We shall construct a sequence of pro-$p$ groups
and epimorphisms $G=G_0\to G_1\to G_2\to \ldots$ with the following properties:
\begin{itemize}
\item[(a)] $G_{i+1}=G_{i}/\la \pi_{G_{i}}(R_{i})\ra^{G_{i}}$ where $R_{i}$ is a subset of $\Gamma$
\item[(b)] $W(R_{i})<\frac{\eps}{2^{i+1}}$ (in particular, $W(\cup_i R_i)<\eps$)
\item[(c)] For each pair $(g,\Lambda)\in\mathcal P$ one of the following holds:
\begin{itemize}
\item[(i)] There exists $n$ such that $(g,\Lambda)$ is of type I for $G_n$
(hence also of type I for $G_i$ for all $i\geq n$)
\item[(ii)] $(g,\Lambda)$ is of type II for $G_n$ for all $n$, and
$\inf\limits_n d_{G_n}(g,\Lambda)>0$.
\end{itemize}
\end{itemize}
We now describe the construction.  First we enumerate all pairs in $\mathcal P$:
$(g_1,\Lambda_1), (g_2,\Lambda_2),\ldots$. Fix $n\geq 0$, and suppose
we have constructed pro-$p$ groups $G_0,G_1,\ldots, G_n$ and their presentations
such that
\begin{itemize}
\item[(1)] (a) and (b) hold for all $0\leq i< n$,
\item[(2)] each of the pairs $(g_1,\Lambda_1),\ldots, (g_n,\Lambda_n)$ is of type
I or II for $G_n$.
\end{itemize}
We now construct $G_{n+1}$ so that (1) and (2) still hold with $n$ replaced by $n+1$.
If $(g_{n+1},\Lambda_{n+1})$ is of type I or II for $G_n$, we set $R_{n}=\emptyset$
and $G_{n+1}=G_n$. If $(g_{n+1},\Lambda_{n+1})$ is of type III for $G_n$, choose
$0<\delta<\frac{\eps}{2^{n+1}}$ such that
\begin{itemize}
\item[(3)] $\delta< d_{G_n}(g_i,\Lambda_i)$ for all $i\leq n$
such that $(g_i,\Lambda_i)$ has type II for $G_n$.
\end{itemize}
Since $(g_{n+1},\Lambda_{n+1})$ is of type III for $G_n$, we can
write $g_{n+1}=lhr$ where $l\in \Lambda_{n+1}$, $h\in N_{G_n}$ and $W(r)<\delta/2$.
By Observation~\ref{discretization} we can write $r$ as an infinite converging product
$r=\prod_{j=1}^{\infty} r_j$ with $r_j\in \Gamma$ such that $\sum\limits_j W(r_j)<\delta$.
Let $R_{n}= \{r_j:j\ge 1\}$ and let $G_{n+1}=G_n/\la \pi_{G_n}( R_n)\ra^{G_n}$.

Conditions (a) and (b) hold for $i=n$ by construction. Also by construction,
the pair $(g_{n+1},\Lambda_{n+1})$ has type I for $G_{n+1}$.
Condition (3) ensures that each of the pairs $(g_i,\Lambda_i)$ with $i\leq n$ which had type II for $G_n$
will also have type II for $G_{n+1}$, and moreover,
$d_{G_n}(g_i,\Lambda_i)=d_{G_{n+1}}(g_i,\Lambda_i)$
for each such pair. The latter ensures that the obtained sequence of groups $\{G_n\}$
satisfies condition (c).

Thus, we have constructed a sequence of groups $\{G_n\}$ satisfying (a)-(c).
Let $N=\overline{\cup_{n=1}^{\infty} N_{G_n}}$. By condition (a) $N$ is normally generated
by a subset of $\Gamma$, condition (b) implies that $def_W(G/N)>0$ by Lemma~\ref{addrel}, and (c) ensures that all pairs in $\mathcal P$ are of type I or II for
$G/N$. Thus, we proved Claim~\ref{claim:LERF} and hence also proved Theorem~\ref{LERF2}.
\end{proof}

We finish this section with the ``virtual version''
of Theorem~\ref{LERF2}:

\begin{Theorem}
\label{LERF2_virtual}
Let $\Gamma$ be an abstract group of positive virtual weighted $p$-deficiency
with invariant PWD subgroup $\Lambda$. Then there is an epimorphism $\pi:\Gamma\to\Gamma'$
s.t. $\Gamma'$ has positive virtual weighted $p$-deficiency with invariant PWD subgroup
$\pi(\Lambda)$, and $\pi(\Lambda)$ is $p$-LERF.
\end{Theorem}
\begin{proof} The proof is similar to that of Theorem~\ref{LERF2}, apart from a few
minor changes described below.

Let $G=\Gamma_{\widehat{p,\Lambda}}$ be the virtual pro-$p$ completion
of $\Gamma$ relative to $\Lambda$ and $U=\Lambda_{\phat}$ (thought of
as a subgroup of $G$. By definition $def_W^G(U)>0$ for some $G$-invariant valuation
$W$ on $U$. As usual, we can assume without loss of generality that $\Lambda$
is $p$-torsion.

Essentially repeating the proof of Claim~\ref{claim:LERF}, we  construct
a closed normal subgroup $N$ of $G$, with $N\subset U$, such that
$def_W^{G/N}(U/N)>0$, $U/N\cong (\Lambda N/N)_{\phat}$ and $\Lambda N/N$
is $p$-LERF. Now let $\Gamma'=\Gamma N/N$. It is easy to check that
$G/N$ is the virtual pro-$p$ completion of $\Gamma'$ relative to $\Lambda N/N$.
Thus, by definition of positive virtual weighted $p$-deficiency for abstract groups
$\Gamma'$ has the required property.
\end{proof}

\section{Proofs of the main results}

In this section we prove Theorem~\ref{zeroone}, Corollary~\ref{HJI}
and Theorem~\ref{HJI_prop} generalized to groups of positive virtual weighted deficiency.

\subsection{Constructing locally zero-one groups}

We start with a simple corollary of Lemma~\ref{sameimage}.

\begin{Corollary}
\label{LERF_substitute}
Let $G$ be a profinite group and $W$ a finite virtual valuation of $G$ defined
on an open normal pro-$p$ subgroup $U$. Let $\Gamma$ be a dense abstract subgroup
of $G$ and $\Lambda$ a subgroup of $\Gamma$ whose closure
$\overline\Lambda$ is open in $G$. Then for any $\eps>0$
there exists a normal subgroup $N$ of $G$ such that if $\pi:G\to G/N$
is the natural projection, then
\begin{itemize}
\item[(i)] $def_W^{\pi(G)}(\pi(U))\geq def_W^{G}(U)-\eps$
\item[(ii)] $\pi(\Lambda)=\pi(\overline\Lambda\cap\Gamma)$, so in particular
$\pi(\Lambda)$ is of finite index in $\pi(\Gamma)$.
\end{itemize}
\end{Corollary}
\begin{Remark}
Note that Theorem~\ref{thm:key} and Corollary~\ref{LERF_substitute}
immediately imply Theorem~\ref{subgroup_main} (in fact, a stronger version of it).
\end{Remark}
\vskip .12cm
We now prove Theorem~\ref{zeroone} using iterated applications of
Theorem~\ref{thm:key} and Corollary~\ref{LERF_substitute}.

\begin{Theorem}[generalization of Theorem~\ref{zeroone}]
\label{zeroone_prop}
Let $G$ be a virtually  pro-$p$ group of positive virtual weighted deficiency and
$\Gamma$ a dense abstract subgroup of $G$. Then $G$ has a quotient $H$
such that the image $\Omega$ of $\Gamma$ in $H$ is a locally zero-one group.
\end{Theorem}
\begin{Remark} Note that $\Omega$ is automatically residually finite since $H$ is profinite.
\end{Remark}
\vskip .12cm
\begin{proof}
First, by Lemma~\ref{fgquot}, we can assume that $\Gamma$ is finitely generated
and virtually $p$-torsion.
Let $\{\Lambda_i\}_{i=1}^{\infty}$ be the set of finitely generated subgroups of $\Gamma$
(ordered arbitrarily). By Theorem~\ref{thm:key}
we can construct a sequence of pro-$p$ groups $G=G_0\to G_1\to G_2\to \ldots$ such that

\begin{itemize}
\item[(1)] $G_{i+1}$ is a quotient of $G_i$ for each $i$
\item[(2)] Each $G_i$ has positive virtual weighted deficiency
\item[(3)] If $\pi_i:G_0\to G_i$ is the natural projection, then
either  $\pi_i(\Lambda_i)$ is finite or the closure of
$\pi_i(\Lambda_i)$ in $G_i$ is open.
\end{itemize}

Corollary~\ref{LERF_substitute} enables us to replace condition (3)
by its stronger version (3') below while preserving (1) and (2):

\begin{itemize}
\item[(3')] $\pi_i(\Lambda_i)$ is either finite or has finite index
in $\pi_i(\Gamma)$.
\end{itemize}

Now let $N_i=\Ker\pi_i$, let $N_{\infty}=\overline{\cup_{i=1}^{\infty}N_i}$ and
$G_{\infty}=G_0/N_{\infty}$. Since each of the groups $G_i$ is infinite by (2)
and $G_0$ is finitely generated, $G_{\infty}$ must also be infinite.
Let $\pi_{\infty}:G_0\to G_{\infty}$ be the natural projection and $\Omega=\pi_{\infty}(\Gamma)$.
By construction $\Omega$ is an infinite torsion group (in particular, it is not virtually cyclic),
and condition (3') implies that every finitely generated subgroup of $\Omega$ is finite
or of finite index. Thus $\Omega$ is a locally zero-one group.
\end{proof}

We finish this subsection with a brief discussion of the problem of existence of residually finite
globally zero-one groups mentioned in the introduction. This is a well known question
(see e.g. \cite[Question~1]{SW} and \cite[Problem~12.43]{Kou}), which appears to be very difficult. As a special case, one may ask if for a given prime $p$ there exist $p$-torsion groups with this property. We note that there are no such groups for $p=2$ -- this fact is well known to experts and follows, for instance,
from Shunkov's theorem~\cite{Shu} which asserts that if $\Gamma$ is a torsion group containing an involution
with finite centralizer, then $\Gamma$ is virtually solvable.
For completeness we will give a short self-contained proof:

\begin{Proposition} There is no residually finite globally zero-one group which is $2$-torsion.
\end{Proposition}
\begin{proof} We start with a general claim:
\begin{Claim} If $\Gamma$ is a $2$-group and $f$ is an automorphism of $\Gamma$ of order $2$,
then $\Gamma$ has a (non-trivial) fixed point.
\end{Claim}
\begin{proof} Take any $g\in \Gamma\setminus\{1\}$. If $f(g)=g$, we are done.
If $f(g)\neq g$, then $f(g)g^{-1}$ has even order $n$, and direct computation shows
that $f$ fixes $(f(g)g^{-1})^{n/2}$
\end{proof}
Now let $\Gamma$ be a residually finite globally zero-one group which is $2$-torsion. We claim
that for any $g\in\Gamma$ of order $2$, the centralizer $C_{\Gamma}(g)$ is infinite.
If not, then by residual finiteness of $\Gamma$ we can find a finite index normal subgroup
$\Lambda$ of $\Gamma$ such that $C_{\Gamma}(g)\cap\Lambda=\{1\}$. This is impossible by the above claim applied to the conjugation action of $g$ on $\Lambda$.
Thus, since $\Gamma$ is globally zero-one, $C_{\Gamma}(g)$ has finite index for any $g\in \Gamma$
of order $2$.

Now let $N$ be the subgroup of $\Gamma$ generated by all elements of order $2$.
Then $N$ is clearly infinite, so it is also of finite index (and finitely generated).
It follows that the centralizer of $N$ in $\Gamma$ is of finite index. Thus
$\Gamma$ is virtually abelian and hence finite, a contradiction.
\end{proof}

\subsection{Constructing hereditarily just-infinite quotients}

We start by formulating Grigorchuk's version of Wilson's classification theorem
for just-infinite groups (see \cite{Gr} and \cite{Wil3} for more details).

\begin{Theorem}
\label{justinfclass}
{\rm (a)} Let $\Gamma$ be a just-infinite abstract group.
Then one of the following holds:
\begin{itemize}
\item[(i)] $\Gamma$ is virtually simple
\item[(ii)] $\Gamma$ is hereditarily just-infinite
\item[(iii)] $\Gamma$ has a finite index subgroup isomorphic to a direct power $\Lambda^n$
where $n\geq 2$ and $\Lambda$ is simple or hereditarily just-infinite
\item[(iv)] $\Gamma$ is a branch group.
\end{itemize}
{\rm (b)} Let $G$ be a just-infinite profinite group.
Then one of the following holds:
\begin{itemize}
\item[(i)] $G$ is hereditarily just-infinite
\item[(ii)] $G$ has an open subgroup isomorphic to a direct power $L^n$
where $n\geq 2$ and $L$ is hereditarily just-infinite
\item[(iii)] $G$ is a branch group.
\end{itemize}
\end{Theorem}
\begin{Remark}
By definition any branch group also has a finite index subgroup which is
a (non-trivial) direct power of another group, and thus Theorem~\ref{justinfclass}
yields the classification statement from \S~1.5.
\end{Remark}
\vskip .2cm

Theorem~\ref{justinfclass} is a very minor refinement of \cite[Theorem~3]{Gr}.
The only difference is that \cite[Theorem~3]{Gr} leaves the possibility that
there exist just-infinite groups which are virtually hereditarily just-infinite,
but not hereditarily just-infinite. That this cannot happen was shown by
C. Reid~\cite[Lemma~5]{Re}.
\vskip .1cm

We proceed with the proofs of (generalized versions of) Corollary~\ref{HJI} and Theorem~\ref{HJI_prop}.

\begin{Observation}
\label{01LERF}
Let $\Gamma$ be a residually finite locally zero-one group. Then
$\Gamma$ is LERF.
\end{Observation}
\begin{proof} Finite index subgroups are open (hence closed) in the profinite topology by definition,
and finite subgroups are closed since $\Gamma$ is residually finite.
\end{proof}

\begin{Corollary}[generalization of Corollary~\ref{HJI}]
\label{HJI_virtual}
Let $\Gamma$ be an abstract group of positive virtual weighted $p$-deficiency.
Then $\Gamma$ has a hereditarily just-infinite quotient, which is virtually $p$-torsion.
\end{Corollary}

\begin{proof}
By definition, $\Gamma$ has a finite index normal subgroup $\Lambda$
such that if $G=\Gamma_{\widehat{p,\Lambda}}$ is the virtual pro-$p$ completion of $\Gamma$ relative to $\Lambda$, then $G$ has positive virtual weighted deficiency.
Now apply Theorem~\ref{zeroone_prop} to $G$ and $\overline{\Gamma}$
(where $\overline{\Gamma}$ is the image of $\Gamma$ in $G$),
and let $\Omega$ be the locally zero-one quotient of $\overline{\Gamma}$
constructed in the proof. Note that $\Omega$ is virtually $p$-torsion by construction and LERF by Observation~\ref{01LERF}.
In particular, $\Omega$ is virtually $p$-LERF.

Let $\Omega'$ be any just-infinite quotient of $\Omega$. Then $\Omega'$ is also
locally zero-one, and by Corollary~\ref{LERF1},  $\Omega'$ has infinite profinite completion.

By Theorem~\ref{justinfclass}(a) (and a remark after it) $\Omega'$ is virtually simple, hereditarily just-infinite,
or a finite extension of a direct power of some infinite group. Since $\Omega'$ is a locally zero-one group,
it clearly cannot have a finite index subgroup of the form
$A\times B$ with $A,B$ infinite, and since $\Omega'$ has infinite profinite completion,
it cannot be virtually simple. Hence $\Omega'$ is hereditarily just-infinite.
\end{proof}

In order to prove Theorem~\ref{HJI_prop} (the pro-$p$ version of Corollary~\ref{HJI}),
we need one more lemma.

Let $H$ be a virtually pro-$p$ group. We will say that $H$ satisfies condition (**) if
\begin{itemize}
\item[(**)] For every positive integer $k$ there exists an open pro-$p$ subgroup $V(k)$ of $H$
and a dense abstract subgroup $\Lambda(k)$ of $V(k)$ such all $k$-generated subgroups of
$\Lambda(k)$ are finite.
\end{itemize}

\begin{Lemma} \label{zeroone_prop2}
Let $G$ be a virtually pro-$p$ group of positive virtual weighted deficiency.
Then $G$ has a quotient $G'$ of positive virtual weighted deficiency
such that all quotients of $G'$ satisfy condition (**).
\end{Lemma}
\begin{proof}
Fix a dense abstract subgroup $\Gamma$ of $G$. Let $W$ be a finite virtual valuation of $G$ defined on an open normal pro-$p$ subgroup $U$
for which $def_W^G(U)>0$.

By the remark following Lemma~\ref{Golod} we can find a normal subgroup $N$ of $G$ such that $def^{G/N}_W(U/N)>0$
and for every finite set $Y\subset \Gamma\cap U$, with $W(Y)<1$, the image of
$\la Y\ra$ in $G/N$ is finite. Let $G'=G/N$.

Note that for each $k$ the group $U_{<1/k}=\{g\in U:\ W(g)<1/k\}$ is open in $G$,
and for each $Y\subset U_{<1/k}$, with $|Y|=k$, we have $W(Y)<1$.
Hence any quotient $H$ of $G'$ satisfies (**) where we let
$V(k)$ be the image of $U_{<1/k}$ in $H$ and
$\Lambda(k)$ the image of $\Gamma\cap U_{<1/k}$ in $H$.
\end{proof}

With the aid of Lemma~\ref{zeroone_prop2}, we can now prove Theorem~\ref{HJI_prop}.

\begin{Theorem} [generalization of Theorem~\ref{HJI_prop}]
\label{HJI_prop_virtual}
Let $G$ be a virtually pro-$p$ group of positive virtual weighted deficiency.
Then $G$ has a hereditarily just-infinite quotient.
\end{Theorem}

\begin{proof} As before, we can assume that $G$ is finitely generated.
Next, by Lemma~\ref{zeroone_prop2} we can assume that all quotients of $G$
satisfy condition (**).
By Theorem~\ref{zeroone_prop} there exists a quotient $H$ of $G$ which
contains a dense locally zero-one subgroup.
Replacing $H$ by its quotient (if necessary), we can assume that $H$ is just-infinite.
\vskip .12cm

Suppose that $H$ is not hereditarily just-infinite. Then by Theorem~\ref{justinfclass}(b)
there exists an open subgroup $M$ of $H$ which decomposes as $M=A\times B$ with $A$ and $B$ infinite.
Let $\pi_A:M\to A$ and $\pi_B:M\to B$ be the projection maps.

Let $k=d(A)$ be the minimal number of (topological) generators of $A$.
Since by construction $H$ satisfies (**), there exists an open pro-$p$ subgroup $V$ of $H$
and a dense abstract subgroup $\Lambda$ of $V$ such that
\begin{equation}
\label{FF}
\mbox{all $k$-generated subgroups of $\Lambda$ are finite.}
\end{equation}
By making $V$ smaller (if necessary) we can assume that
$V=C\times E$ where $C\subseteq A$ and $E\subseteq B$.
\vskip .1cm
Since $\Lambda$ is dense in $V$, its projection $\pi_B(\Lambda)$
is dense in $\pi_B(V)=E$. Hence
\vskip .15cm
\centerline{$E$ has no open subgroups which are (topologically)  generated by $k$ elements.}
\vskip .15cm
\noindent Indeed, if $E$ had an open $k$-generated subgroup $O$,
then since $O$ is a pro-$p$ group, we could find
a generating $k$-tuple for $O$ inside $\pi_B(\Lambda)$, which is impossible
by \eqref{FF}.

Now consider the group $Z=A\times E$. Recall that $H$ contains a dense locally zero-one
subgroup, call it $\Omega$. Since
$\Omega\cap Z$ is dense in $Z$ and $k=d(A)$, there exists
a $k$-generated subgroup $K\subseteq\Omega\cap Z$ such that
$\pi_A(K)$ is dense in $A$. In particular, $K$ must be infinite.
Since $\Omega$ is locally zero-one, we conclude that $K$ has finite index in $\Omega$,
so its closure $\overline{K}$ must be open in $H$.
But then $\overline{\pi_B(K)}\supseteq \pi_B(\overline K)$ is an open
subgroup of $\pi_B(Z)=E$. Hence $\overline{\pi_B(K)}$ is an open subgroup of $E$
topologically generated by $k$ elements.
This contradicts our earlier conclusion.
\end{proof}

As an immediate consequence of the proof of Theorem~\ref{HJI_prop_virtual},
we obtain a positive answer to Problem~20 from Chapter~I in \cite{dSSS}:

\begin{Corollary} There exist hereditarily just-infinite pro-$p$ groups
of infinite lower rank.
\end{Corollary}

Our last result in this section shows that there is a large class of groups
which cannot have hereditarily just-infinite quotients, except for
virtually cyclic ones. Combining this result with Theorems~\ref{HJI_virtual}
and~\ref{HJI_prop_virtual} we deduce that branch groups
cannot have positive virtual weighted deficiency.

\begin{Proposition}$\empty$
\label{branch_pwd}
\begin{itemize}
\item[(a)] Let $G$ be a virtually pro-$p$ (resp. abstract) group, and assume
that some open (resp. finite index) normal subgroup $H$ of $G$ is isomorphic to
$L_1\times\ldots\times L_k$ where $k\geq 2$, all $L_i$'s are isomorphic to each other,
and the conjugation action of $G$ on $H$ permutes the $L_i$'s transitively.
Then all hereditarily just-infinite quotients of $G$
are virtually procyclic (resp. virtually cyclic).
Therefore, by Theorem~\ref{HJI_prop_virtual} $G$
cannot have positive virtual weighted deficiency.

\item[(b)] Now let $G$ be a virtually pro-$p$ (resp. abstract) branch group. Then
open (resp. finite index) subgroups of $G$ do not have non-virtually procyclic
(resp. non-virtually cyclic) hereditarily just-infinite quotients.
Therefore, by Theorem~\ref{HJI_prop} open (resp. finite index) subgroups of $G$
cannot have positive weighted deficiency.
\end{itemize}
\end{Proposition}
\begin{proof} We will deal with the pro-$p$ case; the abstract case
is analogous.

(a) Let $Q$ be a hereditarily just-infinite quotient of $G$
and $\pi:G\to Q$ an epimorphism. Then $\pi(H)$ is just-infinite, so
for each $i$ the group $\pi(L_i)$ is either trivial or open in $Q$,
and since $L_i$'s are conjugate, $\pi(L_i)$ must be open
for all $i$. On the other hand, the subgroups $\pi(L_i)$, $1\leq i\leq k$,
commute with each other. Combining these two properties, we conclude
that $Q$ must be virtually abelian. Since $Q$ is just-infinite,
it must actually be virtually procyclic.
\vskip .1cm

(b) Suppose that $G$ has an open subgroup $H$ which has a non-virtually procyclic
hereditarily just-infinite quotient $Q$. Note that every open subgroup of $H$
also has such a quotient. Since $G$ is branch, by making $H$ smaller we can assume
that $H=L_1\times\ldots\times L_k$ where each $L_i$ is branch. Then at least
one of the subgroups $L_i$ surjects onto an open subgroup of $Q$, which
is also hereditarily just-infinite and non-virtually procyclic. This
is a contradiction since branch groups satisfy the hypothesis of part (a).
\end{proof}

The last result naturally brings up the following question:

\begin{Question} What can one say about groups which contain a finite index subgroup
of positive weighted deficiency?
\end{Question}

We just proved that such groups cannot be branch. A simple observation
is that such groups also cannot be just-infinite.
Indeed, Theorem~\ref{justinfclass} implies that if $G$ is a finite
index subgroup of a just-infinite group, then $G$ has finitely
many commensurability classes of normal subgroups, and it is easy to see
that a PWD group cannot have such property.

\section{Just infinite pro-$p$ groups with the associated graded
algebra of exponential growth}

\subsection{Discussion of the result}

In \cite{Sm} Smoktunowicz constructed examples of (graded) Golod-Shafarevich algebras
all of whose infinite-dimensional homomorphic images have exponential growth.
In this section we prove an analogous result for pro-$p$ groups
(Theorem~\ref{expgrowth} below).

\begin{Definition}\rm Let $G$ be a finitely generated pro-$p$ group,
and let $\omega$ be the augmentation ideal of $\Fp[[G]]$. By $gr \Fp[[G]]$
we denote the algebra $\oplus_{n=0}^{\infty}\omega^n/\omega^{n+1}$, the graded algebra
of $\Fp[[G]]$ with respect to powers of $\omega$.
\end{Definition}

\begin{Theorem}
\label{expgrowth}
Let $G$ be a pro-$p$ group of positive weighted deficiency.
Then $G$ has a quotient $G'$ which also has positive weighted deficiency and
such that for every infinite quotient $H$ of $G'$ the graded algebra $gr \Fp[[H]]$
has exponential growth. In particular, by Theorem~\ref{HJI_prop} there exists a
hereditarily just-infinite pro-$p$ group $H$ for which $gr \Fp[[H]]$ has exponential growth.
\end{Theorem}

\begin{Remark} The last assertion of Theorem~\ref{expgrowth} is the pro-$p$ analogue
of Theorem~\ref{justinfexp}. Theorem~\ref{justinfexp} itself will be established
at the end of the section as an easy consequence of Theorem~\ref{expgrowth} and Corollary~\ref{LERF1}.
\end{Remark}

In spite of the apparent similarity of Theorem~\ref{expgrowth} to the main result of \cite{Sm},
the techniques used in the two papers are completely different. In addition,
the construction from \cite{Sm} is only valid for algebras over fields of infinite transcendence degree,
and it would be interesting to see if our techniques could be used to extend the result of \cite{Sm}
to arbitrary fields.

To prove Theorem~\ref{expgrowth} we relate the growth of graded algebras associated
to pro-$p$ groups to the concept of $W$-index:

\begin{Lemma}
\label{windex_growth}
Let $G$ be a finitely generated pro-$p$ group, $W$ a finite valuation on $G$
and $N$ a normal subgroup of $G$ of infinite $W$-index. Then the algebra
$gr \Fp[[G/N]]$ has exponential growth.
\end{Lemma}

Lemma~\ref{windex_growth} will be proved at the end of \S~9.2. Theorem~\ref{expgrowth}
is a consequence of this lemma and the following theorem which is the main result of this section:

\begin{Theorem} \label{nofiniteindex} Let $G$ be a pro-$p$ group and $W$ a finite valuation on $G$
such that $def_W(G)>0$. Then $G$ has a quotient $G'$ such that $def_W(G')>0$ and each closed
normal subgroup of $G'$ of infinite index is also of infinite $W$-index.
\end{Theorem}

There is an obvious ``naive'' approach to Theorem~\ref{nofiniteindex} --
one needs to make sure that all closed normal subgroups of $G$ of finite $W$-index
map to finite index subgroups of $G'$. According to Case~1 of Theorem~\ref{thm:key}
the latter can be achieved for any single closed normal subgroup of $G$ and more generally
for any countable collection of such subgroups, but this is far from sufficient.
Instead, we address the analogous problem on the level of Lie algebras.

One can define $W$-index for subalgebras of the Lie algebra $L_W(G)$
in complete analogy to the group case; then a subgroup $H$ has finite $W$-index in $G$
if and only if $L_W(H)$ has finite $W$-index in $L_W(G)$.
The key Lemma~\ref{lafingen} below implies that every graded
ideal\footnote{Recall that by an ideal of a restricted Lie algebra we mean a restricted ideal} of $L_W(G)$
of finite $W$-index contains a finitely generated graded ideal of finite $W$-index,
and clearly there are countably many of those, call them $M(1), M(2),\ldots$.
Arguing similarly to Case~1 of Theorem~\ref{thm:key}, we construct a quotient $G'$ of $G$ such that
$def_W(G')>0$ and each of the ideals $M(i)$ maps onto a finite index ideal under the induced map
$L_W(G)\to L_W(G')$. It follows easily that the group $G'$ has the required property.

\subsection{Growth in restricted Lie algebras}

In this subsection we prove Lemma~\ref{windex_growth} and two other lemmas
needed for the proof of Theorem~\ref{nofiniteindex}.

\begin{Lemma}
\label{envelope2}
Let $\Omega$ be a subsemigroup of $((0,1),\cdot)$
s.t. $\Omega\cap (\alpha,1)$ is finite for all $\alpha>0$. Let
$L=\oplus_{\alpha\in \Omega} L_{\alpha}$ be a graded restricted Lie $\Fp$-algebra
with finite-dimensional homogeneous components, and let
$A=\Fp\oplus (\oplus_{\alpha\in \Omega} A_{\alpha})$ be the universal enveloping algebra of $L$
with the associated grading by $\Omega\cup\{1\}$.
Let $c_{\alpha}=\dim L_{\alpha}$ and $a_{\alpha}=\dim A_{\alpha}$
for $\alpha\in\Omega\cup\{1\}$. The following hold:
\begin{itemize}
\item[(i)] $$\sum_{\alpha\in\Omega} a_{\alpha} \alpha^s=\prod_{\alpha\in\Omega}\left(\frac{1-\alpha^{ps}}{1-\alpha^s}\right)^{c_{\alpha}}.$$
where both sides are considered as Dirichlet series in a formal variable $s$.
\item[(ii)] For any $\alpha\in \Omega$ we have $c_{\alpha}\leq a_{\alpha}$
\item[(iii)] The numerical series $\sum_{\alpha\in\Omega} \alpha c_{\alpha}$ converges
if and only if $\sum_{\alpha\in\Omega} \alpha a_{\alpha}$ converges.
\end{itemize}
\end{Lemma}
\begin{proof} (i) is obtained from the Poincare-Birkhoff-Witt theorem for graded restricted
Lie algebras by dimension counting, and (ii) and (iii) are direct consequences of (i).
\end{proof}

\begin{Lemma}
\label{lafingen}
Let $L=\displaystyle \bigoplus_{\alpha\in \Omega} L_\alpha$ be a finitely generated $\Omega$-graded restricted
Lie $\Fp$-algebra, where $\Omega$ is a finitely generated subsemigroup of $((0,1),\cdot)$. Let $D=\displaystyle \bigoplus_{\alpha\in \Omega} D_\alpha$ be a graded ideal such that
$$\sum _{\alpha \in \Omega}\alpha (\dim L_\alpha-\dim D_\alpha)<\infty.$$ Then $D$ contains a  graded ideal $M=\displaystyle \bigoplus_{\alpha\in \Omega} M_\alpha$ of $L$ which is finitely generated as an ideal and
satisfies the inequality $$\sum _{\alpha \in \Omega}\alpha (\dim L_\alpha-\dim M_\alpha)<\infty.$$
\end{Lemma}
\begin{proof}
Let $A$ be the universal enveloping algebra of $L$ with its natural grading by
$\Omega\cup \{1\}$: $$A=\F_p\oplus(\displaystyle \bigoplus_{\alpha\in \Omega} A_\alpha).$$
Let $I$ be the ideal of $A$ generated by $D$; it is then a graded ideal:
$I=\displaystyle \bigoplus_{\alpha\in \Omega}I_\alpha$.

Note that $A/I$ is the universal enveloping algebra for $L/D$. Thus, by Lemma~\ref{envelope2}
the condition $$\sum _{\alpha \in \Omega}\alpha (\dim L_\alpha-\dim D_\alpha)<\infty$$ is equivalent to the condition $$\sum _{\alpha \in \Omega}\alpha (\dim A_\alpha-\dim I_\alpha)<\infty.$$
In particular, there exists $\beta>0$ such that $$\sum_{\alpha\le \beta} \alpha(\dim A_\alpha-\dim I_\alpha)<1.\eqno (***)$$
Since $L$ is finitely generated, $A$ is also finitely generated and so $$B=\F_p\oplus (\displaystyle \bigoplus_{\beta\ge \alpha\in \Omega} A_\alpha)$$ is finitely generated.
Choose $\gamma>0$ such that $B$ is generated by graded elements of degree greater than $\gamma$.
Let $M\subseteq D$ be the ideal of $L$ generated
by $\{D_\alpha:\ \alpha> \gamma\}$ and $J=\oplus_{\alpha}J_{\alpha}\subseteq I$ the ideal of $A$ generated by $M$.
Put $\bar A=A/J=\F_p+\oplus(\displaystyle \bigoplus_{\alpha\in \Omega} \bar A_\alpha)$ and $\bar B=\F_p+\oplus(\displaystyle \bigoplus_{\beta \ge \alpha\in \Omega} \bar A_\alpha)$.
We shall show that $$\sum_{\alpha\in \Omega}\alpha \dim\bar A_\alpha=
\sum_{\alpha\in \Omega}\alpha (\dim A_\alpha-\dim J_{\alpha})<\infty.$$
Since $A/J$ is the universal enveloping algebra for $L/M$, this will finish the proof
by Lemma~\ref{envelope2}(iii).
\vskip .1cm

Let $\{\gamma_1,\ldots,\gamma_k\}=\{\alpha\in \Omega: \beta\ge \alpha>\gamma\}$.
Since  $\bar B$ is generated by $\{\bar A_{\gamma_1},\ldots, \bar A_{\gamma_k}\}$, we obtain that $\dim \bar A_\alpha\le f_\alpha$ for all $\alpha\leq \beta$, where $f_\alpha$ are the coefficients of the
Dirichlet series
\begin{equation}
\label{Dirichlet}
\sum_\alpha f_\alpha \alpha^s=\frac{1}{1-\sum_{i=1}^k \dim \bar A_{\gamma_i}\gamma_i^s}.
\end{equation}
Since $\gamma_i>\gamma$, we have $\dim \bar A_{\gamma_i}=\dim A_{\gamma_i}-\dim J_{\gamma_i}=
\dim A_{\gamma_i}-\dim I_{\gamma_i}$. On the other hand $\gamma_i\le \beta $, so (***) implies that
$$\sum_{i=1}^k \gamma_i \dim \bar A_{\gamma_i}< 1.$$
Hence \eqref{Dirichlet} holds as a numerical equality for $s=1$, and we get
$$\sum_{\alpha\in \Omega}\alpha \dim\bar A_\alpha\leq \sum_{\beta< \alpha\in \Omega}\alpha \dim\bar A_\alpha+\sum_{\beta\ge \alpha\in \Omega}\alpha f_{\alpha}= \sum_{\beta< \alpha\in \Omega}\alpha \dim\bar A_\alpha+\frac 1{1-\sum_{i=1}^k \gamma_i \dim \bar A_{\gamma_i}}<\infty.$$
\end{proof}

We finish this subsection with the proof of Lemma~\ref{windex_growth}.

\begin{proof}[Proof of Lemma~\ref{windex_growth}]
First of all note that $[G:N]_W=[G/N:\{1\}]_W$ by Lemma~\ref{index_stupid}, 
so without loss of generality we can assume that $N=\{1\}$.

Consider the Lie algebra $L_W(G)=\bigoplus\limits_{\alpha\in \Im W} L_{\alpha}$,
and let $c_{\alpha}=\dim L_{\alpha}=\log_p[G_{\alpha}:G_{<\alpha}]$ for $\alpha\in \Im W$.
By definition of $W$-index we have $[G:\{1\}]_W=\prod\limits_{\alpha\in \Im W} \left(\frac{1-\alpha^p}{1-\alpha}\right)^{c_{\alpha}}$. Hence
$\sum_{\alpha\in \Im W}\alpha c_{\alpha}=\infty$.

Let $q=\max(\Im W)\in (0,1)$ be the maximal value of $W$ on $G$, and for $k\in \dbN$ let $d_k=\sum\limits_{q^{k+1}< \alpha\leq q^{k}} c_{\alpha}$.
Thus $\sum_{k=0}^{\infty} d_k q^k=\infty$, whence
$$\limsup_{k\to \infty}\sqrt[k]{d_k}>1\eqno (***)$$

Now let $\omega$ be the augmentation ideal of $\Fp[[G]]$, and let $\{D_k G\}_{k\in\dbN}$
be the Zassenhaus filtration of $G$. Recall that $D_k G=\{g\in G : g-1\in \omega^k\}$.
It is clear that $D_k G\subseteq G_{\leq q^k}$, whence
$\log_p[G:D_k G]\geq \sum_{n<k} d_n$. On the other hand, we have an embedding
$D_k G/ D_{k+1}G\to \omega^k/\omega^{k+1}$. Thus if we put $b_k=\dim \omega^k/\omega^{k+1}$,
then $\sum_{n<k} b_n\geq \log_p[G:D_k G]\geq\sum_{n<k} d_n$, so $\limsup_{k\to \infty}\sqrt[k]{b_k}>1$ by (***).

Finally, since the sequence $\{b_k\}$ is submultiplicative ($b_{k+l}\leq b_k b_l$),
the limit $\lim_{k\to \infty}\sqrt[k]{b_k}$ must exist. Therefore, $\lim_{k\to \infty}\sqrt[k]{b_k}>1$,
so the algebra $gr \Fp[[G]]=\oplus_{k=0}^{\infty}\omega^k/\omega^{k+1}$ is of exponential growth.
\end{proof}

\subsection{Proofs of Theorem~\ref{nofiniteindex} and Theorem~\ref{justinfexp}}

\begin{proof}[Proof of Theorem \ref{nofiniteindex}]

As usual, without loss of generality we can assume that the Lie algebra $L_W(G)$
is finitely generated.

Let $\mathcal I$ be the set of graded ideals $M=\oplus_{\alpha\in \im W} M_\alpha$ of infinite index such that $M$ is finitely generated as an ideal and
$$\sum_{\alpha\in \im W} \alpha(\dim L_\alpha-\dim M_\alpha)<\infty.$$
It is clear that $\mathcal I$ is countable, so we can enumerate its elements:
$\,\mathcal I=\{M(1),M(2),\ldots\}$.

For each $i$ we can find $\alpha_i>0$ such that
$$\sum_{\alpha_i>\alpha\in \im W} \alpha(\dim L_\alpha-\dim M(i)_\alpha)<\frac{def_W(G)}{2^{i}}.$$
Let $ S_i$ be a subset of $G$ such that    $ \{\LT(g)+M(i)_{W(g)}:\  g\in S_i\}$ is a basis of $\displaystyle \bigoplus_{\alpha_i>\alpha\in \im W} L_\alpha/M(i)_\alpha$. Note that
 $W(S_i)= \sum_{\alpha_i>\alpha\in \im W} \alpha(\dim L_\alpha-\dim M(i)_\alpha)<\frac{def_W(G)}{2^{i}}.$

Let $S=\cup_{i\ge 1} S_i$ and let $N$ be the normal subgroup of $G$ generated by $S$. We have that $W(S)<def_G(W)$ and so by Lemma \ref{addrel}, $def_W(G/N)>0$.
\vskip .12cm
Now let us show that all closed normal subgroups of $G/N$ of finite $W$-index are open.
By Lemma~\ref{lafingen} and the choice of $S$ for every  graded ideal  $P=\oplus_\alpha P_\alpha $  of $L_W(G/N)=\oplus_\alpha \bar L_\alpha$, either
\begin{itemize}
\item[(i)] $P$ is of finite index or
\item[(ii)] $\sum_\alpha \alpha (\dim \bar L_\alpha-\dim P_\alpha)=\infty.$
\end{itemize}
Let  $K$ be a closed normal subgroup of $G/N$ of finite $W$-index and
put  $P=L_W(K)$. Then $P$ is an ideal of $L_W(G/N)$.
Since the $W$-index of $K$ is finite, we have $$ \sum_\alpha \alpha
(\dim \bar L_\alpha-\dim P_\alpha)=\sum_\alpha \alpha (c_\alpha(G/N)-c_\alpha(K))<\infty.$$
(where integers $c_{\alpha}(\cdot)$ are defined as in Section~3).
Thus by the above dichotomy $P$ is of finite index in $L_W(G/N)$, so
$K$ must be open in $G/N$.
\end{proof}

\begin{proof}[Proof of Theorem~\ref{justinfexp}] By Theorem~\ref{LERF2} there exists
a PWD group $\Gamma$ which is $p$-LERF. Let $G=\Gamma_{\phat}$,
and let $G'$ be a group satisfying the conclusion of Theorem~\ref{expgrowth}
applied to $G$. By Theorem~\ref{dense} we can replace $G'$ by another quotient
$G''$ such that $G''$ has PWD and $G''=(\Gamma'')_{\phat}$ where $\Gamma''$
is the image of $\Gamma$ in $G''$.

Let $\Delta$ be any infinite quotient of $\Gamma''$. Since $\Gamma$ is $p$-LERF
and $\Delta$ is also a quotient of $\Gamma$,  by Corollary~\ref{LERF1} the pro-$p$ completion
$\Delta_{\phat}$ is infinite.
The natural projection $\Gamma''\to \Delta$ induces the corresponding epimorphism
$G''=(\Gamma'')_{\phat}\to \Delta_{\phat}$, and so by the choice of $G'$
and by Lemma~\ref{windex_growth} the graded algebra $gr \Fp[[\Delta_{\phat}]]$ has
exponential growth.
It is easy to show that $gr \Fp[\Delta]$ is naturally
isomorphic to $gr \Fp[[\Delta_{\phat}]]$ as graded algebras, and thus
$gr \Fp[\Delta]$ also has exponential growth. Finally, by Theorem~\ref{HJI}
$\Delta$ can be chosen hereditarily just-infinite.
\end{proof}

\end{document}